\documentclass[a4paper,11pt,,reqno]{amsart}
\usepackage{caption,mathrsfs}
\DeclareCaptionLabelFormat{adja-page}{\hrulefill\\#1 #2 \emph{(previous page)}}
\usepackage[left=2cm,right=2cm,top=2.25cm,bottom=2.25cm]{geometry}
\usepackage[utf8]{inputenc}
\usepackage[T1]{fontenc}
\usepackage{amssymb}
\usepackage[english]{babel}
\usepackage{amsmath}
\usepackage{amsthm}
\usepackage{epsfig}
\usepackage{mathrsfs}
\usepackage{xcolor}
\usepackage{graphicx}
\usepackage{esint}
\usepackage{enumerate}
\usepackage{cases}
\usepackage[toc,page]{appendix}
\usepackage{hyperref}
\usepackage{xcolor}
\usepackage{verbatim,cleveref}

\newcommand{\dd}{{\, \mathrm d}}
\usepackage{todonotes}\reversemarginpar

\newtheorem{theorem}{Theorem}[section]
\newtheorem{lemma}[theorem]{Lemma}

\newtheorem{remark}[theorem]{Remark}
\newtheorem{proposition}[theorem]{Proposition}
\newtheorem{corollary}[theorem]{Corollary}

\newcommand{\R}{{\mathbb R}}
\newcommand{\N}{{\mathbb N}}
\newcommand{\T}{{\mathbb T}}

\def\E{\mathcal E}

\definecolor{darkgreen}{rgb}{.0,.4,.2}

\def\eps{\varepsilon}
\def\na{\nabla}
\def\pa{\partial}
\def\div{\mathrm{div}\, }

\newcommand{\be}[1]{\begin{equation}\label{#1}}
\newcommand{\ee}{\end{equation}}
\renewcommand{\(}{\left(}
\renewcommand{\)}{\right)}
\newcommand{\irdx}[1]{\int_{\T^d}{#1}\,dx}

\newcommand{\irdv}[1]{\int_{\R^d}{#1}\,\dd v}
\newcommand{\irdxv}[1]{\iint_{\T^d\times\R^d}{#1}\,\dd x \,\dd v}
\newcommand{\iroxv}[1]{\iint_{\Omega\times\R^d}{#1}\,\dd x \,\dd v}

\newcommand{\var}{\varepsilon}

\def\div{{\mathrm {div}}\,}

\newcommand{\g}{{\bf g}}

\usepackage{array,multirow,makecell,booktabs}
\setcellgapes{1pt}
\makegapedcells
\newcolumntype{R}[1]{>{\raggedleft\arraybackslash }b{#1}}
\newcolumntype{L}[1]{>{\raggedright\arraybackslash }b{#1}}
\newcolumntype{C}[1]{>{\centering\arraybackslash }b{#1}}

\usepackage{etoolbox}
\newcounter{taggedeq}
\pretocmd{\equation}{\stepcounter{taggedeq}}{}{}

\title[Nonlinear kinetic Fokker-Planck equations]{Nonlinear kinetic Fokker-Planck equations: \\ existence and diffusion limits}
\author{\'Emeric Bouin}
\address[E. Bouin]{CEREMADE - Université Paris-Dauphine, PSL Research University, UMR CNRS 7534, Place du Mar\'echal de Lattre de Tassigny, 75775 Paris Cedex 16, France.}
\email{bouin@ceremade.dauphine.fr}

\author{Jean Dolbeault}
\address[J. Dolbeault]{CEREMADE - Université Paris-Dauphine, PSL Research University, UMR CNRS 7534, Place du Mar\'echal de Lattre de Tassigny, 75775 Paris Cedex 16, France.}
\email{dolbeault@ceremade.dauphine.fr}

\author{Antoine Mellet}
\address[A. Mellet]{Department of Mathematics
University of Maryland
College Park, MD 20742}
\email{mellet@umd.edu}

\date{June 30, 2026}


\begin{document}

\begin{abstract}
In this paper, we focus on a new type of non-linear kinetic Fokker-Planck equation where the non-linearity comes from a non-linear diffusion in the velocity variable. The existence of solutions in suitable Lebesgue spaces is proved, together with important entropy estimates on these solutions. We then study the diffusive limit of such equation.
\end{abstract}

\maketitle
\thispagestyle{empty}

\section{Introduction and main results}

\subsection{A nonlinear kinetic Fokker--Planck equation}

Kinetic equations provide a fundamental framework for describing the evolution of particle systems at an intermediate scale between microscopic stochastic dynamics and macroscopic continuum models. Among them, kinetic Fokker--Planck equations are used to describe various physical phenomena in statistical mechanics, plasma physics and molecular dynamics. They typically describe the evolution of a distribution function $f(t,x,v)$ in the phase space, where $x\in\mathbb{R}^d$ denotes the position variable and $v\in\mathbb{R}^d$ the velocity. A prototypical kinetic Fokker--Planck equation takes the form
\[
\partial_t f + v\cdot\nabla_x f
= \nabla_v\cdot\left(\nabla_v f + v f\right),
\]
which corresponds to the Kolmogorov forward equation associated with a Langevin process involving friction and stochastic forcing. Such models play an important role in the mathematical description of interacting particle systems subject to noise and relaxation mechanisms.

This paper is devoted to the analysis of the non-linear Vlasov-Fokker-Planck equation
\begin{equation}\label{eq:mainVFP}
\frac{\partial f}{\partial t}+\,v\cdot\nabla_x f=\Delta_v f^m+\nabla_v\cdot (v\,f), \qquad t > 0, \; x \in \Omega, \; v \in \R^d.
\end{equation}
When $m=1$, we recover the classical equation written above, so we are interested in the case $m \in (0,1) \cup (1,+\infty)$, though further restrictions on $m$ shall arise later on.

One of the central problems in kinetic theory is the derivation of macroscopic equations from kinetic models through suitable asymptotic limits. These limits provide a rigorous link between microscopic particle dynamics and continuum descriptions such as diffusion or hydrodynamic equations. Among the most classical regimes is the \emph{diffusion limit}, in which the kinetic equation is studied under a scaling where collisions or relaxation effects dominate the dynamics on fast time scales. In this regime, the distribution function rapidly relaxes toward a local equilibrium in velocity, while the macroscopic density evolves on a slower time scale according to a diffusion equation.

The mathematical analysis of diffusion limits for kinetic equations has a long history and has been extensively studied for various models, including the linear Boltzmann equation, BGK-type models, and kinetic Fokker--Planck equations. Early rigorous derivations of diffusion limits for linear transport equations can be traced back to the works of Larsen and Keller \cite{LarsenKeller1974} and to the Hilbert expansion method developed in the context of neutron transport theory. Later developments include the systematic study of diffusive limits for kinetic equations with relaxation operators, as presented for instance in works of Bardos, Golse, and Levermore \cite{BardosGolseLevermore1991,BardosGolseLevermore1991b} and in the book of Cercignani, Illner, and Pulvirenti \cite{CercignaniIllnerPulvirenti1994}.

In the case of the linear Boltzmann equation, diffusion limits have been analyzed using a variety of techniques including asymptotic expansions, compactness arguments, and entropy methods. Notable contributions include the works of Golse, Lions, Perthame, and Sentis \cite{GolseLionsPerthameSentis1988}, who developed a general framework for deriving macroscopic limits from kinetic models, and later extensions that addressed more general collision operators and boundary conditions. These methods have also been adapted to treat nonlinear kinetic equations and systems with more complex interaction mechanisms.

For kinetic Fokker--Planck equations, the derivation of diffusion limits presents both similarities and specific challenges compared to the Boltzmann setting. The collision operator acts as a diffusion operator in velocity space and generates a relaxation toward a Maxwellian equilibrium. However, the transport operator couples spatial and velocity variables, leading to a degenerate structure in which dissipation occurs only in part of the phase space variables. This feature makes the analysis closely related to the theory of hypoelliptic and hypocoercive operators.

Nonlinear diffusion equations arise in a wide range of physical, biological, and engineering contexts where transport processes depend on the local density of the evolving quantity. A prototypical example is the nonlinear diffusion equation
\[
\partial_t u = \Delta u^k,
\]
where $u=u(t,x)$ denotes a nonnegative density and $m>0$ is a parameter describing the nonlinear dependence of the diffusion flux on the density. This equation, commonly referred to as the \emph{porous medium or fast diffusion equation}, provides a fundamental model for density-dependent diffusion phenomena.

When $k>1$, the equation is known as the \emph{porous medium equation} (PME). In this regime the diffusion becomes degenerate as $u$ approaches $0$, leading to the formation of free boundaries and finite propagation speed. This behavior contrasts sharply with classical linear diffusion, where disturbances propagate instantaneously. The porous medium equation was originally introduced in the context of gas flow through porous media and filtration processes; see the early works of Muskat and Wyckoff as well as the systematic mathematical analysis developed later. Since then, the PME has become a paradigmatic nonlinear partial differential equation describing degenerate diffusion phenomena.

When $0<k<1$, the equation is called the \emph{fast diffusion equation} (FDE). In this regime the diffusion coefficient becomes singular as $u\to 0$, leading to very different qualitative properties. In particular, solutions may exhibit instantaneous propagation, strong smoothing effects, and, depending on the value of $m$, even finite-time extinction. The fast diffusion equation appears in various applications including plasma physics, thin film dynamics, population dynamics, and geometric flows.

The mathematical theory of nonlinear diffusion equations has been extensively developed over the past decades. Fundamental analytical results concerning existence, uniqueness, regularity, and asymptotic behavior of solutions can be found in the classical works of Aronson, Bénilan, Crandall, and Pierre. The systematic treatment of the porous medium equation was developed in particular by Vázquez and collaborators, culminating in the comprehensive monograph \cite{Vazquez2007}. Another important reference is the book by DiBenedetto \cite{DiBenedetto1993}, which provides a detailed analysis of degenerate parabolic equations and their regularity properties.

One of the most remarkable features of nonlinear diffusion equations is the presence of self-similar solutions governing the large-time behavior of solutions. In the porous medium regime, the long-time asymptotics of solutions are described by the celebrated Barenblatt–Pattle self-similar profiles. These solutions play a role analogous to the Gaussian kernel in linear diffusion and provide the fundamental attractors of the dynamics. The study of convergence toward these self-similar solutions has been carried out using a variety of techniques including entropy methods, scaling arguments, and functional inequalities.

In the fast diffusion regime, the asymptotic behavior is more subtle and strongly depends on the value of the exponent $m$. Several critical exponents appear that determine different qualitative behaviors of solutions. For instance, below certain thresholds solutions may vanish in finite time, while above them the dynamics resembles that of the porous medium equation but with faster spreading. The mathematical analysis of these phenomena has been developed in a number of influential works, including those of Bénilan, Brezis, and Crandall, and later extensions addressing extinction phenomena and asymptotic profiles.

In the context of kinetic equations, relative entropy methods have been extensively applied to prove rigorous diffusion or hydrodynamic limits. The idea is to measure the ``distance'' between the solution of a kinetic equation and a local equilibrium parameterized by macroscopic variables. By controlling the evolution of the relative entropy and its production rate, one can obtain uniform-in-$\varepsilon$ estimates that justify the convergence to macroscopic equations such as diffusion, Euler, or Navier–Stokes equations in suitable scaling limits \cite{BardosGolseLevermore1991b,GolseSaintRaymond2004}.

Relative entropy techniques are also widely used in the analysis of nonlinear diffusion equations. For instance, in the study of the porous medium and fast diffusion equations, the relative entropy with respect to self-similar Barenblatt profiles provides a natural functional to quantify convergence rates and asymptotic behavior \cite{Vazquez2007,Otto2001}. This approach allows one to derive functional inequalities and decay estimates that are often sharper than those obtained from purely energy-based methods.

Beyond deterministic PDEs, relative entropy has found significant applications in probability theory, interacting particle systems, and stochastic processes. It provides a natural framework to study propagation of chaos, large deviations, and the mean-field limit of many-body systems. In these contexts, the entropy functional quantifies the deviation of the empirical measure from its limiting law and allows one to control fluctuations in a probabilistically meaningful way \cite{MischlerMouhot2013}.

The present work aims to establish existence and uniqueness of solution for \eqref{eq:mainVFP} and rigorously establish macroscopic limits.

\subsection{The nonlinear Fokker-Planck operator}

The operator
\begin{align*}
\mathsf{L}[f](v) & : = \Delta_vf^m+\nabla_v(v\,f) = \text{div}_v \left( f \, \na_v \left( \frac{m}{m-1} f^{m-1} +\frac{|v|^2}{2}\right)\right)
\end{align*}
is associated to the family of local equilibria,
\[
v\mapsto \left(\mu-\frac{m-1}{2\,m}\,|v|^2\right)_+^\frac1{m-1}\,,
\]
parametrised by a constant $\mu\in \R$. Observe that for any $\mu\in \R$, the mass of an equilibrium is finite if and only if
$$m>\frac{d-2}{d}.$$
Under this condition, we can parametrise these equilibria by their mass: we define $G[\rho]$ by,
\begin{equation*}\label{G}
G[\rho](v) = \left(\mu[\rho]- \frac{m-1}{2\,m}\,|v|^2\right)_+^\frac1{m-1}, \qquad \irdv{G[\rho](v)} = \rho.
\end{equation*}
This defines the function $\mu:\R_+\to\R_+$ and a quick computation gives, for all $m>\frac{d-2}{d}$ (and $m\neq 1$):
\begin{equation}\label{eq:mu}
\mu[ \rho] = \mu_1\,\rho^{k-1}, \qquad
\mu_1:=
\left( 2\left(\frac{2\,\pi\,m}{1-m}\right)^\frac d2\frac{\Gamma\(\frac{1}{1-k} \)}{\Gamma\(\frac1{m-1}\)}
\right)^{\frac{1}{1-k}},
\end{equation}
where $m$ and $k$ are related by the relation
\begin{equation}\label{eq:km}
\frac{1}{k-1}= \frac{1}{m-1} + \frac d 2.
\end{equation}
The function $v\mapsto G[\rho](v)$ is the classical Barenblatt-Pattle equilibrium.
In particular, it is compactly supported when $m>1$ and heavy-tailed when $m<1$.

\medskip Next, we recall that the non-linear Fokker-Planck operator is associated with the entropy functional
$$
H[f] = \int_{\R^d}\left( \frac{f(v)^m}{m-1} + \frac{|v|^2 }{2}f(v) \right) \dd v.
$$
This functional is convex for all values of $m>0$ (with the usual convention that $\frac{f^m}{m-1}$ is replaced by $f \ln f$ when $m=1$), although it may take negative values when $m<1$ and the condition that $H[G[\rho]]$ is finite requires the additional restriction on $m$:
$$ m>\frac{d}{d+2}$$
in order to ensure the integrability of $G[\rho]^m$ and $|v|^2 G[\rho]$.
In fact, we have the following classical result, relying on the Jensen inequality,
\begin{proposition}\label{prop:entro}
For all $m>\frac{d}{d+2}$
and for any $\rho\in \R$, we have
$$
H[f]\geq H[G[\rho]] \, , \qquad \mbox{ for all } f(v) \geq 0 \mbox{ such that } \irdv{f(v)} =\rho
$$
with equality if and only if $f(v)=G[\rho](v)$.
\end{proposition}

When $m>\frac{d}{d+2}$, we define:
\begin{equation}\label{eq:nu}
\nu[\rho] : =\frac 1 d \irdv{|v|^2\,G[\rho](v)} = \irdv{G[\rho](v)^m}= \nu_1 \rho^k
\end{equation}
with $k$ given by \eqref{eq:km} and where $\nu_1$ and $\mu_1$ (defined in \eqref{eq:mu}) satisfy
\begin{equation}\label{eq:nu1}
\frac{k}{1-k} \nu_1 = \frac{m}{1-m} \mu_1 .
\end{equation}
We then have
$$H[G[\rho]] = \frac{\nu_1}{k-1}\rho^k.$$

\medskip The relation between $k$ and $m$, \eqref{eq:km}, plays an important role in the paper. We note that when $m>1$, we find $k\in(1,\frac{d+2}{d})$ (in dimension $3$, it yields $k\in (1,\frac 5 3)$).

When $m<1$, there are several critical values that play a role in the analysis of the nonlinear Fokker-Planck equation.
We already mentioned the condition $m>\frac{d-2}{d}$ required for the integrability of $G[\rho]$ and the more restrictive $m>\frac{d}{d+2}$ to ensure finite entropy.
The condition $m\in\left(\frac{d}{d+2},1\right)$ is equivalent to $k\in (0,1)$.

The intermediate regime $m \in \left( \frac{d-2}{d} , \frac{d}{d+2}\right]$, for which $G[\rho]$ has no variance, will not be discussed in terms of diffusion limits: a fractional regime may appear, and is out of the scope of the present paper. Note that a moment of order $a$ of $G$ is finite when $m> \frac{d+a-2}{d+a}$.

\medskip Returning to the full kinetic equation \eqref{eq:mainVFP}, we define the entropy functional
\begin{equation*}
\mathcal E[f] = \iroxv{\left(\frac{f^m}{m-1} +\frac{|v|^2}{2} f \right)}.
\end{equation*}
A simple (formal) computation, assuming enough regularity to justify all integrations by parts, shows that solutions of \eqref{eq:mainVFP} satisfy the following:
\begin{equation} \label{eq:entropy_kinetic}
\frac{d}{dt} \mathscr E(f)(t) +
\iint_{\Omega \times\R^d} f \left\vert v + \frac{m}{m-1} \na_v f^{m-1} \right\vert^2 \, \dd x \, \dd v = 0.
\end{equation}
This entropy inequality plays a central role in the derivation of the diffusion limit.

\subsection{Existence of weak solutions}

Our first task is to show the existence and uniqueness of a global-in-time solution to \eqref{eq:mainVFP}. We state the results separately for the case $m>1$ and $m<1$ has some of the statements differ slightly. In both cases, we address both the case of a bounded domain with periodic boundary conditions $\Omega = \T^d$ and the whole set $\Omega=\R^d$. As part of our result, we show that these solutions satisfy the entropy inequality that will be needed to carry out the diffusion limit.

\begin{theorem}[Well-posedness when $m>1$]\label{thm:existence}
Fix $m>1$ and let $\Omega =\T^d$ or $\R^d$.
For any non-negative initial condition $f_{\text{in}}\in L^1\cap L^\infty (\Omega\times\R^d)$ such that
\begin{equation}\label{eq:xvfin} \iint_{\Omega\times \R^d} (|x|^2 + |v|^2) f_{\text{in}} (x,v)\, \dd x \, \dd v <\infty,
\end{equation}
there exists a unique non-negative weak solution $f(t,x,v)$ to
\begin{equation}\label{eq:FP}
\begin{cases}
\pa_t f + v\cdot\na_x f = \Delta_v f^m +\text{div}_v (vf) & \mbox{ in } (0,\infty)\times \Omega\times \R^d\\
f(0,x,v) = f_{\text{in}}(x,v) & \mbox{ in } \Omega\times \R^d.
\end{cases}
\end{equation}
For any $T>0$, we have $f\in L^\infty((0,T);L^1\cap L^\infty(\Omega\times \R^d))$, $\na_v f^{\frac{q}{2}} \in L^2((0,T)\times \Omega\times\R^d)$ for all $q>m$ and
$f$ satisfies the entropy inequality
\begin{equation}\label{eq:entropy}
\mathcal E[f(t)] + \int_0^t \iint_{\Omega\times \R^d} f \left| v + \frac{m}{m-1}\na_v f^{m-1}\right|^2 \, \dd x \, \dd v \leq \mathcal E[f_{\text{in}}].
\end{equation}

Furthermore, if there is $M_0$ such that $f_{\text{in}}(x,v) \leq \left( M_0-\frac{m-1}{m}\frac{|v|^2}{2}\right)_+^{\frac{1}{m-1}}$ in $\Omega\times\R^d$, then
\begin{equation}\label{eq:maxbar} f(t,x,v) \leq \left( M_0-\frac{m-1}{m}\frac{|v|^2}{2}\right)_+^{\frac{1}{m-1}} \qquad \mbox{ in } (0,\infty)\times \Omega\times\R^d.
\end{equation}
\end{theorem}

In the fast diffusion case $m<1$, the heavy tail of the equilibrium state suggests that some additional difficulties arise when trying to control large velocity (and large $|x|$ when $\Omega =\R^d$).

\begin{theorem}[Well-posedness when $m<1$]\label{thm:existencefast}
Fix $m \in \left( \frac{d-2}{d} ,1 \right)$ if $\Omega =\T^d$ or $m \in \left( \frac{d-1}{d} ,1 \right)$ if $\Omega =\R^d$.
Then for any non-negative initial condition $f_{\text{in}}\in L^1\cap L^\infty (\Omega\times\R^d)$ such that
$$ \iint_{\Omega\times \R^d} (|x|^a + |v|^b) f_{\text{in}} (x,v)\, \dd x \, \dd v <\infty$$
for $a$, $b\in(0,2]$, small enough,
there exists a unique nonnegative weak solution $f$ to
\begin{equation}
\begin{cases}
\pa_t f + v\cdot\na_x f = \Delta_v f^m +\text{div}_v (vf) & \mbox{ in } (0,\infty)\times \Omega\times \R^d\\
f(0,x,v) = f_{\text{in}}(x,v) & \mbox{ in } \Omega\times \R^d.
\end{cases}
\end{equation}
For all $T>0$, we have $f\in L^\infty((0,T);L^1\cap L^\infty(\Omega\times \R^d))$ and $\na_v f^{\frac{q}{2}} \in L^2((0,T)\times \Omega\times\R^d)$ for all $q>m$.

Furthermore, if $m \in \left( \frac{d}{d+2} ,1 \right)$ when $\Omega =\T^d$ or $m \in \left( \frac{d}{d+1} ,1 \right)$ when $\Omega =\R^d$, then
$f$ satisfies the entropy inequality
\begin{equation}
\mathcal E[f(t)] + \int_0^t \iint_{\Omega\times \R^d} f \left| v + \frac{m}{m-1}\na_v f^{m-1}\right|^2 \, \dd x \, \dd v \leq \mathcal E[f_{\text{in}}].
\end{equation}

Finally, if there is $M_0$ such that $f_{\text{in}}(x,v) \leq \left( M_0+\frac{1-m}{m}\frac{|v|^2}{2}\right)^{\frac{-1}{1-m}}$ in $\Omega\times\R^d$, then
\begin{equation}f(t,x,v) \leq \left( M_0+\frac{1-m}{m}\frac{|v|^2}{2}\right)^{\frac{-1}{1-m}} \qquad \mbox{ in } (0,T)\times \Omega\times\R^d.
\end{equation}
\end{theorem}

\newpage\subsection{Diffusion limits}

\subsubsection{Formal asymptotics}
We now turn to diffusion limits of \eqref{eq:mainVFP}. We consider the rescaled equation
\begin{equation}\label{Eqn}
\var^2\partial_t f_\var+ \var\,v\cdot\nabla_xf_\var= \Delta_vf_\var^m+\nabla_v\cdot (v\,f_\var), \qquad t > 0, \; x \in \Omega, \; v \in \R^d
\end{equation}
with initial condition
\begin{equation}\label{eq:initial_epsilon}
f_\var(0,x,v) = f_{in,\var}(x,v).
\end{equation}
The solutions of \eqref{Eqn} given by Theorem \ref{thm:existence} and \ref{thm:existencefast} satisfy in particular
\begin{equation}\label{entro}
\frac{d}{dt} \mathcal E[f_\var(t) ]
= - \frac{1}{\eps^2} \iroxv{f _\var\left| v + \frac{m}{m-1}\na_v f_\var^{m-1}\right|^2 }\leq 0.
\end{equation}

When $m=1$, it is a classical result that in the limit $\eps\to0$
the solution $f_\var(t,x,v) $ converges to a thermodynamical equilibrium function $G[\rho](v)$ whose density solves a linear diffusion equation.
We will prove that when $m\neq1$, the same limit leads to a nonlinear equation, of porous media type when $m>1$ and fast diffusion type when $m<1$.

In order to formally derive these limiting equations, we define the macroscopic quantities associated to \eqref{Eqn}:
\[
\rho_\var(t,x):=\irdv{f_\var(t,x,v)}, \qquad j_\var(t,x):=\frac1\var\irdv{v\,f_\var(t,x,v)},
\]
and
$$
P_\eps(t,x) := \irdv{v\otimes v \,f_\var(t,x,v)}.
$$
These functions solve, at least formally,
\begin{equation}\label{eq:macro}
\begin{cases}
\pa_t \rho_\eps + \text{div}_x j_\eps = 0, \\
\eps^2\pa_t j_\eps + \text{div}_x P_\eps = - j_\eps.
\end{cases}
\end{equation}
(the continuity equation is obtained by integrating \eqref{Eqn} with respect to $v$,
while the momentum equation follows from multiplying \eqref{Eqn} by $v$ and integrating).
The entropy inequality \eqref{entro} suggests that $f_\eps(t,x,v) $ converges to $G[\rho(t,x)](v)$ when $\eps\to0$ and therefore
$$
P_\eps(t,x) \to \irdv{v\otimes v\, G[\rho(t,x)](v)}
$$
The fact that $G[\rho(t,x)](v)$ only depends on $|v|^2$ implies that $\irdv{v_i v_j \, G[\rho(t,x)](v)}=0$ for all $i\neq j$ and so (recalling the definition \eqref{eq:nu})
$$
\irdv{v\otimes v\, G[\rho(t,x)](v)} = \frac 1 d \irdv{ |v|^2 \,G[\rho(t,x)](v)}\, {\mathbb I} = \nu(\rho) \, {\mathbb I} .
$$

Passing to the limit in \eqref{eq:macro} thus yields (formally of course)
$$ \frac{\partial\rho }{\partial t}+\div _x j =0 , \qquad j = -\na_x \nu(\rho).$$
Hence the limiting density $\rho(t,x)$ solves the nonlinear diffusion equation
\begin{equation*}\label{DL}
\frac{\partial\rho}{\partial t}= \Delta_x\nu (\rho),\qquad \nu(\rho) = \nu_1 \rho^k.
\end{equation*}
This is a nonlinear diffusion equation, of porous media type ($k\in(1,\frac{d+2}{d})$) when $m>1$ and fast diffusion type ($k\in(0,1)$) when $\frac{d}{d+2}<m<1$.

Differentiating the equation $\nu(\rho) = \irdv{ [\rho](v)^m}$ (see \eqref{eq:nu}), we derive the following relation that will be useful later on:
\begin{equation}\label{eq:numu}
\nu'(\rho) = \frac{m}{m-1} \mu'(\rho)\rho.
\end{equation}
This allows us to write the limiting equation in the form
\begin{equation}\label{DL1}
\frac{\partial\rho}{\partial t}=\text{div}_x \left(\rho \na \frac{m}{m-1}\mu(\rho)\right)
\end{equation}
with $p= \frac{m}{m-1}\mu(\rho)$ playing the classical role of the pressure in the porous media equation.

\medskip Establishing the convergence of $f_\eps(t,x,v)$ to an equilibrium $G[\rho(t,x)](v)$ with $\rho(t,x)$ solution of \eqref{DL1} is more challenging than in the linear case $m=1$.
We will develop here two approaches to such a result: The first approach relies on a classical compactness argument and the second one relies on a relative entropy method, which is not so classical in the context of diffusion limit.

The first method requires fewer assumptions on the initial data but involves stronger restrictions on $m$ (when $m<1$). It does not provide any rate of convergence.
The second approach requires the existence of smooth solutions of the limiting equation (which we can only reasonably expect to hold in bounded domain when $m>1$ and requires smooth initial conditions).
This approach provides some convergence result in the full range $m>\frac{d}{d+2}$ and gives a convergence rate.

\subsubsection{Diffusion limit via compactness method}

\begin{theorem}[Diffusion limit: compactness method]
Assume $m > 1$, $\Omega =\T^d$ or $\R^d$.
Let $f_{\text{in}}$ be a non-negative function satisfying
\begin{equation}\label{eq:initieps}
f_{\text{in}}\in L^1(\Omega\times \R^d), \quad \E[f_{\text{in}}] <\infty
\end{equation}
and there exists $M>0$ such that
\begin{equation}\label{eq:finB}
f_{\text{in}}(x,v) \leq G[M](v) \qquad \forall v\in \R^d, \; x\in \Omega.
\end{equation}
When $\Omega=\R^d$, we further assume
$$ \iint_{\R^d\times \R^d} \lfloor x \rceil ^\delta f_{\text{in}}(x,v)\dd v\dd x<\infty \qquad \mbox{ for some } \delta\in(0,1) .$$
Denote by $f_\eps$ the (unique) solution to \eqref{eq:mainVFP} given by Theorem \ref{thm:existence} and define $\rho_\eps= \int_{\R^d} f_\eps(\cdot,\cdot,v)\, \dd v$.
Then, for all $T>0$, $\rho_\eps(t,x)$ converges strongly in $L^1((0,T)\times\Omega)$ to $\rho(t,x)$ and $f_\eps(t,x,v)$ converges strongly in $L^1((0,T)\times\Omega\times\R^d)$ to the function $G[\rho(t,x)](v)$ where $\rho$ is the (unique) solution to
$$
\pa_t \rho - \nu_1\Delta \rho^k =0\quad \mbox{ in } [0,T]\times \Omega, \qquad \rho(0,\cdot) = \int_{\R^d} f_{\text{in}}(\cdot,v)\, \dd v,
$$
with $k$ and $\nu_1$ defined by \eqref{eq:km} and \eqref{eq:nu1}.
\end{theorem}
Note that the result holds if $f_{\text{in},\eps}$ depends on $\eps$ as long as it satisfies the bounds \eqref{eq:initieps} uniformly in $\eps$ and $\int_{\R^d} f_{\text{in},\eps}(\cdot,v)\, \dd v$ has a limit $\rho_{\text{in}}$ when $\eps\to0$.

\begin{theorem}[Diffusion limit: compactness method]
Assume $m \in (\frac d{d+2},1)$ when $\Omega =\T^d$ (and $m>\frac 1 2$ when $d=2$ and $m>\frac 1 3$ when $d=1$) and $m \in (\frac{d}{d+1},1)$ when $\Omega=\R^d$.
Let $f_{\text{in}}$ be a non-negative function satisfying
\begin{equation}
f_{\text{in}}\in L^1(\Omega\times \R^d), \quad \E[f_{\text{in}}] <\infty
\end{equation}
and
\begin{equation}
f_{\text{in}}(x,v) \leq G[M](v) \qquad \forall v\in \R^d, \; x\in \Omega.
\end{equation}
When $\Omega=\R^d$, we further assume
$$ \iint_{\R^d\times \R^d} |x|^2 f_{\text{in}}(x,v)\dd v\dd x<\infty .$$
Denote by $f_\eps$ the (unique) solution of \eqref{eq:mainVFP} on $[0,T]$ given by Theorem \ref{thm:existence} and define $\rho_\eps= \int_{\R^d} f_\eps(\cdot,\cdot,v)\, \dd v$.
Then for all $T>0$, $\rho_\eps$ converges strongly in $L^1((0,T)\times\Omega)$ to $\rho$ and $f_\eps$ converges strongly to the function $G[\rho]$ where $\rho$ is the (unique) solution to
$$
\pa_t \rho - \nu_1\Delta \rho^k =0\quad \mbox{ in } [0,T]\times \Omega, \qquad \rho(0,x) = \int_{\R^d} f_{\text{in}}(x,v)\, \dd v.
$$
\end{theorem}

At this point, the dependence on the exponents deserves an explanation. While $m>\frac{d-2}d$ corresponds to the Herrero-Pierre exponent in~\cite{MR797051} for the existence in $\R^d$, we have the same limitation if $\Omega=\T^d$ but need $m>\frac{2d-2}{2d}=\frac{d-1}d$ if $\Omega=\R^d$. A similar limitation appears for the diffusion limit with $m>\frac d{d+2}$ on $\T^d\times\R^d$ and $m>\frac{2d}{2d+2}=\frac d{d+1}$ on $\R^d\times\R^d$, because second moments are also needed.

\subsubsection{Diffusion limit via relative entropy method.}
The convergence of $f_\var$ can also be proved using some carefully defined relative entropy.
A first attempt at such a proof would be to compare the microscopic distribution $f_\eps(t,x,v)$ to the thermodynamical equilibrium $G[\rho_0(t,x)](v)$ (where $\rho_0$ solves \eqref{DL}) by defining
$$
\mathcal E[f_\eps|\rho_0](t) = \frac{1}{m-1}\irdxv{ f_\eps^m - G[\rho_0]^m-m G[\rho_0]^{m-1} (f_\eps-G[\rho_0])}.
$$
This approach, however, does not seem to work. Instead, we compare $f_\eps$ to the local equilibrium with density $\rho_0$ and with macroscopic flux close to $ \irdv{ v f_\eps(\cdot,\cdot,v)}= \eps j^\eps \sim - \eps \na_x \nu(\rho_0)$.
For a given density $\rho>0$ and velocity $u\in\R^d$, we thus define
\begin{equation}\label{eq:Gw}
G[\rho,u]: v \mapsto \left(\mu[\rho]-\frac{m-1}{2\,m}\,|v-u|^2\right)_+^\frac1{m-1}.
\end{equation}
which satisfies
$$\irdv{G[\rho,u](v)} = \rho, \qquad \irdv{ v G[\rho,u](v)} = \rho u.$$

Recalling the formulation \eqref{DL1} of the asymptotic equation, we define
the macroscopic velocity
\begin{equation}\label{eq:w}
w_0 = -\frac{m}{m-1}\na_x \mu(\rho_0),
\end{equation}
and we will compare $f_\eps$ to $G[\rho_0,\eps w_0]$ (which has flux $\frac 1 \eps \irdv{ v G[\rho_0,\eps w_0](v)} = \rho_0 w_0$).

This leads us to the following relative entropy functional:
\begin{equation}\label{eq:RE}
\mathcal E[f_\eps|\rho_0,w] = \frac{1}{m-1}\irdxv{ f_\eps^m - G[\rho_0,\eps w_0]^m-m G[\rho_0,\eps w_0]^{m-1} (f_\eps-G[\rho_0,\eps w_0])}.
\end{equation}
The convexity of the function $s\mapsto \frac{s^{m}}{m-1}$ implies that $\mathcal E[f_\eps|\rho_0,w_0]\geq 0$ and $\mathcal E[f_\eps|\rho_0,w_0]=0$ if and only if $f(t,x,v)=G[\rho_0(t,x),\eps w_0(t,x)](v)$.

\begin{remark}
In addition, we can rewrite this relative entropy in a way that separate the approach of $f_\eps$ to the thermodynamical equilibrium $G[\rho_\eps,\eps w_0]$ from the convergence of $\rho_\eps $ to $\rho_0$. Indeed, we have
\begin{align}
\mathcal E[f_\eps|\rho_0,w_0]
& = \frac{1}{m-1}\irdxv{ f_\eps^m - G[\rho_\eps,\eps w_0]^m-m G[\rho_\eps,\eps w_0]^{m-1} (f_\eps-G[\rho_\eps,\eps w_0])}\nonumber \\
& \qquad + \frac{1}{k-1} \irdx{\nu[\rho_\eps] - \nu[\rho_0] - \nu'[ \rho_0] (\rho_\eps-\rho_0) }
\end{align}
where we recall that $\rho_\eps (t,x) = \irdv{f_\eps(t,x,v)}$.
\end{remark}
Next, using the definition of $G$, we observe that \eqref{eq:RE} satisfies
\begin{align*}
\mathcal E[f_\eps|\rho_0,w_0]
= \frac{1}{m-1} \irdxv{ f_\eps^m - m \(\mu[\rho_0]-\frac{m-1}{2\,m}\,|v-\eps w_0|^2\)_+ f_\eps } + \irdx{ \nu(\rho_0)}\\
\leq \frac{1}{m-1} \irdxv{ f_\eps^m - m \(\mu[\rho_0]-\frac{m-1}{2\,m}\,|v-\eps w_0|^2\) f_\eps } + \irdx{ \nu(\rho_0)}.
\end{align*}
(with equality when $m<1$ and $m>1$ if $f_\eps$ and $G[\rho_0,\eps w_0]$ have the same support).
The quantity in the right hand side turns out to be better suited for the computations that follows. As a consequence, we define
\begin{align}\label{eq:REH}
\mathscr H_\eps(t) = \frac{1}{m-1} \irdxv{ f_\eps^m - m \(\mu[\rho_0]-\frac{m-1}{2\,m}\,|v-\eps w_0|^2\) f_\eps } + \irdx{ \nu(\rho_0)}.
\end{align}
\begin{remark}
We can write
\begin{align*}
\mathscr H_\eps(t) =& \irdxv{ \left( \frac{f_\eps^m}{m-1} + \frac{1}{2}\,|v-\eps w_0|^2 f_\eps - \frac{m}{m-1} \mu[\rho_0] f_\eps \right)} + \irdx{ \nu(\rho_0)}\\
& = \irdxv{ \left( \frac{f_\eps^m}{m-1} + \frac{1}{2}\,|v-\eps w_0|^2 f_\eps\right)} \\
&-\irdxv{ \left( \frac{G[\rho_0,\eps w_0]^m}{m-1} + \frac{1}{2}\,|v-\eps w_0|^2 G[\rho_0,\eps w_0]\right)} \\
& \quad +\frac{\nu_1}{k-1} \irdx{\rho_\eps(x)^k}
- \irdx{ \frac{m}{m-1} \mu[\rho_0] \rho_\eps } + \nu_1 \irdx{ \rho_0^k}
\end{align*}
where we used the fact that
$$
\mathcal E[G[\rho]] = \frac{\nu_1}{k-1} \irdx{\rho(x)^k}.
$$
Using \eqref{eq:nu}, \eqref{eq:nu1} and \eqref{eq:km}, we deduce
\begin{align*}
\mathscr H_\eps(t) =& \irdxv{ \left( \frac{f_\eps^m}{m-1} + \frac{1}{2}\,|v-\eps w_0|^2 f_\eps \right)} \\
&-\irdxv{\left( \frac{G[\rho_0,\eps w_0]^m}{m-1} + \frac{1}{2}\,|v-\eps w_0|^2 G[\rho_0,\eps w_0]\right)}\\
& + \frac{\nu_1}{k-1} \irdx{\left(\rho_\eps^k-\rho_0^k-k\rho_0^{k-1}(\rho_\eps-\rho_0)\right)}
\end{align*}
\end{remark}

One of the main contribution of this paper is the following proposition, which we can use to prove the convergence of $ \mathcal E_\eps(t)$ to zero:

\begin{proposition}\label{prop:relativeentropym}
Given $f_\eps$ solution to \eqref{Eqn} and $\rho_0$ solution to \eqref{DL1}, the relative entropy defined by \eqref{eq:REH} satisfies:
\begin{align}
& \frac{d}{dt}\mathscr H_\eps(t) +\frac{1}{4 \eps^{2}} \irdxv{ \frac{1}{ f_\eps} |\na_v f_\eps^m +(v- \eps w_0) f_\eps |^2 } \nonumber \\
&\qquad \leq |k-1|\| \div w\|_{L^\infty} \mathscr H_\eps(t)\nonumber \\
&\qquad\quad + \eps^4 \|\pa_t w\|^2_{L^\infty} \irdxv{ f_\eps } + \eps^2 \|Dw\|_{L^\infty}^2 \irdxv{ | v| ^2 f_\eps} \nonumber\\
&\qquad \quad
+\eps^2 \left(\frac {k-1 } 2\right)^2 \|\div w\|^2 _{L^\infty} \irdxv{|v-\eps w_0|^2 f_\eps }\label{eq:relativeentropy}
\end{align}
where we recall that $w_0= -\na \frac{m}{m-1} \mu(\rho_0) = - \nu_1\frac{k}{k-1} \na \rho_0^{k-1} $.
\end{proposition}

Given an initial condition $\rho_{in,0}(x)$, we define $ w_{in,0} = -\frac{m}{m-1}\na_x \mu(\rho_{in,0})$ and set
\begin{equation}\label{eq:initeps}
f_{in,\eps}(x,v) = G[\rho_{in,0},\eps w_{in,0}].
\end{equation}
Then we have
\begin{align*}
\mathscr H_\eps(0) &= \frac{1}{m-1} \irdxv{ f_{in,\eps}^m - m \(\mu[\rho_{in,0}]-\frac{m-1}{2\,m}\,|v-\eps w_{in,0}|^2\) f_{in,\eps} } + \irdx{ \nu(\rho_{in,0})}\\
&=0.
\end{align*}

\begin{theorem}
Assume that $m>\frac{d}{d+2}$ and $\Omega=\T^d$. Let $\rho_{in,0}(x)$ be an initial condition such that the solution of
$$
\begin{cases}
\pa_t \rho - \nu_1 \Delta \rho^k = 0 & \mbox{ in } (0,T)\times\T^d \\
\rho(0,x) = \rho_{in,0}(x) & \mbox{ in } \T^d
\end{cases}
$$
satisfies
$$
\| D^2 \rho_0^{k-1}\|_{L^\infty((0,T)\times\T^d)} <\infty, \qquad \|\pa_t \na \rho_0^{k-1}\|_{L^\infty((0,T)\times\T^d)} <\infty
$$

Let $f_\eps(t,x,v)$ be the solution of \eqref{Eqn} with initial data given by \eqref{eq:initeps}. Then:
$$
\mathscr H_\eps(t)\to 0
$$
and $f_\eps(t,x,v)$ converges strongly in $L^1((0,T)\times\T^d\times\R^d)$ to $G[\rho_0]$.
\end{theorem}
Note that the initial data \eqref{eq:initeps} could be replaced with the condition $\mathscr H_\eps(0)\to0$ as $\eps\to0$.

\section{Existence and uniqueness of weak solutions to \texorpdfstring{\eqref{Eqn}}{Eqn}}

We will provide below (\Cref{sec:mgtr1T}) a complete proof of the existence of weak solutions to \eqref{Eqn} when $m>1$ and $\Omega=\T^d$.
We will then describe the modifications required to handle the case $m<1$ with $\Omega =\T^d$ (\Cref{sec:m1}) and then with $\Omega =\R^d$ (\Cref{sec:oR}).

\subsection{Existence of weak solutions when \texorpdfstring{$m>1$}{m>1} and \texorpdfstring{$\Omega=\T^d$ ($m>1$)}{Omega=Td}} \label{sec:mgtr1T}

The outline of the proof is as follows:
\begin{enumerate}
\item The first step consists in proving the existence of a solution for a regularised (uniformly elliptic) equation (see \eqref{eq:regular}) via a fixed point argument.
\item Next, we will show that up to the truncation and regularisation of the initial data, the solution of this regularised equation is smooth and satisfies a number of important a priori estimates.
\item We then pass to the limit in the regularisation of the initial data by deriving other important estimates.
\item Finally, we will pass to the limit in the elliptic regularisation of the equation.
\end{enumerate}

Both the fixed point argument in the first step and the limits in the final step require some compactness arguments.
Such compactness will be obtained following a classical approach, first developed by
Lions in \cite{Lions1994BoltzmannLandau}, which relies on a combination of an averaging lemma and a $H^1_v$ regularity of the solution.
For the final step, we cannot expect to use this method to get compactness on $f$, due to the degeneracy of the diffusion, and we will instead prove some compactness for $f^m$.

\noindent{\bf \# Step 1:} \underline{Existence for a regularised equation}.

The first step is to regularise the equation into a non-linear but uniformly elliptic equation.
More precisely, we have the following result:

\begin{proposition}\label{prop:regular}
Assume that $\Phi: \R\to\R$ is a continuous function such that
$$0< \kappa \leq \Phi(s)\leq M, \qquad \mbox{ for all } s\in \R.$$
Given a non-negative function $f_{\text{in}}$ as in Theorem \ref{thm:existence}, there exists a nonnegative solution $f$ to,
\begin{equation}\label{eq:regular}
\begin{cases}
\pa_t f+ v\cdot\na_x f = \text{div}_v (\Phi(f)\na_v f + vf) \qquad & \mbox{ in } [0,T]\times \T^d\times \R^d,\\[8pt]
f(0,\cdot,\cdot)= f_{\text{in}} & \mbox{ in } \T^d\times \R^d,
\end{cases}
\end{equation}
such that
\begin{equation}\label{eq:LpLinfty} \| f(t)\|_{L^1(\T^d\times\R^d)} = \| f_{\text{in}}\|_{L^1(\T^d\times\R^d)} , \quad \| f(t)\|_{L^\infty(\T^d\times\R^d)} \leq e^{dT} \| f_{\text{in}}\|_{L^\infty(\T^d\times\R^d)},\quad \forall t\in [0,T],
\end{equation}
and
$$ \sup_{t\in [0,T]} \iint_{\T^d\times\R^d} |v|^2 |f(t,x,v)|^2\, \dd x \, \dd v\leq C(\kappa,M,f_{\text{in}},T).
$$
Furthermore, if $s\mapsto \Phi(s)$ is smooth and $f_{\text{in}} \in \mathcal{C}^{\infty}(\T^d\times \R^d)$, then $f$ is smooth.
\end{proposition}

Concerning the last statement (regularity of $f$), we note that when $s\mapsto \Phi(s)$ is smooth and $f_{\text{in}} \in \mathcal{C}^{\infty}(\T^d\times \R^d)$, the regularity of $f$ follows from the standard regularity theory:
since the solution is in $L^\infty(\T^d\times \R^d)$, De Giorgi estimates yield some $\mathcal{C}^\alpha(\T^d\times \R^d)$ regularity.
We can then iterate Schauder estimates to get as much smoothness as desired. The main difficulty is thus to show the existence of $f$, and this will be done
by a fixed point argument.

\medskip Given $g\in L^2([0,T]\times \T^d\times \R^d)$, we denote by $\mathsf{T}[g]:= \tilde{g}$ the unique solution to
\begin{equation}\label{eq:fTg}
\begin{cases}
\pa_t \tilde{g} + v\cdot\na_x \tilde{g} = \text{div}_v (\Phi(g)\na_v \tilde{g} +v\tilde{g}) & \mbox{ in } [0,T]\times \T^d\times \R^d,\\
\tilde{g}(0,\cdot,\cdot) = f_{\text{in}} & \mbox{ in } \T^d\times \R^d.
\end{cases}
\end{equation}
The existence and uniqueness of $\tilde{g}$ is ensured by the following lemma:

\begin{lemma}\label{lem:fixptT}
Assume that $f_{\text{in}}\geq 0$ and $f_{\text{in}}\in L^1\cap L^\infty(\T^d \times \R^d)$. Then
equation \eqref{eq:fTg} has a unique weak solution $\tilde{g}\in L^2([0,T]\times\T^d; H^1(\R^d))$ which satisfies \eqref{eq:fTg} in the sense of distributions. Furthermore, $\tilde{g}$ is non-negative and satisfies:
\begin{equation}\label{eq:H1}\sup_{t\in [0,T]} \| \tilde{g}\|_{L^2( \T^d\times \R^d)} + \kappa \| \na_v \tilde{g}\|_{L^2([0,T]\times \T^d\times \R^d)} \leq C(T)
\| f_{\text{in}}\|_{L^2( \T^d\times \R^d)}
\end{equation}
(with $C(T)$ dependent on $T$ but independent of $g$) and
\begin{equation}\label{eq:L1Lin} \| \tilde{g}(t)\|_{L^1} = \| f_{\text{in}}\|_{L^1}, \qquad \| \tilde{g}(t)\|_{L^\infty} \leq \| f_{\text{in}}\|_{L^\infty} \exp(dt) \qquad \forall t\in [0,T].
\end{equation}
Finally, there exists $C=C(\kappa,M,T,f_{\text{in}})$ (independent of $g$) such that
\begin{equation}\label{eq:vv2}
\sup_{t\in [0,T]} \int_{ \T^d\times \R^d} |v|^2\tilde{g}(t,x,v)^2\, \dd x \, \dd v \leq \int_{ \T^d\times \R^d} |v|^2f_{\text{in}}(x,v)^2\, \dd x \, \dd v \exp({C(\kappa,M,d) T}).
\end{equation}
\end{lemma}

\begin{proof}[{\bf Proof of Lemma \ref{lem:fixptT}}]

The existence and uniqueness of $\tilde{g}$ solution to \eqref{eq:fTg} can be proved using Degond's approach in \cite{Degond86}.

This solution satisfies $ \tilde{g}\in L^2([0,T]\times\T^d; H^1(\R^d))$, and $\pa_t \tilde{g}+v\cdot\na_x \tilde{g} \in L^2([0,T]\times\T^d; H^{-1}(\R^d))$. The bounds \eqref{eq:H1}, \eqref{eq:L1Lin} are classical and the estimate \eqref{eq:vv2} follows from the following computation:
\begin{align*}
\frac{d}{dt} \frac 1 2 \iint_{ \T^d\times \R^d} |v|^2\tilde{g}(t,x,v)^2\, \dd x \, \dd v
& = - \int_{ \T^d\times \R^d} \na_v( |v|^2 \tilde{g}) \cdot (\Phi(g)\na_v \tilde{g} + v\tilde{g} )\, \dd x \, \dd v\\
& \leq - \kappa \iint_{ \T^d\times \R^d} |v|^2 |\na_v \tilde{g}|^2\, \dd x \, \dd v
- 2 \iint_{ \T^d\times \R^d} \tilde{g} v \cdot (\Phi(g)\na_v \tilde{g} + v\tilde{g} )\, \dd x \, \dd v \\
& \quad - \iint_{ \T^d\times \R^d} |v|^2 \tilde{g}v\cdot \na_v \tilde{g} \, \dd x \, \dd v\\
& \leq - \kappa \iint_{ \T^d\times \R^d} |v|^2 |\na_v \tilde{g}|^2\, \dd x \, \dd v
+2 M \iint_{ \T^d\times \R^d} \tilde{g} |v||\na_v \tilde{g}|\, \dd x \, \dd v\\
& \quad- 2 \iint_{ \T^d\times \R^d} |v|^2 \tilde{g}^2 \, \dd x \, \dd v
+\frac{d+2}{2} \iint_{ \T^d\times \R^d} |v|^2 \tilde{g}^2\, \dd x \, \dd v\\
& \leq - \frac{\kappa}{2} \iint_{ \T^d\times \R^d} |v|^2 |\na_v \tilde{g}|^2\, \dd x \, \dd v
+C(\kappa,M,d) \iint_{ \T^d\times \R^d} |v|^2 \tilde{g}^2\, \dd x \, \dd v.
\end{align*}
A Gr\"onwall argument then yields \eqref{eq:vv2}.
This computation can be justified rigorously by approximating $|v|^2$ by $\beta_k(v) = \min\{|v|^2,k\}$ (so that $\beta_k \tilde{g}^2$ is integrable) to first get
\begin{align*}
&\frac{d}{dt} \frac 1 2 \iint_{ \T^d\times \R^d} \beta_k(v) \tilde{g}(t,x,v)^2\, \dd x \, \dd v\\
& \leq - \kappa \iint_{ \T^d\times \R^d} \beta_k(v) |\na_v \tilde{g}|^2\, \dd x \, \dd v
+M \iint_{ \T^d\times \R^d} \tilde{g} |\na_v\beta_k(v)||\na_v \tilde{g}|\, \dd x \, \dd v \\
& \quad - \iint_{ \T^d\times \R^d} v\cdot\na_v \beta_k(v) \tilde{g}^2 \, \dd x \, \dd v
+ \frac 1 2 \iint_{ \T^d\times \R^d}\div(v\beta_k(v)) \tilde{g}^2\, \dd x \, \dd v.
\end{align*}
Using the fact that $|\na_v\beta_k(v) |\leq2 \sqrt {\beta_k(v)}$ and $|v \na_v\beta_k(v) |\leq2 \beta_k(v)$, we deduce
\begin{align*}
\frac{d}{dt} \frac 1 2 \iint_{ \T^d\times \R^d} \beta_k(v) \tilde{g}(t,x,v)^2\, \dd x \, \dd v
\leq C(\kappa,M,d) \iint_{ \T^d\times \R^d} \beta_k(v) \tilde{g}^2\, \dd x \, \dd v.
\end{align*}
This yields \eqref{eq:vv2} with $\beta_k(v)$ instead of $|v|^2$ and the result follows by passing to the monotonous limit $k\to\infty$.
\end{proof}

\medskip We are now ready to prove Proposition \ref{prop:regular} via a fixed point argument.

\begin{proof}[{\bf Proof of Proposition \ref{prop:regular}}]

Let us show that the operator $\mathsf{T}: L^2([0,T]\times \T^d\times \R^d)\to L^2([0,T]\times \T^d\times \R^d)$ is continuous, compact and bounded:
$$ \| \mathsf{T}\|_{L^2([0,T]\times \T^d\times \R^d)\to L^2([0,T]\times \T^d\times \R^d)} \leq C,$$
(where $C$ depends on the initial condition $f_{\text{in}}$, but not on $g$).

The main step in the proof of this is the compactness of the operator $\mathsf{T}$ since all other claims come from the previous lemma. We present this argument in details here since a similar approach will be used to pass to the limit in the regularisation of the function $\Phi$.

Consider a bounded sequence $\left( g_k \right)_{k\in \N}$ in $L^2([0,T]\times \T^d\times \R^d)$.
We note that $\tilde{g}_k:=\mathsf{T}[g_k]$ is bounded in $L^\infty([0,T]; L^p(\T^d\times \R^d))$ for all $p\in [1,\infty]$ and
\begin{enumerate}
\item $\pa_t \tilde{g}_k + v\cdot \na_x \tilde{g}_k$ is bounded in $L^2([0,T]\times\T^d; H^{-1}(\R^d))$.
\item $\tilde{g}_k$ is bounded in $L^2([0,T]\times\T^d; H^1(\R^d))$.
\item $|v| \tilde{g}_k $ is bounded in $L^\infty([0,T]; L^2(\T^d\times\R^d))$.
\end{enumerate}

First, we can always fix a subsequence along which
$$
\tilde{g}_k \rightharpoonup \g, \qquad \mbox{ weakly in } L^1((0,T)\times \T^d\times \R^d).
$$

We shall now show that the first two bounds implies that $\tilde{g}_k$ is relatively compact in $L^1([0,T]\times \T^d \times B_R)$ for all $R>0$.
More precisely, Theorem 1 in \cite{golse2025velocity} implies that for all $\psi \in \mathcal D(\R^d)$, we have that the family of functions,
$$ \tilde{g}_k \star \psi : v \mapsto \int_{\R^d} \tilde{g}_k(\cdot,\cdot,w) \psi(v-w)\, \dd w \mbox{ is relatively compact in } L^1((0,T)\times \T^d\times B_R) .$$
We thus fix a mollifier $\psi$ and set $\psi_\delta = \frac{1}{\delta^{d}}\psi(\frac{v}{\delta})$. We have,
$$\left( \tilde{g}_k \star \psi_\delta \right)_{k\in \N} \mbox{ is relatively compact in } L^1((0,T)\times \T^d\times B_R) \mbox{ for all $\delta>0$}.$$
and
$$ \| \tilde{g}_k \star \psi_\delta - \tilde{g}_k \|_{L^2((0,T)\times \T^d\times \R^d)} \leq \delta \|\na _v \tilde{g}_k \| _{L^2((0,T)\times \T^d\times \R^d)} \int_{\R^d} |u|\psi(u)\, d u \leq C\delta.
$$
It follows that for all $R>0$, there exists a subsequence along which
$$ \tilde{g}_k \to \g \qquad \mbox{ strongly in } L^1((0,T)\times \T^d\times B_R).$$
Since $\tilde{g}_k$ is bounded in $L^\infty([0,T]\times \T^d \times \R^d)$, the convergence also holds in $L^2([0,T]\times \T^d \times B_R)$.

Finally, we use the third bound to say that since $|v| \tilde{g}_k $ is bounded in $L^2([0,T]; L^2(\T^d\times\R^d))$, one can also extract to ensure that $|v| \tilde{g}_k$ converges weakly to $|v| \g$ in $L^2([0,T]; L^2(\T^d\times\R^d))$. By lower semi-continuity, that $\Vert |v| \g \Vert_{L^2([0,T]; L^2(\T^d\times\R^d))}$ is finite and smaller than $\sup_{k} \Vert |v| \tilde{g}_k \Vert_{L^2([0,T]; L^2(\T^d\times\R^d))} := \mathfrak{C}_k$. Altogether, we write,
\begin{align*}
\| \tilde{g}_k- \g\|_{L^2([0,T]\times \T^d \times \R^d)}^2
& \leq \| \tilde{g}_k- \g\|_{L^2([0,T]\times \T^d \times B_R)}^2 + \frac{1}{2 R^2}\iiint_{[0,T]\times \T^d \times B_R^c} |v|^2 \left( \tilde{g}_k^2 + \g^2 \right)\, \dd x \, \dd v\,dt \\
& \leq \| \tilde{g}_k -\g\|_{L^2([0,T]\times \T^d \times B_R)}^2 + \frac{\mathfrak{C}_k}{R^2}.
\end{align*}
We deduce that $\limsup_{k\to\infty} \| \tilde{g}_k-\g\|_{L^2([0,T]\times \T^d \times \R^d)}^2 \leq \frac {CT}{R^2}$ and since this holds for all $R>0$, we have
$$\lim_{k\to\infty} \| \tilde{g}_k-\g\|_{L^2([0,T]\times \T^d \times \R^d)}^2 =0.$$
\end{proof}

\noindent{\bf \# Step 2:} \underline{Regularisation of the initial data.}

Let $\Phi (s) := \min\{M,m|s|^{m-1} +\kappa\}$.
Proposition \ref{prop:regular} gives the existence of a solution to \eqref{eq:regular}, which is non-negative and bounded in $L^\infty(\T^d \times \R^d)$.

Since $f_{\text{in}}$ does not depend on $M$ and neither does the $L^\infty(\T^d \times \R^d)$ bound \eqref{eq:LpLinfty}, we can take $M$ large enough so that this solution is smaller that $M$ almost everywhere.

It is thus a solution to
\begin{equation}\label{eq:kappa}
\begin{cases}
\pa_t f_\kappa + v\cdot\na_x f_\kappa = \Delta_v f_\kappa^m + \kappa \Delta_v f_\kappa+\text{div}_v (vf_\kappa), \qquad \mbox{ in } [0,T]\times \T^d\times \R^d\\[5pt]
f_\kappa(0,\cdot,\cdot)= f_{\text{in}}.
\end{cases}
\end{equation}
(Note importantly that the only non trivial dependency in the parameters is the one in $\kappa$).

In order to derive the a priori estimates necessary to pass to the limit $\kappa\to0$, we first approximate the initial data $f_{\text{in}}$ by a sequence of smooth functions $f_{\text{in}}^{(k)}$ such that $f_{\text{in}}^{(k)}$ converges strongly to $f_{\text{in}}$ in $L^p(\T^d\times\R^d)$ (for all $p \geq 1$) and $\left\Vert f_{\text{in}}^{(k)}\right\Vert _{L^\infty(\T^d\times\R^d)} \leq \left\Vert f_{\text{in}} \right\Vert_{\mathrm L^\infty(\T^d\times\R^d)}$, and such that,
\begin{equation}\label{eq:regbar}
f_{\text{in}}^{(k)}(x,v) \leq \left( k-\frac{m-1}{m}\frac{|v|^2}{2}\right)_+^{\frac{1}{m-1}}, \qquad \mbox{ for all } x\in \T^d, \; v\in \R^d
\end{equation}
(we can further assume compact support for $f_{\text{in}}^{(k)}$).

Note that local equilibria to \eqref{eq:kappa} are not proper Barenblatt-Pattle profiles anymore. They solve
\begin{equation*}
\text{div}_v \left( \nabla_v G_\kappa^m + \kappa \nabla_v G_\kappa + vG_\kappa\right) = 0.
\end{equation*}
Looking for positive solutions vanishing at infinity, leads to
$$\nabla_v \left( \frac{m}{m-1}G_\kappa^{m-1} + \kappa \ln(G_\kappa) \right) = -v, $$
which is
\begin{equation*}
\frac{m}{m-1}f_\kappa^{m-1} + \kappa \ln(f_\kappa) = - \frac{\vert v \vert^2}{2} + \mu.
\end{equation*}
The function $\psi_\kappa(s) = \frac {m}{m-1} s^{m-1} + \kappa \log s $ is monotone increasing from $(0,+\infty)$ to $\R$ and satisfies,
$$\lim_{s\to 0 }\psi_\kappa(s) = -\infty,\quad \lim_{s\to +\infty}\psi_\kappa(s) = +\infty .$$
Therefore, it has an inverse $\psi_\kappa^{-1}:\R\to (0,\infty)$ satisfying
$$\lim_{s\to -\infty }\psi_\kappa^{-1}(s) = 0,\quad \lim_{s\to +\infty}\psi_\kappa^{-1}(s) = +\infty .$$
For any $\mu$, we thus define the function,
$$G_{\kappa,\mu}(v) = \psi_\kappa^{-1} \left(\mu -\frac{v^2}{2}\right), $$
which is therefore a local equilibrium and a stationary solution to \eqref{eq:kappa}. Furthermore, since $\psi_\kappa(s) \geq \kappa \log(s)$ (recall that $m>1$), we have $\psi_\kappa^{-1}(x)\leq e^{\frac{x}{\kappa}}$ and therefore,
$$
G_\kappa(v)\leq e^{\frac{\mu}{\kappa}-\frac{v^2}{2\kappa}}.
$$
Next, we note that $\psi(s) \leq \frac {m}{m-1} s^{m-1} +\kappa \ln(k)$ for $s\leq k$ and so
$$ \psi^{-1}(x) \geq \max\left\{\min\left\{ k, \left(\frac{m-1}{m} (x-\kappa \ln(k))\right)^{\frac{1}{m-1}} \right\} , 0 \right\} .$$
We thus have
$$
G_{\kappa,\mu}(v) \geq \max\left\{
\min\left\{ k,\left(\frac{m-1}{m} (\mu-\kappa \ln(k)) -\frac{m-1}{m} \frac{|v|^2}{2}\right)^{\frac{1}{m-1}} \right\} , 0 \right\}=
\left(k-\frac{m-1}{m} \frac{|v|^2}{2}\right)_+^{\frac{1}{m-1}},
$$
when $\mu = \kappa \ln(k) + k \frac{m}{m-1}$. Note that then, using \eqref{eq:regbar}, $G_{\kappa,\mu} \geq f_{\text{in}}^{(k)}$.

Since \eqref{eq:kappa} satisfies a comparison principle (which is straightforward for classical solutions), we conclude that
\begin{equation}\label{eq:maxpsi}
f_\kappa^{(k)}(t,x,v) \leq G_{\kappa,\mu}(v) \leq k e^{\frac{k}{\kappa} \frac{m}{m-1}} e^{-\frac{v^2}{2\kappa}}.
\end{equation}

\noindent{\bf \# Step 3:} \underline{A priori estimates and passing to the limit in the regularisation parameter.}

Working with the approximation $f_{\kappa}^{(k)}$ from the previous step, we recall that $f_{\kappa}^{(k)}$ is a non-negative smooth solution of \eqref{eq:kappa} which satisfies the decay estimate \eqref{eq:maxpsi}.
For $p>1$, we multiply \eqref{eq:kappa} by $\left(f_{\kappa}^{(k)}\right)^{p-1}$ and integrate to find:
\begin{equation} \label{eq:p}
\begin{split}
&\frac{d}{dt} \iint_{\T^d\times \R^d} \frac{\left\vert f_{\kappa}^{(k)} \right\vert^p}{p}\, \dd x \, \dd v\\
&\qquad + \frac{4m(p-1)}{(p+m-1)^2} \iint_{\T^d\times \R^d} \left\vert \na_v (f_{\kappa}^{(k)})^{\frac{p+m-1}{2}}\right\vert ^2 \, \dd x \, \dd v
+\kappa \frac{4(p-1)}{p^2} \iint_{\T^d\times \R^d} \left\vert \na_v (f_{\kappa}^{(k)})^{\frac{p}{2}}\right\vert ^2 \, \dd x \, \dd v\\
& \qquad \qquad \leq \frac{(p-1)d}{p} \iint_{\T^d\times \R^d} \left\vert f_{\kappa}^{(k)}\right\vert ^p \, \dd x \, \dd v.
\end{split}
\end{equation}
We deduce the following estimate:
\begin{align}
&\sup_{t\in [0,T]}\iint_{\T^d\times \R^d} \frac{\left|f_{\kappa}^{(k)}(t)\right|^p}{p}\, \dd x \, \dd v \nonumber \\
& \qquad +
\frac{4m(p-1)}{(p+m-1)^2} \iiint_{(0,T)\times \T^d\times \R^d} \left|\na_v (f_{\kappa}^{(k)})^{\frac{p+m-1}{2}}\right|^2 \, \dd t \, \dd x \, \dd v \nonumber \\
&\qquad \qquad +\kappa \frac{4(p-1)}{p^2} \iiint_{(0,T)\times \T^d\times \R^d} \left|\na_v (f_{\kappa}^{(k)})^{\frac{p}{2}}\right|^2 \, \dd t \, \dd x \, \dd v\nonumber \\
&\qquad \qquad \qquad \leq e^{(p-1)d T} \iint_{\T^d\times \R^d} \frac{\left|f_{\text{in}}^{(k)}\right|^p}{p}\, \dd x \, \dd v \label{eq:pp}
\end{align}
Furthermore, multiplying \eqref{eq:kappa} by $\frac{v^2}{2}$ and integrating (recall that $f_{\kappa}^{(k)}$ satisfies the decay estimate \eqref{eq:maxpsi}),
we get:
\begin{align*}
\frac{d}{dt} \iint_{\T^d\times \R^d} \frac{|v|^2}{2} f_{\kappa}^{(k)} \, \dd x \, \dd v + \iint_{\T^d\times \R^d} |v|^2 f_{\kappa}^{(k)} \, \dd x \, \dd v\leq \iint_{\T^d\times \R^d} d \left( \left\vert f_{\kappa}^{(k)} \right\vert^m + \kappa f_{\kappa}^{(k)} \right) \, \dd x \, \dd v,
\end{align*}
and so, after integrating,
\begin{equation}\label{eq:velocity}
\sup_{t\in [0,T]} \iint_{\T^d\times \R^d} \frac{|v|^2}{2} f_{\kappa}^{(k)} (t)\, \dd x \, \dd v \leq C\left(\left\|f_{\text{in}}^{(k)}\right\|_{L^m}, \kappa \left\|f_{\text{in}}^{(k)}\right\|_{L^1}, d,m, \iint_{\T^d\times \R^d} |v|^2 f_{\text{in}}^{(k)}\, \dd x \, \dd v\right).
\end{equation}
Finally, we can similarly derive the approximated entropy inequality:
\begin{align}
\mathcal E[f_{\kappa}^{(k)}](t)&
+ \int_0^t \iint_{\T^d\times \R^d} f_{\kappa}^{(k)} \left|v+ \frac{m}{m-1}\na_v \left(f_{\kappa}^{(k)}\right)^{m-1} \right|^2 \, \dd s \, \dd x \, \dd v\nonumber \\
& \qquad \leq -\kappa \frac{4}{m} \int_0^t \iint_{\T^d\times \R^d} \left|\na_v (f_{\kappa}^{(k)})^{\frac m2}\right|^2 \, \dd s \, \dd x \, \dd v + \kappa d\int_0^t \iint_{\T^d\times \R^d} f \, \dd s \, \dd x \, \dd v+ \mathcal E[f_{\text{in}}^{(k)} ]\nonumber \\
& \qquad \leq \kappa d \int_0^t\iint_{\T^d\times \R^d} f + \mathcal E[f_{\text{in}}^{(k)} ].\label{eq:entropykappa}
\end{align}

We can now pass to the limit on this regularisation ($k\to\infty$) by proceeding as in the compactness argument from \# Step 1. We obtain a solution of \eqref{eq:kappa} in the following weak sense
\begin{multline*}\label{eq:kappaweak}
\iiint_{[0,T] \times \T^d \times \R^d} \left(\pa_t \varphi + v\cdot\na_x \varphi \right)\,f_\kappa \, \dd s \,\dd x \, \dd v-\iiint_{[0,T] \times \T^d \times \R^d} \na_v \varphi \cdot \na_v f_\kappa^{m} + (v \cdot \na_v \varphi) \, f_\kappa \, \dd s \,\dd x \, \dd v \\=- \iint_{\T^d \times \R^d} \varphi(0,x,v) \, f_{\text{in}}(x,v) \, \dd x \, \dd v,
\end{multline*}
for all $\varphi \in \mathcal{D}\left([0,T) \times \T^d \times \R^d \right)$, (in particular $\varphi(T,x,v)=0$ for $(x,v) \in \T^d \times \R^d$) which satisfies the bounds \eqref{eq:pp}, \eqref{eq:velocity} and \eqref{eq:entropykappa}.

\medskip\noindent {\bf \# Step 4}: \underline{Limit $\kappa\to0$: Compactness of $f_\kappa$.}

First, we can always fix a subsequence along which
\begin{equation}\label{eq:weakf}
f_\kappa \rightharpoonup f \mbox{ weakly in } L^1((0,T)\times \T^d\times \R^d) \mbox{ as } \kappa\to0.
\end{equation}
In order to prove that $f_\kappa$ converges strongly, we rely on the same compactness argument as in \# Step 1. However, the bound on $f_\kappa$ in $L^2([0,T]\times\T^d; H^1(\R^d))$ is not uniform in $\kappa$, so we need to adapt this argument:
We will first prove strong convergence of $f_\kappa^{m}$ in $L^1([0,T]\times \T^d\times\R^d)$ and then deduce the convergence of $f_\kappa$.
For this, we need the following bounds:
\begin{enumerate}
\item The family $\left(\pa_t f_\kappa^m + v\cdot \na_x f_\kappa^m\right)_{\kappa \in (0,1]}$ is bounded in $L^2((0,T]\times\T^d; H^{-1}(\R^d)) + L^1 [0,T]\times\T^d\times \R^d)$.
\item The family $\left(f_\kappa^m\right)_{\kappa \in (0,1]}$ is bounded in $L^2([0,T]\times\T^d; H^1(\R^d))$.
\item The family $\left(|v|^2 f_\kappa^m\right)_{\kappa \in (0,1]} $ is bounded in $L^\infty([0,T]; L^1(\T^d\times\R^d))$.
\end{enumerate}

To prove the first bound, we first note that \eqref{eq:pp} with $p=m$ gives:
\begin{align}
& \sup_{t\in [0,T]]}\iint_{\T^d\times \R^d} \frac{|f_\kappa(t)|^m}{m}\, \dd x \, \dd v + \frac{4m(m-1)}{(2m-1)^2} \iiint_{(0,T)\times \T^d\times \R^d} \left\vert \na_v f_\kappa^{m-\frac{1}{2}} \right\vert^2 \nonumber \\
& \qquad \qquad +\kappa \frac{4(m-1)}{m^2} \iiint_{(0,T)\times \T^d\times \R^d} \left\vert \na_v f_\kappa^{\frac{m}{2}} \right\vert^2 \leq e^{(m-1)d T} \iint_{\T^d\times \R^d} \frac{|f_{\text{in}}|^m}{m}\, \dd x \, \dd v. \label{eq:pm}
\end{align}
Furthermore, equation \eqref{eq:kappa} implies
\begin{align*}
&\pa_t f_\kappa^m+ v\cdot\na_x f_\kappa^m \\
& =mf_\kappa^{m-1}[ \Delta_v f_\kappa^m + \kappa \Delta_v f_\kappa+ v \cdot\na_v f_\kappa + d f_\kappa] \\
& =m \div_v[ f_\kappa^{m-1} \na_v f_\kappa^m ]-m \na_v f_\kappa^{m-1}\cdot \na_v f_\kappa^m + \kappa m \div_v[ f_\kappa^{m-1} \na_v f_\kappa] \\
& \qquad \qquad - \kappa m \na_v f_\kappa^{m-1} \cdot \na_v f_\kappa + v \cdot\na_v f_\kappa^{m} + dm f_\kappa^m\\
& =\frac{m^2}{m-\frac 1 2} \div_v\left[ f_\kappa^{m-\frac 12} \na_v f_\kappa^{m-\frac 1 2} \right]-\frac{m^2(m-1)}{(m-\frac 1 2)^2} \left|\na_v f_\kappa^{m-\frac 12}\right|^2 + 2 \kappa\div_v\left[ f_\kappa^{\frac m 2} \na_v f_\kappa^{\frac m 2}\right] \\
& \qquad \qquad - \kappa \frac{4(m-1)}{m} \left|\na f_\kappa^{m-\frac 12} \right|^2 + 2 v f_\kappa^{\frac m 2} \cdot\na_v f_\kappa^{\frac m 2} + dm f_\kappa^m
\end{align*}
and recalling that $f_\kappa$ is bounded in $L^\infty(\T^d\times \R^d)$, we can use \eqref{eq:pm} to show that the right hand side is bounded in $L^2([0,T]\times\T^d; H^{-1}(\R^d)) + L^1 ([0,T]\times\T^d\times \R^d)$ uniformly with respect to $\kappa \leq 1$.

The second bound follows from \eqref{eq:pp} with $p=m+1$, which gives
\begin{equation}\label{eq:mm}
\sup_{t\in [0,T]}\iint_{\T^d\times \R^d} \frac{|f_\kappa(t)|^{m+1}}{m+1}\, \dd x \, \dd v +
\iiint_{(0,T)\times \T^d\times \R^d} \left\vert\na_v f_\kappa^{m} \right\vert^2 \, \dd s \, \dd x \, \dd v\leq e^{md T} \iint_{\T^d\times \R^d} \frac{|f_{\text{in}}|^{m+1}}{m+1}\, \dd x \, \dd v .
\end{equation}
The third bound is given by testing over $\vert v\vert^2$,
\begin{align*}
\frac{d}{dt} \int_{\T^d\times \R^d} \vert v \vert^2 f_\kappa^m
& = \frac12 \frac{m^2}{m-\frac 1 2} \int_{\T^d\times \R^d} d f_\kappa^{2m-1} \, \dd x \dd v + \kappa \int_{\T^d\times \R^d} d f_\kappa^{m} \, \dd x \dd v \\
&\qquad - \int_{\T^d\times \R^d} \div_v (v \vert v \vert^2) f_\kappa^{m} \, \dd x \dd v + dm \int_{\T^d\times \R^d} \vert v \vert^2 f_\kappa^m \, \dd x \dd v\\
& \qquad \qquad -\int_{\T^d\times \R^d} \left(\frac{m^2(m-1)}{(m-\frac 1 2)^2} + \kappa \frac{4(m-1)}{m}\right) \left|\na_v f_\kappa^{m-\frac 12}\right|^2 \, \dd x \dd v \\
&\leq \frac{m^2d}{2m-1} \int_{\T^d\times \R^d} f_\kappa^{2m-1} \, \dd x \dd v + \kappa d \int_{\T^d\times \R^d} f_\kappa^{m} \, \dd x \dd v \\
&\qquad \qquad - \int_{\T^d\times \R^d} (d+2-dm) \vert v \vert^2 f_\kappa^{m} \, \dd x \dd v.
\end{align*}
Since $2m-1 > 1$, we can take $p=2m-1$ in \eqref{eq:pp} to bound the first integral of the latter r.h.s. uniformly. We deduce that for $\kappa\leq 1$, we have
\begin{equation}\label{eq:momvkappa}
\sup_{t\in [0,T]}\int_{\T^d\times \R^d} \vert v \vert^2 f_\kappa^m (t,x,v) \dd v \dd x \leq C(f_{in},T).
\end{equation}
We can thus proceed as in the proof of Proposition \ref{prop:regular} to show that, up to extracting a further subsequence there exists a function $h$ such that
$$\lim_{\kappa \to 0} \, \| f_\kappa^m-h\|_{L^2([0,T]\times \T^d \times \R^d)}^2 =0,$$
and
$$\lim_{\kappa \to 0} \,f_\kappa^m = h, \qquad \mbox{ a.e. in }[0,T]\times \T^d \times \R^d.$$
This implies that $f_\kappa$ converges to $h^{\frac 1 m} $ almost everywhere and so the weak limit \eqref{eq:weakf} gives $f = g ^{\frac{1}{m}}$.
Finally, the bound in $L^\infty$ and the bound on the moment in $v$ yields:
$$ f_\kappa \to f \mbox{ strongly in } L^p([0,T]\times \T^d \times \R^d) \mbox{ for all } p\in(1,\infty).
$$

We can now pass to the limit in all the terms of the equation \eqref{eq:kappa}.
We note that the inequality \eqref{eq:p} with $p=2$ implies that $\kappa \Delta_v f_\kappa \to 0$ in $L^2([0,T]\times\T^d;H^{-1}(\R^d))$, so we recover the equation \eqref{eq:FP}.

\medskip It only remains to pass to the limit in \eqref{eq:entropykappa}: we note that since the entropy is convex, we only need to make sure that we can pass to the limit in the entropy dissipation, which we write:
\begin{align*}
& \int_0^t \iint_{\T^d\times \R^d} f_\kappa\left|v+ \frac{m}{m-1}\na_v f_\kappa^{m-1} \right|^2 \, \dd s \, \dd x \, \dd v\\
& = \int_0^t \iint_{\T^d\times \R^d} |v|^2 f_\kappa\, \dd s \, \dd x \, \dd v
+ \frac{2m}{m-1} \int_0^t \iint_{\T^d\times \R^d} f_\kappa v\cdot \na_v f_\kappa^{m-1} \, \dd s \, \dd x \, \dd v\\
& \qquad \qquad + \left( \frac{m}{m-1}\right)^2 \int_0^t \iint_{\T^d\times \R^d} f_\kappa\left| \na_v f_\kappa^{m-1} \right|^2\, \dd s \, \dd x \, \dd v\\
& = \int_0^t \iint_{\T^d\times \R^d} |v|^2 f_\kappa\, \dd s \, \dd x \, \dd v + \left( \frac{m}{m-\frac 1 2}\right)^2 \int_0^t \iint_{\T^d\times \R^d} \left| \na_v f_\kappa^{m-\frac 12} \right|^2 \, \dd s \, \dd x \, \dd v\\
& \qquad \qquad \qquad \qquad - 2 d \int_0^t \iint_{\T^d\times \R^d} f_\kappa^{m} \, \dd s \, \dd x \, \dd v,
\end{align*}
so that using the strong convergence of $f_\kappa^m$ in $L^1$ and the weak convergence of $\na_v f_\kappa ^{m-\frac 12}$ and $|v|^2 f_\kappa$ in $L^2$, we get by Fatou's lemma, up to extraction,
$$\liminf_{\kappa\to \infty}
\int_0^t \iint_{\T^d\times \R^d} f_\kappa \left|v+ \frac{m}{m-1}\na_v f_\kappa ^{m-1} \right|^2 \, \dd s \, \dd x \, \dd v
\geq
\int_0^t \iint_{\T^d\times \R^d} f \left|v+ \frac{m}{m-1}\na_v f ^{m-1} \right|^2 \, \dd s \, \dd x \, \dd v
$$

\subsection{Existence of weak solutions when \texorpdfstring{$\Omega=\T^d$ ($\frac{d-2}{d} < m<1$)}{Omega=Td (d/(d+2)≤m<1)}}\label{sec:m1}

In this section, we explain what has to be changed in the proof explained in \Cref{sec:mgtr1T}, to achieve the case of $m<1$ on the torus $\T^d$.

When $m<1$, we need to change the definition of the regularised elliptic coefficient to
$$ \Phi(f)= m\left(\max\{f,M\}+\kappa\right)^{m-1}.$$
It satisfies (for $f\geq 0$):
$$\frac{m}{(\kappa+M)^{1-m}}\leq \Phi(f) \leq \frac{m}{\kappa^{1-m}}.$$
We can thus use Proposition \ref{prop:regular} to get the existence of a nonnegative solution to \eqref{eq:regular}.

Since this solution satisfies the $L^\infty$ bound \eqref{eq:LpLinfty}, we can take $M$
large enough so that it is smaller that $M$ almost everywhere and we get thus a function $f_\kappa$ solving:
\begin{equation}\label{eq:kappa2}
\begin{cases}
\pa_t f_\kappa+ v\cdot\na_x f_\kappa = \Delta_v \left(f_\kappa+\kappa\right)^{m} +\text{div}_v (vf_\kappa) \qquad & \mbox{ in } [0,T]\times \T^d\times \R^d\\
f(0,\cdot,\cdot)= f_{\text{in}} & \mbox{ in } \T^d\times \R^d
\end{cases}
\end{equation}
Again, we can then approximate the initial data $f_{\text{in}}$ by smooth, compactly supported functions $f_{\text{in}}^{(k)}$.
The corresponding solution $f_\kappa^{(k)}$ is smooth.

In order to get an upper bound of the type \eqref{eq:maxpsi}, we have to study local equilibria to \eqref{eq:kappa2}. They satisfy,
$$\nabla_v (G_\kappa+ \kappa)^{m} + v G_\kappa = 0\qquad \Longleftrightarrow \qquad \frac{m}{G_\kappa (G_\kappa+ \kappa)^{1-m}}\nabla_v G_\kappa = - v .$$
We introduce the increasing function (for some $s_0>0$),
$$ \psi_\kappa(s) = \int_{s_0}^s \frac{m}{u(u+\kappa)^{1-m}}\, , \qquad s\in (0,+\infty).$$
A solution $G_\kappa$ satisfies
$$\psi_\kappa(G_\kappa(v)) = \mu - \frac{\vert v \vert^2}{2},$$
which has a unique solution for every $\mu$.

We have $\frac{m}{u(u+\kappa)^{1-m}} \leq \frac{m}{\kappa^{1-m}} \frac 1 u$ and so
$$
\psi_\kappa(s) \geq \frac{m}{\kappa^{1-m}} \log\left(\frac{s}{s_0}\right) \qquad \forall s\in (0,s_0].
$$
This gives, when $G_\kappa(v) \leq s_0$,
$$
s_0\exp\left( \frac{\kappa^{1-m}}{m}\left(\mu - \frac{\vert v \vert^2}{2} \right) \right)\geq G_\kappa(v).
$$
Since $\frac{m}{u(u+\kappa)^{1-m}} \leq \frac{m}{u^{2-m}}$, we have
$$
\psi_\kappa(s) \leq \frac{m}{m-1} \left( s^{m-1} -s_0^{m-1}\right),\qquad \forall s\in [s_0,\infty],
$$
This gives,
$$
\left(s_0^{m-1}+\frac{m-1}{m} \mu - \frac{m-1}{m}\frac{\vert v \vert^2}{2} \right)^{\frac{1}{m-1}} \leq G_\kappa(v),\qquad \forall v\in \R^d,
$$
and the conclusions of \# Step 2 and \# Step 3 follow similarly choosing $s_0^{m-1}+\frac{m-1}{m} \mu = k$.

For \# Step 4, we work with $f_\kappa$ instead of $f_\kappa^m$. We need the following bounds:
\begin{enumerate}
\item The family $\left(f_\kappa\right)_{\kappa \in (0,1]}$ is bounded in $L^2([0,T]\times\T^d; H^1(\R^d))$.
\item[$(2a)$] When $m>\frac{d}{d+2}$, the family $\left(|v|^2 f_\kappa\right)_{\kappa \in (0,1]} $ is bounded in $L^\infty([0,T]; L^1(\T^d\times\R^d))$,
\item[$(2b)$] When $m>\frac{d-2}{2}$ choose $a\in (0,2)$ such that $\frac{d+a-2}{d+a} <m<\frac{d+a-2}{d}$.

Then the family $\left(\lfloor v\rceil^a f_\kappa\right)_{\kappa \in (0,1]} $ is bounded in $L^\infty([0,T]; L^1(\T^d\times\R^d))$.
\item[$(3)$] The family $\left(\pa_t f_\kappa + v\cdot \na_x f_\kappa\right)_{\kappa \in (0,1]}$ is bounded in $L^2((0,T]\times\T^d; H^{-1}(\R^d)) + L^1 [0,T]\times\T^d\times \R^d)$.
\end{enumerate}
The bound (2a) isn't really necessary here, since the case $m>\frac{d}{d+2}$ is covered by (2b). We will nevertheless prove (2a) separately because its derivation is more straightforward than (2b) and it provides a natural bound on the kinetic energy which is used for the derivation of the entropy estimate.

\noindent{\it Proof of bound (1):} Bound (1) is obtained by multiplying the equation \eqref{eq:kappa2} by $f^{2-m}$ and integrating (as in \eqref{eq:p} to get:
\begin{equation*} 
\begin{split}
&\frac{d}{dt} \iint_{\T^d\times \R^d} \frac{\left\vert f_{\kappa} \right\vert^{3-m}}{3-m}\, \dd x \, \dd v + m(2-m) \iint_{\T^d\times \R^d} \left( \frac{f_\kappa}{f_\kappa+\kappa}\right)^{1-m} \left\vert \na_v f_{\kappa} \right\vert ^2 \, \dd x \, \dd v\\
& \qquad \qquad\qquad \qquad\qquad \qquad = \frac{(2-m)d}{3-m} \iint_{\T^d\times \R^d} \left\vert f_{\kappa}\right\vert ^{3-m} \, \dd x \, \dd v.
\end{split}
\end{equation*}
We then notice that since $m<1$, the term $\left( \frac{f_\kappa}{f_\kappa+\kappa}\right)^{1-m}$ is greater than $\left( \frac{\|f_\kappa\|_\infty}{\|f_\kappa\|_\infty+\kappa}\right)^{1-m}$. We deduce:
\begin{align*}
&\sup_{t\in [0,T]}\iint_{\T^d\times \R^d} \frac{\left|f_{\kappa}(t)\right|^{3-m}}{3-m}\, \dd x \, \dd v + m(2-m)\left( \frac{\|f_{in}\|_\infty}{\|f_{in}\|_\infty+\kappa}\right)^{1-m} \iiint_{(0,T)\times \T^d\times \R^d} \left|\na_v f_{\kappa}\right|^2 \, \dd t \, \dd x \, \dd v \nonumber \\
&\qquad \qquad \qquad \leq e^{(2-m)d T} \iint_{\T^d\times \R^d} \frac{\left|f_{\text{in}}\right|^{3-m}}{3-m}\, \dd x \, \dd v .
\end{align*}
For $\kappa<1$, we have in particular
$$ \iiint_{(0,T)\times \T^d\times \R^d} \left|\na_v f_{\kappa}\right|^2 \, \dd t \, \dd x \, \dd v \leq
\frac{e^{(2-m)d T}}{m(2-m)}\left( \frac{\|f_{in}\|_\infty+1}{\|f_{in}\|_\infty}\right)^{1-m}
\iint_{\T^d\times \R^d} \frac{\left|f_{\text{in}}\right|^{3-m}}{3-m}\, \dd x \, \dd v
$$
which implies the result (recall that $f_\kappa$ is bounded in $L^\infty(0,T;L^1\cap L^\infty(\T^d\times \R^d))$ and therefore in $L^\infty(0,T;L^2(\T^d\times \R^d))$.

\medskip\noindent{\it Proof of bound (2a):}
Next we show
that when $m\in(\frac{d}{d+2},1)$, $\iint_{\T^d\times \R^d} |v|^2 f_\kappa(t) \dd v \dd x$ is bounded. First, we write, using \eqref{eq:kappa2}:
\begin{align*}
\frac{d}{dt} \iint_{\T^d\times \R^d} |v|^2 f_\kappa(t) \dd v \dd x & =
2d\iint_{\T^d\times \R^d} \left( (f_\kappa+\kappa)^m-\kappa^m \right) \dd v\dd x
- 2 \iint_{\T^d\times \R^d} |v|^2 f_\kappa(t) \dd v \dd x
\end{align*}
For $m\in (0,1)$, we have $(s+\kappa)^m-\kappa^m \leq s^m$ for all $s\geq 0$. Indeed, the function $x\mapsto (1+x)^m-x^m$ has a negative derivative and is therefore decreasing in $[0,+\infty)$. This implies that $ (1+x)^m-x^m\leq 1$. Taking $x=\kappa/s$ yields the inequality.
We deduce
\begin{align*}
\frac{d}{dt} \iint_{\T^d\times \R^d} |v|^2 f_\kappa(t) \dd v \dd x & \leq
2d\iint_{\T^d\times \R^d} f_\kappa ^m\dd v\dd x
- 2 \iint_{\T^d\times \R^d} |v|^2 f_\kappa(t) \dd v \dd x
\end{align*}
Finally, Carlson-Levin's inequality \eqref{eq:CL0} implies that for $m\in(\frac d {d+2},1)$, we have
$$
\int_{ \R^d} f ^m \dd v
\lesssim \left(\int_{ \R^d} f \dd v \right)^{m-\frac{d(1-m)}{2}} \left( \int_{ \R^d} |v|^2 f \dd v \right)^{\frac{d(1-m)}{2}}.
$$
Recalling that $\frac{m-1}{k-1} = 1 - \frac{d(1-m)}{2}$, we deduce:
\begin{align*}
\iint_{ \T^d\times \R^d} f ^m \dd v \dd x
& \lesssim \left( \int_{ \T^d} \rho^k \dd x \right)^{1-\frac{d(1-m)}{2}}\left( \iint_{\T^d\times \R^d} |v|^2 f \dd v\dd x \right)^{\frac{d(1-m)}{2}}\\
& \lesssim \int_{ \T^d} \rho^k \dd x + \frac 1 d \iint_{\T^d\times \R^d} |v|^2 f \dd v\dd x
\end{align*}
(the second inequality follows from Young's inequality which is valid here because $\frac{d(1-m)}{2}<1$ when $m>\frac{d-2}{d}$).

We deduce
$$\frac{d}{dt} \iint_{\T^d\times \R^d} |v|^2 f_\kappa(t) \dd v \dd x \lesssim \int_{ \T^d} \rho_\kappa^k \dd x . $$
Since $k<1$ and $\T^d$ is bounded, we have
$$\int_{ \T^d} \rho_\kappa^k \dd x \leq \left(\int_{ \T^d} \rho_\kappa \dd x \right)^{k}|\T^d|^{1-k} = \left(\iint_{ \T^d\times\R^d} f_{in} \dd x \dd v \right)^{k}|\T^d|^{1-k} .$$
The bound on $ \iint_{\T^d\times \R^d} |v|^2 f_\kappa(t) \dd v \dd x$ follows (crucially, this bounds uses the boundedness of $\T^d$ and will need to be adapted when $\Omega=\R^d$).

\medskip\noindent{\it Proof of bound (2b):}
When $m\leq \frac{d}{d+2}$, the argument above fails, but we can still control some moments in $v$ as long as $m>\frac{d-2}{2}$ by using the generalized Carlson-Levin inequality \eqref{eq:CL}.
Multiplying \eqref{eq:kappa} by $\lfloor v\rceil^a$ and integrating (recall that $f_{\kappa}$ satisfies the decay estimate \eqref{eq:maxpsi}),
we get:
\begin{align*}
&\frac{d}{dt} \iint_{\T^d\times \R^d} \lfloor v\rceil^a f_{\kappa} \, \dd x \, \dd v \\
&= \iint_{\T^d\times \R^d} \Delta \left( \lfloor v\rceil^a \right) \, \left( \left( f_{\kappa} + \kappa \right)^m - \kappa^m \right) \, \dd x \, \dd v -\iint_{\T^d\times \R^d} a \vert v \vert^2 \lfloor v\rceil^{a-2} f_{\kappa} \, \dd x \, \dd v, \\
&\leq \iint_{\T^d\times \R^d} \lfloor v\rceil^{a-2}\, \left\vert f_{\kappa} \right\vert^m \, \dd x \, \dd v -\iint_{\T^d\times \R^d} a \vert v \vert^2 \lfloor v\rceil^{a-2} f_{\kappa} \, \dd x \, \dd v,\\
&\leq \iint_{\T^d\times \R^d} \lfloor v\rceil^{a-2}\, \left\vert f_{\kappa}^{(k)}\right\vert^m \, \dd x \, \dd v - a \iint_{\T^d\times \R^d} \lfloor v\rceil^{a} f_{\kappa} \, \dd x \, \dd v +\iint_{\T^d\times \R^d} a \lfloor v\rceil^{a-2} f_{\kappa} \, \dd x \, \dd v\\
&\leq \iint_{\T^d\times \R^d} \lfloor v\rceil^{a-2}\, \left\vert f_{\kappa} \right\vert^m \, \dd x \, \dd v - a \iint_{\T^d\times \R^d} \lfloor v\rceil^{a} f_{\kappa} \, \dd x \, \dd v + a \iint_{\T^d\times \R^d} f_{\kappa} \, \dd x \, \dd v,
\end{align*}
since $\Delta_v \left(\lfloor v\rceil^{a}\right) = a \lfloor v\rceil^{a-4} \left( d + (d+a-2) \vert v \vert^2\right) \geq 0$.

Since $a\in(0,2)$, \Cref{prop:CL} implies
\begin{align*}
\int_{\R^d} \lfloor v\rceil^{a-2} f(v)^m \, \dd v
& \lesssim \left( \int_{\R^d} f(v) \, \dd v\right)^{\frac{a(m-1)+2-d(1-m)}{a}} \left( \int_{\R^d} \lfloor v\rceil^a f(v) \, \dd v\right)^{\frac{a-2+d(1-m)}{a}}
\end{align*}
and so
$$ \int_{\T^d\times \R^d} \lfloor v\rceil^{a-2} f(v)^m \, \dd v \dd x
\lesssim \left(\int_{\T^d} \rho^{\frac{a(m-1)+2-d(1-m)}{2-d(1-m)}} \dd x\right)^{\frac{2-d(1-m)}{a}} \left( \iint_{\T^d \times \R^d} \lfloor v\rceil^a f(v) \, \dd v\dd x\right)^{\frac{a-2+d(1-m)}{a}}
$$
where
$$\frac{a(m-1)+2-d(1-m)}{2-d(1-m)} = 1 - \frac{a(1-m)}{2-d(1-m)} <1$$
and we can now end the proof similarly as for the bound (2a).

\medskip\noindent{\it Proof of bound (3):} Finally, to prove the bound (3), we need to show that $\na_v (f_\kappa+\kappa)^m$ is bounded in $L^2$ and that $vf_\kappa$ is bounded in $L^2$.

For the first one, we consider the function
$$\psi_\kappa(s) = \frac{1}{m+1}\left( (s+\kappa)^{m+1} - \kappa^{m+1} - (m+1) \kappa^m s\right)$$
and we note that\footnote{This follows from the fact that $\psi_\kappa(s)\leq \frac{(2s)^{m+1}}{m+1}$ when $s\geq \kappa$ and $\psi_\kappa''(s)\leq m(2\kappa)^{m-1}$ when $s\leq\kappa$}
\begin{equation}\label{eq:psikappa}
0\leq \psi_\kappa(s)\leq \frac{2^{m+1}}{m+1} s^{m+1} + m\kappa^{m}s \qquad \forall s\geq 0.
\end{equation}
It follows that the function $\psi_{\kappa}(f_\kappa)$ is integrable and using \eqref{eq:kappa2}, we find:
\begin{align}
\frac{d}{dt} \iint_{\T^d\times \R^d} \psi_\kappa(f_\kappa)\, \dd x \, \dd v
&= \iint_{\T^d\times \R^d} (f_\kappa +\kappa)^m \pa_t f_\kappa \, \dd x \, \dd v \nonumber \\
& =- \iint_{\T^d\times \R^d} \left\vert \na_v (f_{\kappa}+\kappa)^{m} \right\vert ^2 \, \dd x \, \dd v- m \iint_{\T^d\times \R^d} (f+\kappa)^{m-1} f_\kappa \na_vf_\kappa \cdot v \, \dd x \, \dd v.\label{eq:bhkj}
\end{align}
We now introduce the function $\phi_\kappa$ such that
$$ \phi_\kappa'(s) = m s (s+\kappa)^{m-1}, \qquad \phi_\kappa(0)=0,$$
so that the last term in \eqref{eq:bhkj} can be written in the form
$$
- m \iint_{\T^d\times \R^d} (f+\kappa)^{m-1} f_\kappa \na_vf_\kappa \cdot v \, \dd x \, \dd v = - \iint_{\T^d\times \R^d} \phi_\kappa'( f_\kappa) \na_vf_\kappa \cdot v \, \dd x \, \dd v
= d \iint_{\T^d\times \R^d} \phi_\kappa( f_\kappa) \, \dd x \, \dd v.
$$
We can check that $(\psi_\kappa-\phi_\kappa)(0)=0$, $(\psi_\kappa-\phi_\kappa)'(0)=0$ and $(\psi_\kappa-\phi_\kappa)''(s)=m(1-m)(s+\kappa)^{m-2}s\geq 0$ for $s\geq 0$, and therefore $\phi_\kappa(s)\leq \psi_\kappa(s)$ for all $s\geq 0$. We can now deduce from \eqref{eq:bhkj} that
\begin{align*}
\frac{d}{dt} \iint_{\T^d\times \R^d} \psi_\kappa(f_\kappa)\, \dd x \, \dd v + \iint_{\T^d\times \R^d} \left\vert \na_v (f_{\kappa}+\kappa)^{m} \right\vert ^2 \, \dd x \, \dd v \leq d \iint_{\T^d\times \R^d} \psi_\kappa(f_\kappa)\, \dd x \, \dd v.
\end{align*}
This implies in particular that
\begin{align*}
\int_0^T\iint_{\T^d\times \R^d} \left\vert \na_v (f_{\kappa}+\kappa)^{m} \right\vert ^2 \, \dd x \, \dd v &\leq e^{dT} \iint_{\T^d\times \R^d} \psi_\kappa(f_{in})\, \dd x \, \dd v \\
&\leq e^{dT} \iint_{\T^d\times \R^d}
\frac{2^{m+1}}{m+1} f_{in}^{m+1} + m\kappa^{m}f_{in}\, \dd x \, \dd v.
\end{align*}
Since the right hand side is bounded uniformly in $\kappa$, the results follows.

In order to show the bound of $vf_\kappa$ in $L^2$, we write, using \eqref{eq:kappa2},
\begin{align*}
\frac{d}{dt} \iint_{\T^d\times \R^d} \vert v\vert^2 f_{\kappa}^{2} \, \dd x \, \dd v &= \iint_{\T^d\times \R^d} 2\vert v\vert^2 \left(\Delta_v \left(f_\kappa+\kappa\right)^{m} \right) f_{\kappa} \, \dd x \, \dd v + \iint_{\T^d\times \R^d} 2\vert v\vert^2 \text{div}_v (vf_\kappa) f_{\kappa} \, \dd x \, \dd v
\end{align*}
We then compute:
\begin{align*}
&\iint_{\T^d\times \R^d} 2\vert v\vert^2 \left(\Delta_v \left(f_\kappa+\kappa\right)^{m} \right) f_{\kappa} \, \dd x \, \dd v \\
&= \iint_{\T^d\times \R^d} 2\vert v\vert^2 \left(\Delta_v \left(f_\kappa+\kappa\right)^{m} \right) (f_{\kappa}+\kappa) \, \dd x \, \dd v - \iint_{\T^d\times \R^d} 2 \kappa \vert v\vert^2 \left(\Delta_v \left(f_\kappa+\kappa\right)^{m} \right) \, \dd x \, \dd v \\
&= -\iint_{\T^d\times \R^d} 2\left(\nabla_v \left(f_\kappa+\kappa\right)^{m} \right) \cdot \nabla_v \left( \vert v\vert^2 (f_{\kappa}+\kappa) \right) \, \dd x \, \dd v - \iint_{\T^d\times \R^d} 4 \kappa d \left(\left(f_\kappa+\kappa\right)^{m} - \kappa^m \right) \, \dd x \, \dd v \\
&= -\iint_{\T^d\times \R^d} 2m\vert v\vert^2 \left(f_\kappa+\kappa\right)^{m-1}\left\vert \nabla_v (f_{\kappa}+\kappa) \right\vert^2 \, \dd x \, \dd v
-\iint_{\T^d\times \R^d} 4 v \cdot \nabla_v \left(f_\kappa+\kappa\right)^{m} (f_{\kappa}+\kappa)\, \dd x \, \dd v \\
&\qquad \qquad \qquad \qquad - \iint_{\T^d\times \R^d} 4 \kappa d \left(\left(f_\kappa+\kappa\right)^{m} - \kappa^m \right) \, \dd x \, \dd v\\
&= -\iint_{\T^d\times \R^d} 2m\vert v\vert^2 \left(f_\kappa+\kappa\right)^{m-1}\left\vert \nabla_v (f_{\kappa}+\kappa) \right\vert^2 \, \dd x \, \dd v
+\iint_{\T^d\times \R^d} \frac{4 d m}{m+1} \left( \left(f_\kappa+\kappa\right)^{m+1} - \kappa^{m+1} \right)\, \dd x \, \dd v \\
&\qquad \qquad \qquad \qquad - \iint_{\T^d\times \R^d} 4 \kappa d \left(\left(f_\kappa+\kappa\right)^{m} - \kappa^m \right) \, \dd x \, \dd v\\
& \leq \iint_{\T^d\times \R^d} \frac{4 d m}{m+1} \left( \left(f_\kappa+\kappa\right)^{m+1} - \kappa^{m+1} \right)\, \dd x \, \dd v.
\end{align*}
Similarly,
\begin{align*}
&\iint_{\T^d\times \R^d} 2\vert v\vert^2 \text{div}_v (vf_\kappa) f_{\kappa} \, \dd x \, \dd v \\
&= \iint_{\T^d\times \R^d} 2\vert v\vert^2 \left(v \cdot \nabla_v f_{\kappa} + d f_{\kappa} \right) f_{\kappa}\, \dd x \, \dd v \\
&=\iint_{\T^d\times \R^d} \left(2 d \vert v\vert^2- 2 \vert v\vert^2 -d \vert v\vert^2 \right) f_{\kappa}^2\, \dd x \, \dd v = \iint_{\T^d\times \R^d} \left(d-2 \right) \vert v\vert^2 f_{\kappa}^2\, \dd x \, \dd v .
\end{align*}
We deduce:
$$
\frac{d}{dt} \iint_{\T^d\times \R^d} \vert v\vert^2 f_{\kappa}^{2} \, \dd x \, \dd v
\leq\frac{4 d m}{m+1}
\iint_{\T^d\times \R^d} \left( \left(f_\kappa+\kappa\right)^{m+1} - \kappa^{m+1} \right)\, \dd x \, \dd v +\left(d-2 \right) \iint_{\T^d\times \R^d} \vert v\vert^2 f_{\kappa}^2\, \dd x \, \dd v .
$$
The inequality \eqref{eq:psikappa} implies
$$
0\leq (s+\kappa)^{m+1} - \kappa^{m+1} \leq (m+1) \kappa^m s + 2^{m+1}s^{m+1}
$$
hence
$$
\frac{d}{dt} \iint_{\T^d\times \R^d} \vert v\vert^2 f_{\kappa}^{2} \, \dd x \, \dd v
\leq C
\iint_{\T^d\times \R^d} \kappa^m f_\kappa + f_\kappa^{m+1}\, \dd x \, \dd v +\left(d-2 \right) \iint_{\T^d\times \R^d} \vert v\vert^2 f_{\kappa}^2\, \dd x \, \dd v .
$$
Since the $L^1$ and $L^2$ Lebesgue norms of $f_\kappa$ are controlled, we obtain
$$
\sup_{t\in[0,T]} \iint_{\T^d\times \R^d} \vert v\vert^2 f_{\kappa}^{2}(t,x,v) \, \dd x \, \dd v \leq C(T) \qquad \forall \kappa\leq 1
$$
for some constant depending on $m$, $d$, $T$ and $f_{in}$ but not on $\kappa$ (this bound does not depend on $|\T^d|$ and is thus valid when $\Omega =\R^d$ as well).

\subsection{Existence of weak solutions when \texorpdfstring{$\Omega=\R^d$ ($m>1$)}{Omega=R (m>1)}} \label{sec:oR}

We have so far focused on the case $\Omega=\T^d$. When $\Omega=\R^d$, the proofs go through similarly, provided we can control some moments of the solution with respect to $x$ in order to get compactness results in $\R^d$.
In particular, the result of \Cref{lem:fixptT} still holds, and in addition to the bounds already listed there, we can show that the solution $\tilde{g}(t,x,v)$ of \eqref{eq:fTg} satisfies
\begin{equation}\label{eq:xx2}
\sup_{t\in [0,T]} \int_{ \T^d\times \R^d} |x|^2\tilde{g}(t,x,v)^2\, \dd x \, \dd v \leq \left(\int_{ \T^d\times \R^d} (|x|^2+|v|^2)f_{\text{in}}(x,v)^2\, \dd x \, \dd v \right) \, \exp({C(\kappa,M,d) T}).
\end{equation}

Indeed, we have
\begin{align*}
\frac {d}{dt} \frac 1 2 \iint_{\R^d\times \R^d } |x|^2 \tilde{g}^2\, \dd x\, \dd v
& = \iint_{\R^d\times \R^d } (x\cdot v) \, \tilde{g}^2\, \dd x \,dv
- \iint_{\R^d\times \R^d } |x|^2 \Phi(g) |\na_v \tilde{g}|^2\, \dd x\, \dd v\\
& \qquad \qquad \qquad + \frac{d}{2} \iint_{\R^d\times \R^d } |x|^2 \tilde{g}^2\, \dd x\,dv\\
& \leq \frac{d+1}{2} \iint_{\R^d\times \R^d } |x|^2 \tilde{g}^2\, \dd x \,dv
+ \frac 1 2 \iint_{\R^d\times \R^d } |v|^2 \tilde{g}^2\, \dd x\, \dd v.
\end{align*}
Using \eqref{eq:vv2}, we deduce \eqref{eq:xx2}.

With the same notations as in the proof of Proposition \ref{prop:regular}, it follows that $|x| \, \tilde{g}_k $ is bounded in $L^\infty([0,T]; L^2(\R^d\times\R^d))$,
which is all that we need to carry out the proof of Proposition \ref{prop:regular} in $\R^d$.
Indeed, we can proceed as before to show that for all $R>0$, there exists a subsequence along which
$$\tilde{g}_k \to \g, \qquad \mbox{ strongly in } L^1((0,T)\times B_R\times B_R)$$
and use the estimates the fact that $(|x|+|v|) \tilde{g}_k $ is bounded in $L^\infty([0,T]; L^2(\R^d\times\R^d))$ to deduce that
$$\lim_{k\to\infty} \| \tilde{g}_k - \g\|_{L^2([0,T]\times \R^d \times \R^d)}^2 =0.$$

We proceed similarly with the limit $\kappa\to0$: $f_\kappa(t,x,v)$ satisfies
\begin{align*}
\frac{d}{dt}\iint_{\R^d\times \R^d } |x|^2 f_\kappa^m\, \dd x\, \dd v
& = 2 \iint_{\R^d\times \R^d } (x\cdot v) \, f_\kappa^m\, \dd x \,dv \\
& \leq \iint_{\R^d\times \R^d } |x|^2 \, f_\kappa^m\, \dd x \,dv + \iint_{\R^d\times \R^d } |v|^2 \, f_\kappa^m\, \dd x \,dv
\end{align*}
We then use \eqref{eq:momvkappa} and \eqref{eq:xvfin} to get the bound
$$\sup_{t\in [0,T]}\iint_{\R^d\times \R^d } |x|^2 f_\kappa^m(t,x,v)\, \dd x\, \dd v \leq C(T,f_{in})
$$
As above, this bound allows us to carry out the limit $\kappa\to0$ as in the case $\Omega = \T^d$.

\subsection{The case \texorpdfstring{$m<1$}{m<1} and \texorpdfstring{$\Omega = \R^d$}{Rd}}

\Cref{lem:fixptT} and \Cref{prop:regular} are still true in this case after doing the modification explained in the previous subsection. When $\Omega =\R^d$, the last step $\#$ Step $4$ should be handled by applying Carlson-Levin's inequality with respect to the $x$ variable, which requires controlling the moment $|x|^2 f_\kappa$ and the restriction $k>\frac{d}{d+2}$ that is $m>\frac{2d}{2d+2}$.

The key estimate is the following:
\begin{align*}
\frac{d}{dt}\iint_{\R^d\times \R^d}
\frac 1 2 (|v|^2+|x|^2) f(t,x,v) \dd v
& = \iint_{\R^d\times \R^d} x\cdot v f \dd v
+ d \iint_{\R^d\times \R^d} f^m \dd v - \iint_{\R^d\times \R^d} |v|^2 f \dd v\\
& \leq \frac 1 2 \iint_{\R^d\times \R^d} (|v|^2+|x|^2) f \dd v
+ d \iint_{\R^d\times \R^d} f^m \dd v
\end{align*}
We can then conclude using Carlson-Levin's inequality in $\R^{2d}$: As long as $m\in(\frac {2d} {2d+2},1)$, we have
$$
\iint_{\R^d\times \R^d} f ^m \dd v \dd x
\leq C\left(\iint_{\R^d\times \R^d} f \dd v \dd x \right)^{m-2d(1-m)} \left( \iint_{\R^d\times \R^d} (|v|^2 +|x|^2) f \dd v \dd x\right)^{2d(1-m)}.
$$
More generally, when $m\in(\frac{d-1}{d},1)$ for any $a \in (0,1)$,
\begin{align*}
\frac{d}{dt}\iint_{\R^d\times \R^d}
\varphi^a f(t,x,v) \,\dd x \dd v
& = \iint_{\R^d\times \R^d} \left( a(a-1) \vert \nabla_v \varphi \vert^2 \varphi^{a-2} + a \varphi^{a-1} \Delta \varphi\right) f^m \,\dd x \dd v \\
&\qquad \qquad + \iint_{\R^d\times \R^d} a \varphi^{a-1} \left( v \cdot \nabla_x \varphi - v \cdot \nabla_v \varphi\right) f \,\dd x \dd v,
\end{align*}
with $\varphi$ prescribed as
$\varphi(x,v) = 1+ \vert v \vert^2 + 2 \beta v \cdot x + \beta \vert x \vert^2 = 1+ (v+\beta x)^2 + \beta(1-\beta) \vert x \vert^2$ with $\beta < 1$ ($v \cdot \nabla_x \varphi - v \cdot \nabla_v \varphi = 2 (\beta-1) \vert v \vert^2 $). We get
\begin{align*}
\frac{d}{dt}\iint_{\R^d\times \R^d}
\varphi^a f(t,x,v) \,\dd x \dd v
&\lesssim \iint_{\R^d\times \R^d} \varphi^{a-1} f^m \,\dd x \dd v + 2(\beta - 1) a\iint_{\R^d\times \R^d} \varphi^{a-1} \vert v \vert^2 f \,\dd x \dd v\\
& \leq \iint_{\R^d\times \R^d} \varphi^{a-1} f^m \,\dd x \dd v,
\end{align*}
and we conclude similarly as previously, recalling \eqref{eq:CL}.

\subsection{Uniqueness of the solution}
The uniqueness part of the theorem follows from the following proposition which we prove in this section:

\begin{proposition}[Uniqueness]\label{prop:unique}
Let $\Omega = \T^d$ or $\R^d$. Given $f_{\text{in}} \in L^1\cap L^\infty(\Omega\times \R^d)$ and $T>0$, equation \eqref{eq:FP} has at most one solution
$f\in L^\infty((0,T);L^1\cap L^\infty(\Omega\times \R^d))$ such that $\na_v f^{m} \in L^2((0,T)\times \Omega\times\R^d)$.
\end{proposition}

\begin{proof}[{\bf Proof of Proposition \ref{prop:unique}}]
Let $f_1$ and $f_2$ be two solutions and set $g_i = f_i^m$.
The difference $ h =f_1 -f_2$ solves
$$
\begin{cases}
\pa_t h + v\cdot\na_x h -\text{div}_v (v h) = \Delta_v (g_1-g_2) & \mbox{ in } (0,T)\times \Omega\times \R^d\\[5pt]
h (0,x,v) = 0 & \mbox{ in } \Omega\times \R^d
\end{cases}
$$
in the following sense:
$$
\int_0^T \iint_{\Omega \times \R^d}{ h \left( -\pa_t \psi - v\cdot \na_x \psi +v\cdot\na_v\psi\right) + \na_v (g_1-g_2)\cdot\na_v\psi }\,dt= 0,
$$
for test functions $\psi \in \mathcal{D}([0,T))$ (thus satisfying $\psi(T,x,v)=0$).

Define $\psi$, solution to the backward transport equation
$$
\begin{cases}
- \pa_t \psi - v\cdot\na _x \psi + v\cdot \na_v \psi = e^{-\lambda t} (g_1-g_2)_k & \mbox{ in } (0,T)\times\Omega\times\R^d,\\
\psi(T,x,v) = 0 & \mbox{ in } \Omega\times\R^d,
\end{cases}
$$
where $(g_1-g_2)_k$ is a regularisation of $g_1-g_2$ in all variables. The function $\psi$ is thus a acceptable test function. We get,
$$
\int_0^T \iint_{\Omega \times \R^d}{ e^{-\lambda t} h (g_1-g_2)_k - e^{\lambda t} \na_v \left[ \pa_t \psi + v\cdot\na _x \psi - v\cdot \na_v \psi \right] \cdot\na_v\psi }\,dt= 0.
$$
Furthermore,
\begin{align*}
& \int_0^T \iint_{\Omega \times \R^d}{ e^{\lambda t} \na_v \left[ \pa_t \psi + v\cdot\na _x \psi - v\cdot \na_v \psi\right] \cdot\na_v\psi
}\,dt\\
& = \iint_{\Omega \times \R^d}{ e^{\lambda T}\frac{ |\na_v\psi|^2(T,x,v)}{2} - \frac{|\na_v\psi|^2(0,x,v)}{2} }-\int_0^T \iint_{\Omega \times \R^d}{ \lambda e^{\lambda t}\frac{|\na _v\psi|^2
}{2}}\,dt \\
& \qquad +\left( d-2\right) \int_0^T \iint_{\Omega \times \R^d}{ e^{\lambda t}\frac{|\na _v\psi|^2 }{2}
}\,dt\\
& = -\iint_{\Omega \times \R^d}{ \frac{|\na_v\psi|^2(0,x,v)}{2} } - \left(\lambda +2 - d\right) \int_0^T \iint_{\Omega \times \R^d}{ e^{\lambda t} \frac{|\na _v\psi|^2 }{2}
}\,dt
\end{align*}
We now take $\lambda = d-2$, and we deduce
$$
\int_0^T \iint_{\Omega \times \R^d}{ e^{-\lambda t} h (g_1-g_2)_k}\,dt+\iint_{\Omega \times \R^d}{ \frac{|\na_v\psi|^2(0,x,v)}{2} }= 0,
$$
and so
$$
\int_0^T \iint_{\Omega \times \R^d}{ e^{-\lambda t} h (g_1-g_2)_k}\,dt \leq 0.
$$
We can now pass to the limit $k\to \infty$ so get
$$
\int_0^T \iint_{\Omega \times \R^d}{ e^{-\lambda t} h (g_1-g_2)}\,dt \leq 0.
$$
Since $h (g_1-g_2) = (f_1-f_2)(f_1^m-f_2^m)\geq 0$ a.e., we deduce that $f_1=f_2$.
\end{proof}

\section{Convergence: compactness method}

We start with a detailed proof in the case $m>1$ and with $\Omega=\T^d$ or $\Omega=\R^d$. From now on, recall that $f_\eps$ the (unique) solution to \eqref{eq:mainVFP} on $[0,T]$ given by Theorem \ref{thm:existence}. A key role in the proof is played by the entropy inequality which we recall here:
\begin{equation}\label{eq:entropyeps}
\mathcal E[f_\eps(t)] + \frac 1 {\eps^2} \int_0^t D[f_\eps(s)]\, \dd s \leq \mathcal E[f_{\text{in}}]
\end{equation}
where
$$\mathcal E[f] = \iint_{\Omega\times\R^d}{\frac{f^m}{m-1} +\frac{|v|^2}{2} f \dd x \dd v}
$$
and
$$
D[f] (t)= \iint_{\Omega\times\R^d} f \left| v + \frac{m}{m-1}\na_v f^{m-1}\right|^2 \, \dd x \, \dd v.$$

\subsection{The case \texorpdfstring{$m>1$}{m>1}.}

\subsubsection{A priori estimates}
We gather in the proposition below the a priori estimates on $f_\eps$ that will be used throughout the proof:
\begin{proposition}\label{prop:apriori}
Recall that $f_{\text{in}}(x,v)$ satisfies \eqref{eq:initieps}.
Then $f_\eps(t,x,v)$ satisfies, for all $t >0$,
\begin{align}
& \| f_\eps(t)\|_{L^1(\Omega\times\R^d)} = \|\rho_\eps(t)\|_{L^1(\Omega)} = \| f_{\text{in}}\|_{L^1(\Omega\times\R^d)}, \label{eq:L1bd} \\
& \|f_\eps(t)\|_{L^m(\Omega\times\R^d)}\leq C,\label{eq:Lmbd} \\
& \iint_{\Omega\times \R^d} |v|^2 f_\eps(t,x,v)\dd v\dd x \leq C,\label{eq:mombd} \\
& \|\rho_\eps(t)\|_{L^k(\Omega)}\leq C, \label{eq:Lkbd}\\
& \int_0^\infty D (f_\eps(t))\, dt \leq C \eps^2, \label{eq:dissbd}
\end{align}
where $C=\E[f_{\text{in}}]$ is independent of $\eps$. Furthermore, \eqref{eq:finB} implies
\begin{equation}\label{eq:Linftyeps}
\| f_\eps\|_{L^\infty((0,\infty)\times \Omega \times\R^d)} \leq C, \qquad 
\| \rho_\eps\|_{L^\infty((0,\infty)\times \Omega)} \leq C
\end{equation}
\end{proposition}
\begin{proof}[{\bf Proof of \Cref{prop:apriori}}]
The first equality is the mass conservation property. The next four inequalities follow from the entropy inequality \eqref{eq:entropyeps} and the fact that $\E[f]\geq 0$ when $m>1$.

Finally, \eqref{eq:finB} and the comparison principle implies that $f_\eps(t,x,v)\leq G[M](v)$ for all $t>0$ which yields \eqref{eq:Linftyeps}.

\end{proof}

Proposition \ref{prop:apriori} provides bounds on the density $\rho_\eps$.
The flux $j_\eps=\frac 1 \eps \int_{\R^d} v f_\eps(\cdot,\cdot,v)\, \dd v$ also plays a key role in the proof. It satisfies:

\begin{lemma}\label{lem:j}
The flux $j_\eps(t,x)$ is bounded in $L^2((0,T);L^1(\Omega)) $.
\end{lemma}

\begin{proof}[{\bf Proof of \Cref{lem:j}}]
We write, for $(t,x) \in (0,T) \times \Omega$,
$$j_\eps(t,x) = \frac 1 \eps \irdv{vf_\eps(t,x,v)}
= \frac 1 \eps \irdv{f_\eps(t,x,v)\left[v + \frac{m}{m-1} \na_v f_\eps^{m-1}(t,x,v)\right]}
$$
We deduce
\begin{equation}\label{eq:jepsbd}
|j_\eps| \leq \left( \frac 1 {\eps^2}\irdv{f_\eps \left|v + \frac{m}{m-1} \na_v f_\eps^{m-1}\right|^2} \right)^{\frac12}\rho_\eps^{\frac12}
\end{equation}
and so
$$
\int_\Omega |j_\eps(t,x)|\, \dd x \leq
\left( \int_\Omega \rho_\eps(t,x) \dd x\right)^{\frac12}
\left( \frac 1 {\eps^2}D(f_\eps(t))\right)^{\frac12}.
$$
We deduce from \eqref{eq:L1bd} and \eqref{eq:dissbd} that $j_\eps$ is bounded in $L^2((0,T);L^1(\Omega)) $.
\end{proof}

Finally, this lemma implies some control on the moment of $f_\eps$ in $x$ which are required when $\Omega=\R^d$:
\begin{lemma}\label{lem:momentdelta}
When $\Omega=\R^d$,
there exists a constant $C$, independent of $\eps$ but depending on $T$ and the initial condition such that, for any $\delta \leq 1$,
\begin{equation}\label{eq:momentxeps}
\iint_{\R^d\times\R^d} \lfloor x \rceil ^\delta f_\eps(t,x,v) \dd v \dd x \leq C \qquad\forall t\in [0,T],
\end{equation}
and
$$\int_{\R^d} \lfloor x \rceil ^\delta \rho_\eps(t,x)\dd x \leq C\qquad\forall t\in [0,T].$$
\end{lemma}
\begin{proof}[{\bf Proof of \Cref{lem:momentdelta}}]
We compute
\begin{align*}
\frac{d}{dt}\iint_{\R^d\times\R^d} \lfloor x \rceil ^\delta f_\eps(t,x,v) \dd v \dd x
& = \frac 1 \eps \iint_{\R^d\times\R^d} \delta \lfloor x \rceil ^{\delta-2} x \cdot v f_\eps(t,x,v) \dd v \dd x\\
& = \int_{\R^d} \delta \lfloor x \rceil ^{\delta-2} x \cdot j_\eps(t,x)\dd x\\
& \leq \int_{\R^d}| j_\eps(t,x)|\dd x
\end{align*}
since $|\delta \lfloor x \rceil ^{\delta-2} x| \leq \delta \lfloor x \rceil ^{\delta-1}\leq 1$ when $\delta\leq 1$.
The result then follows from Lemma \ref{lem:j}.
\end{proof}

\subsubsection{Strong convergence of the density}
\begin{proposition}\label{prop:rho}
The family $(\rho_\eps)_{\eps \in (0,1]}$ is relatively compact in $L^1((0,T)\times\Omega)$:
There exists a function $\rho \in L^1((0,T)\times\Omega)$ such that, up to a subsequence, $\rho_\eps$ converges to $\rho$ strongly in $L^1((0,T)\times\Omega)$ and almost everywhere.

Furthermore, $\rho \in C^{1/2}(0,T;W^{-1,1}(\Omega))$ and $\rho(0,x)=\int_{\R^d} f_{in}(x,v)\dd v$.
\end{proposition}

\begin{proof}[{\bf Proof of \Cref{prop:rho}}]
We can write the kinetic equation in the form
\begin{equation}\label{eq:fg}
\eps \pa_ t f_\eps + v\cdot \na_x f_\eps = \div_v g_\eps
\end{equation}
where
$$ g_\eps = \frac{1}{\eps} \left[v+\frac{m}{m-1}\na_v f_\eps^{m-1}\right] f_\eps $$
satisfies
$$
\iint_{\Omega\times\R^d }|g_\eps(t,x,v) | \dd x \dd v \leq
\left( \iint_{\Omega\times\R^d }
f_\eps \dd v \dd x\right)^{\frac{1}{2}}\left(
\frac{1}{\eps^2}
\iint_{\Omega\times\R^d } f_\eps\left| v+\frac{m}{m-1}\na_v f_\eps^{m-1}\right|^2 \dd v \dd x \right)^{\frac{1}{2}}
$$
so \eqref{eq:L1bd} and \eqref{eq:dissbd} imply that
\begin{equation}\label{eq:gbounded}
g_\eps(t,x,v)
\mbox{ is bounded in } L^2((0,T);L^1(\Omega\times\R^d)).
\end{equation}

We can then proceed as in \cite[Proposition 6.1]{ElGhaniMasmoudi2010} to show that $(\rho_\eps)_{\eps \in (0,1]}$ is relatively compact in $L^1((0,T)\times\Omega)$. We recall the main step of this argument for the sake of completeness:

First, we recall (see \eqref{eq:macro}) that $\pa_t \rho_\eps =- \text{div}_x j_\eps = 0$
and so Lemma \ref{lem:j} implies
\begin{equation}\label{eq:patrho} \pa_t \rho_\eps \mbox{ is bounded in } L^2((0,T);W^{-1,1}(\Omega)) .
\end{equation}
Compactness with respect to $x$ will be a consequence of the following velocity averaging lemma, which proof can be found in \cite{ElGhaniMasmoudi2010,golse2025velocity}:
\begin{lemma}\label{lem:ave}
Assume that $(h_\eps)_{\eps \in (0,1]}$ is bounded in $L^2((0,T)\times \Omega\times\R^d)$ and let $(h_{0,\eps})_{\eps \in (0,1]}$ and $(h_{1,\eps})_{\eps \in (0,1]}$ be bounded in $L^1((0,T)\times \Omega\times\R^d)$ such that
\begin{equation}\label{eq:fh}
\eps\pa_t h_\eps + v\cdot\na _x h_\eps = h_{0,\eps}+ \na_v\cdot h_{1,\eps}.
\end{equation}
Then for all $\psi\in C^\infty_0(\R^d)$, we have
$$\lim_{z\to 0} \quad\left\| \int_{\R^d} (h_\eps(\cdot,\cdot+z,v)-h_\eps(\cdot,\cdot,v)) \psi(v)\, \dd v \right\|_{L^1((0,T)\times\Omega)} = 0,$$
uniformly in $\eps$.
\end{lemma}

The $L^1$ and $L^\infty$ bounds \eqref{eq:L1bd} and \eqref{eq:Linftyeps} imply that
$f_\eps$ is bounded in $ L^2((0,T)\times \Omega\times\R^d)$ and in view of \eqref{eq:fg}, we can use Lemma \ref{lem:ave} to get
$$ \lim_{z \to 0 }\left\| \int_{\R^d} (f_\eps(t,\cdot+z,\cdot )-f_\eps(t,\cdot,\cdot)) \psi(v)\, \dd v \right\|_{L^1((0,T)\times\Omega)} = 0,$$
uniformly in $\eps$. Furthermore, using the bound on $\iint_{\R^d\times \Omega} |v|^2 f_\eps \dd x \dd v $ and using the same argument as in \# Step 1 in the proof of \Cref{prop:regular}, we can take $\psi\equiv 1$ and get
$$\lim_{z \to 0 } \int_0^T\int_{\R^d}| \rho_\eps(t,x+z)-\rho_\eps(t,x)|\, \dd x\, dt = 0,
$$
uniformly in $\eps$. Together with \eqref{eq:patrho} and \eqref{eq:momentxeps}, this implies Proposition \ref{prop:rho} (in particular \eqref{eq:patrho} implies that $\rho_\eps$ is bounded in $C^{1/2}(0,T;W^{-1,1}(\Omega))$ and that we can pass to the limit in the initial data).

\end{proof}

\subsubsection{Passing to the limit in the equation}
We now have all the ingredients necessary to pass to the limit in the equations \eqref{eq:macro} which we recall here for convenience:
$$
\begin{cases}
\pa_t \rho_\eps + \text{div}_x j_\eps = 0 \\
\eps^2\pa_t j_\eps + \text{div}_x P_\eps = - j_\eps
\end{cases}
$$
with
$$
P_\eps(t,x) = \irdv{v\otimes v \,f_\var(t,x,v)},
$$

In view of Proposition \ref{prop:rho} and Lemma \ref{lem:j} we can extract a subsequence such that $\rho_\eps$ converges strongly in $L^1((0,T)\times\Omega)$ to $\rho$ and $j^\eps$ converges weakly in $L^1((0,T)\times\Omega)$ to $j$. This allows us to pass to the limit in the continuity equation to get:
$$\pa_t \rho + \text{div}_x j = 0 $$
in the sense of distributions.

The main step is thus the following proposition

\begin{proposition}\label{prop:convP}
The following holds:
$$
\lim_{\eps \to 0} P_\eps = \nu (\rho) \, \mathbb I \qquad \mbox{ in } L^1((0,T)\times\Omega)
$$
where $\nu(\rho) = \int_{\R^d} G[\rho](v)\dd v$ (see \eqref{eq:nu}).
\end{proposition}
\begin{proof}[{\bf Proof of \Cref{prop:convP}}]
We start by writing
\begin{align*}
P_\eps & = \irdv{v\otimes [v \,f_\var + \na_v f_\eps^m]} - \irdv{v\otimes \na_v f_\eps^m} \\
& = \irdv{v\otimes [v \,f_\var + \na_v f_\eps^m]} + \mathbb I \irdv{f_\eps^m}.
\end{align*}
The first term satisfies
\begin{multline*}
\irdx{\left|\irdv{v\otimes [v \,f_\var + \na_v f_\eps^m]}\right|}
\\\leq \left(\irdxv{|v|^2 f_\eps}\right)^{\frac{1}{2}} \left(\irdxv{ f_\eps \left| v + \frac{m}{m-1}\na_v f_\eps^{m-1}\right|^2 }\right)^{\frac{1}{2}}
\end{multline*}
and so
$$
\int_0^T \irdx{\left|\irdv{v\otimes [v \,f_\var + \na_v f_\eps^m]}\right|} \leq C\eps.
$$
We thus need to show that $\irdv{f_\eps^m}$ converges to $\nu(\rho)$.
We write:
\begin{align*}
\left\Vert \irdv{f_\eps(v)^m} - \nu(\rho) \right\Vert_{ L^1((0,T)\times\Omega)} &\leq \int_0^T \int_\Omega \left\vert \irdv{f_\eps^m} - \irdv{G[\rho_\eps(s,x)](v)^m} \right\vert \, \dd x \dd s \\ 
&\qquad \qquad + \int_0^T \int_\Omega \left\vert \nu(\rho_\eps) - \nu(\rho) \right\vert \, \dd x \dd s .
\end{align*}
The strong convergence of $\rho_\eps$ in $L^1$, together with the uniform $L^\infty$ bound \eqref{eq:Linftyeps} implies the strong convergence of $\nu(\rho_\eps) = \nu_1 \rho_\eps^k$ to $\nu(\rho)$. So we need to prove that the first term converges to zero, which follows from Corollary \ref{cor:ineqm}.

Indeed, applying \eqref{eq:ineqm} to $v\mapsto f_\eps(t,x,v)$ and using \eqref{eq:dissbd}, we get
\begin{equation*}
\left\Vert \int_{\R^d} f_\eps^m(v) - G[\rho_\eps]^m(v)\, \dd v\right\Vert_{ L^1((0,T)\times\Omega)} \lesssim C \eps
\end{equation*}
and the result follows.
\end{proof}

This proposition implies
$$
\lim_{\eps \to 0} j^\eps = \na_x \, \nu[\rho] \qquad\mbox{ in } \mathcal D'((0,T)\times\Omega).$$
and shows that the limit $\rho$ solves the nonlinear diffusion equation.

Finally, we note that the strong convergence of $\rho^\eps$ to $\rho$ implies that of $G[\rho_\eps]$ to $G[\rho]$, so the following lemma implies the strong convergence of $f_\eps$ to $G[\rho]$:

\begin{lemma}\label{lem:convf2}
We have:
\begin{equation}\label{eq:convfG} \lim_{\eps \to 0} \int_0^T\iint_{\Omega\times\R^d} \left\vert f_\eps-G[\rho_\eps] \right\vert\, \dd x\, \dd v\, dt = 0.
\end{equation}
\end{lemma}

\begin{proof}[{\bf Proof of \Cref{lem:convf2}}]

For clarity, we will not always indicate the dependence on the variable $t$, $x$ and $v$ in the proof below.
In order to use the results of the Appendix (which only involve the variable $v$), we denote
$$H[f_\eps] =\int_{\R^d} \frac{f_\eps^m}{m-1} + \frac{|v|^2}{2} f_\eps\dd v$$
which is a function of $x$ and $t$ for which we have
$$\E_\eps[f_\eps] = \int_\Omega H[f_\eps]\, \dd x
$$
(which is a function of $t$).

First, we classically have (see \cite[Corollary 13]{DPD} and Proposition \ref{prop:entropydissipation})
\begin{equation}\label{eq:GN}
0\leq \E[f_\eps(t)] - \E[G[\rho_\eps(t)]]\leq D_\eps(t) \qquad \mbox{ when } m>1
\end{equation}
and so \eqref{eq:dissbd} implies
\begin{equation}\label{eq:entropeps}\int_0^T \E[f_\eps(t)] - \E[G[\rho_\eps(t)]]\, dt \leq C \eps^2.
\end{equation}

We then use Proposition \ref{prop:f-g}:
for all $\delta>0$ there exists $C_{\delta,m,d}$ such that
$$\left(\int_{\R^d} |f_\eps(t,x,v)-G[\rho_\eps(t,x)](v)|\dd v \right)^k \leq \delta \rho_\eps(t,x)^k + C_{\delta,m,d} (H[f_\eps(t,x)]-H[G[\rho_\eps(t,x)]])$$
for all $x\in \Omega$ and $t\geq 0$.
It follows (using \eqref{eq:Lkbd} and \eqref{eq:entropeps}):
\begin{align*}
\int_0^T \int_{\Omega} \left(\int_{\R^d} |f_\eps(t,x,v)-G[\rho_\eps(t,x)](v)|\dd v \right)^k\dd x \dd t
& \leq \delta \int_0^T \int_{\Omega} \rho_\eps^k\dd x \dd t + C_{\delta,m,d} \int_0^T \E[f_\eps]-\E[G[\rho_\eps]]\, dt\\
& \leq C \delta + C \eps^2.
\end{align*}
We deduce
\begin{align*}
\limsup_{\eps\to 0} \int_0^T \int_{\Omega} \left(\int_{\R^d} |f_\eps(t,x,v)-G[\rho_\eps(t,x)](v)|\dd v \right)^k\dd x \dd t \leq C \delta
\end{align*}
and since this holds for all $\delta>0$, we must have
$$\limsup_{\eps\to 0} \int_0^T \int_{\Omega} \left(\int_{\R^d} |f_\eps(t,x,v)-G[\rho_\eps(t,x)](v)|\dd v \right)^k\dd x \dd t =0. $$
When $\Omega=\T^d$, the result follows (since $k>1$ when $m>1$).
When $\Omega=\R^d$, we get convergence in $L^1_{loc}((0,T)\times\R^d ; L^1(\R^d))$, which together with the bounds \eqref{eq:momentxeps} implies the convergence in $L^1((0,T)\times\R^d ; L^1(\R^d))$.
\end{proof}

\subsection{The case \texorpdfstring{$m<1$}{m<1}.}

\subsubsection{A priori estimates}
When $m<1$, the entropy, which still convex and decreasing in time, can take negative values and we need to use the following proposition:
\begin{proposition}\label{prop:CLbd}
When $\Omega=\T^d$ and $m\in(\frac{d}{d+2},1)$, and given $f$ such that $(1+|v|^2)f \in L^1(\Omega\times\R^d) $,
there exists $C(d,m, \|f\|_{L^1(\Omega\times\R^d}))$ such that
\begin{equation}\label{eq:CLbd1}
\frac{1}{1-m}\iint_{\T^d\times\R^d} f^m\dd v \dd x
\leq \frac 1 4 \iint_{\T^d\times \R^d} |v|^2 f \dd v \dd x + C(d,m, \|f\|_{L^1(\Omega\times\R^d})) 
\end{equation}

When $\Omega =\R^d$ and $m\in(\frac{d}{d+1},1)$,
and given $f$ such that $(1+|v|^2)f \in L^1(\Omega\times\R^d) $,
there exists $C=C(d,m, \|f\|_{L^1(\Omega\times\R^d}))$ such that
\begin{equation}\label{eq:CLbd2}
\frac{1}{1-m}\iint_{\R^d\times\R^d} f^m\dd v \dd x
\leq \frac 1 4 \iint_{\R^d\times \R^d}( |v|^2 +|x|^2 ) f \dd v \dd x + C(d,m, \|f\|_{L^1(\Omega\times\R^d})) 
\end{equation}
\end{proposition}
\begin{proof}[{\bf Proof of \Cref{prop:CLbd}.}]
We write
\begin{align*}
\iint_{\T^d\times\R^d} f^m\dd v \dd x 
& = \iint_{\T^d\times\R^d} \left((1+|v|^2) f\right)^m (1+|v|^2)^{-m}\dd v \dd x \\
& \leq \left( \iint_{\T^d\times\R^d} (1+|v|^2) f \dd v \dd x \right)^{m}
\left( \iint_{\T^d\times\R^d} (1+|v|^2)^{-\frac{m}{1-m}} \dd v \dd x \right)^{1-m}
\end{align*}
where this last integral is finite since the condition $m>\frac{d}{d+2}$ is equivalent to $\frac{m}{1-m} >\frac{d}{2}$ (and $\T^d$ is bounded).
The result follows from Young's inequality.

When $\Omega=\R^d$, we proceed similarly, writing
\begin{align*}
\iint_{\R^d\times\R^d} f^m\dd v \dd x 
& = \iint_{\R^d\times\R^d} \left((1+|v|^2+|x|^2) f\right)^m (1+|v|^2+|x|^2)^{-m}\dd v \dd x \\
& \leq \left( \iint_{\R^d\times\R^d} (1+|v|^2+|x|^2) f \dd v \dd x \right)^{m}
\left( \iint_{\R^d\times\R^d} (1+|v|^2+|x|^2)^{-\frac{m}{1-m}} \dd v \dd x \right)^{1-m}
\end{align*}
and this last integral is finite since the condition $m>\frac{d}{d+1}$ is equivalent to $\frac{m}{1-m} >d$.
\end{proof}

\begin{corollary}\label{cor:CLbd}
When $\Omega =\T^d$ and $m\in(\frac{d}{d+2},1)$, there exists $C$ depending only on $m$, $d$ and the initial condition such that
$$\iint_{\T^d\times\R^d} f_\eps^m(t)\dd v \dd x + \iint_{\T^d\times\R^d} |v|^2f_\eps(t)\dd v \dd x
+ \frac 1 {\eps^2} \int_0^t D[f_\eps(s)]\, \dd s \leq C.$$
When $\Omega =\R^d$ and $m\in(\frac{d}{d+1},1)$, there exists $C$ depending only on $m$, $d$ and the initial condition such that
$$
\iint_{\R^d\times\R^d} f_\eps^m(t)\dd v \dd x + \iint_{\R^d\times\R^d} (|v|^2+ |x|^2 ) f_\eps(t) \dd x \dd v
+ \frac 1 {\eps^2} \int_0^t D[f_\eps(s)]\, \dd s \leq C.
$$ 
\end{corollary}
\begin{proof}[{\bf Proof of \Cref{cor:CLbd}.}]
We begin again with the case $\Omega =\T^d$ and $m\in(\frac{d}{d+2},1)$. In that case, \eqref{eq:CLbd1} and the conservation of mass gives
\begin{align*}
\E[f_\eps] & = -\frac{1}{1-m} \iint_{\T^d\times\R^d} f_\eps^m\dd v \dd x + \iint_{\T^d\times\R^d}\frac{1}{2} |v|^2f_\eps\dd v \dd x \\
& \geq \iint_{\T^d\times\R^d}\frac{1}{4} |v|^2f_\eps\dd v \dd x - C
\end{align*}
with $C$ independent of $\eps$.
The entropy inequality \eqref{eq:entropyeps} together with the conservation of mass then implies
\begin{align*}
\iint_{\T^d\times\R^d}\frac{1}{4} |v|^2f_\eps(t)\dd v \dd x
+ \frac 1 {\eps^2} \int_0^t D[f_\eps(s)]\, \dd s \leq
\mathcal E[f_{\text{in}}] + C
\end{align*}

When $\Omega=\R^d$ and $m\in(\frac{d}{d+1},1)$, we first note that
\begin{align*}
\frac{d}{dt}\iint_{\R^d\times\R^d} |x|^2 f_\eps(t,x,v)\dd v\dd x
& = \frac 1 \eps \iint_{\R^d\times\R^d} x\cdot v f_\eps(t,x,v)\dd v\dd x\\
& = \frac 1 \eps \iint_{\R^d\times\R^d} x\cdot[ v f_\eps(t,x,v)+ \na_v f_\eps^m (t,x,v) ] \dd v\dd x\\
&\leq \left( \frac 1 {\eps^2} \iint_{\R^d\times\R^d} f_\eps(t,x,v) \left|v +
\frac{m}{m-1} \na_v f_\eps^{m-1}(t,x,v)\right|^2\dd x \dd v \right)^{\frac12} \\
& \qquad \times \left(\iint_{\R^d\times\R^d} |x|^2f_\eps(t,x)\dd v\dd x\right) ^{\frac12}
\end{align*}
that is
$$\frac{d}{dt}\iint_{\R^d\times\R^d} |x|^2 f_\eps(t,x,v) \dd x \dd v\leq
\left( \frac 1 \eps D_\eps(t)\right)^{\frac12} \left(\iint_{\R^d\times\R^d} |x|^2 f_\eps(t,x,v)\dd x \dd v \right) ^{\frac12}
$$
Combined with \ref{entro}, this implies
$$
\frac{d}{dt}
\iint_{\R^d\times\R^d} \frac{f_\eps^m}{m-1} + \frac{1}{2}(|v|^2+ |x|^2 ) f_\eps \dd x \dd v
+ \frac 1 {\eps^2} D[f_\eps(t)] \leq \left( \frac 1 \eps D_\eps(t)\right)^{\frac12} \left(\iint_{\R^d\times\R^d} |x|^2 f_\eps(t,x,v)\dd x \dd v \right) ^{\frac12}
$$
and so
\begin{equation}\label{eq:entropyx}
\frac{d}{dt}
\iint_{\R^d\times\R^d} \frac{f_\eps^m}{m-1} + \frac{1}{2}(|v|^2+ |x|^2 ) f_\eps \dd x \dd v
+ \frac 1 {2 \eps^2} D[f_\eps(t)] \leq \iint_{\R^d\times\R^d} |x|^2 f_\eps(t,x,v)\dd x \dd v
\end{equation}
We now use \eqref{eq:CLbd1} and the conservation of mass to write
\begin{equation}\label{eq:momentxv}
\frac 1 4 \iint_{\R^d\times \R^d}( |v|^2 +|x|^2 ) f_\eps \dd v \dd x \leq
\iint_{\R^d\times\R^d} \frac{f^m_\eps}{m-1}\dd v \dd x
+ \frac 1 2 \iint_{\R^d\times \R^d}( |v|^2 +|x|^2 ) f_\eps \dd v \dd x + C
\end{equation}
which together with \eqref{eq:entropyx} implies
$$
\frac{d}{dt}
\iint_{\R^d\times\R^d} \frac{f_\eps^m}{m-1} + \frac{1}{2}(|v|^2+ |x|^2 ) f_\eps \dd x \dd v
+ \frac 1 {2 \eps^2} D[f_\eps(t)] \leq 4 \iint_{\R^d\times\R^d} \frac{f_\eps^m}{m-1} + \frac{1}{2}(|v|^2+ |x|^2 ) f_\eps \dd x \dd v + C
$$
Finally a Gronwall argument yields
$$
\iint_{\R^d\times\R^d} \frac{f_\eps^m(t)}{m-1} + \frac{1}{2}(|v|^2+ |x|^2 ) f_\eps(t) \dd x \dd v
+ \frac 1 {2 \eps^2} \int_0^t D[f_\eps(s)]\, \dd s \leq C
$$
for a constant $C$ depending only on $m$, $d$ and the initial condition, and we can use \eqref{eq:momentxv} again to conclude.
\end{proof}

With this corollary, we recover the same bounds as in the case $m>1$, namely:
\begin{proposition}\label{prop:apriorim1}
Under the assumptions of Theorem \ref{thm:convm1}, $f_\eps(t,x,v)$ satisfies
\begin{align}
& \| f_\eps(t)\|_{L^1(\Omega\times\R^d)} = \|\rho_\eps(t)\|_{L^1(\Omega)} = \| f_{\text{in}}\|_{L^1(\Omega\times\R^d)}, \label{eq:L1bdm1} \\
& \|f_\eps(t)\|_{L^m(\Omega\times\R^d)}\leq C,\label{eq:Lmbdm1} \\
& \iint_{\Omega\times \R^d} |v|^2 f_\eps(t,x,v)\dd v\dd x \leq C,\label{eq:mombdm1} \\
& \|\rho_\eps(t)\|_{L^k(\Omega)}\leq C, \label{eq:Lkbdm1}\\
& \int_0^\infty D (f_\eps(t))\, dt \leq C \eps^2, \label{eq:dissbdm1}
\end{align}
where $C=\E[f_{\text{in}}]$ is independent of $\eps$. Furthermore,
\begin{equation}\label{eq:Linftyepsm1}
\| f_\eps\|_{L^\infty((0,\infty)\times \Omega \times\R^d)} \leq C, \qquad 
\| \rho_\eps\|_{L^\infty((0,\infty)\times \Omega)} \leq C
\end{equation}
and
the flux $j_\eps(t,x)$ is bounded in $L^2((0,T);L^1(\Omega)) $.
\end{proposition}

\subsubsection{Strong convergence of the density}
As in the case $m>1$, we can then prove:
\begin{proposition}\label{prop:rhom1}
The family $(\rho_\eps)_{\eps \in (0,1]}$ is relatively compact in $L^1((0,T)\times\Omega)$:
There exists a function $\rho \in L^1((0,T)\times\Omega)$ such that, up to a subsequence, $\rho_\eps$ converges to $\rho$ strongly in $L^1((0,T)\times\Omega)$ and almost everywhere.
\end{proposition}

\begin{proof}[{\bf Proof of \Cref{prop:rhom1}}]
The proof is similar to that of Proposition \ref{prop:rho}. In particular, we still have that $ \pa_t \rho_\eps$ is bounded in $L^2((0,T);W^{-1,1}(\Omega))$ (this follows from the continuity equation \eqref{eq:macro} and the bound on $j_\eps$ given by Proposition \ref{prop:apriorim1}.

We can still write
\begin{equation}
\eps \pa_ t f_\eps + v\cdot \na_x f_\eps = \div_v g_\eps
\end{equation}
where
$$ g_\eps = \frac{1}{\eps} \left[v+\frac{m}{m-1}\na_v f_\eps^{m-1}\right] f_\eps $$
is bounded in $ L^2((0,T);L^1(\Omega\times\R^d))$.

The $L^\infty$ bound \eqref{eq:Linftyepsm1} implies that $f_\eps $ is bounded in $L^2((0,T);L^2(\Omega\times\R^d))$ and we can use the same averaging lemma argument to show that
$$\lim_{z \to 0} \; \left\| \int_{\R^d} ( f_\eps(,\cdot+z,v)- f_\eps(\cdot,\cdot,v)) \psi(v)\, \dd v \right\|_{L^1((0,T)\times\Omega)} = 0$$
uniformly with respect to $\eps>0$.
Finally, we can use the bound on $\iint_{\R^d\times \Omega} |v|^2 f_\eps \dd x \dd v $ to take $\psi\equiv 1$ conclude as before (using the bound on $\iint_{\R^d\times \Omega} |x|^2 f_\eps \dd x \dd v $ when $\Omega =\R^d$).
\end{proof}

Proposition \ref{prop:convP} can be proved as before: It relies on Corollary \ref{cor:ineqm} which is valid for $m\geq \frac{d-1}{d}$ ($m>\frac 1 2$ when $d=2$).

Furthermore, Proposition \ref{prop:entropm<1} together with Proposition \ref{prop:entropydissipation} imply:
\begin{lemma}\label{lem:convf}
Assume $m\geq \frac{d-1}{d}$, $m>\frac{1}{2}$
and for $\Omega=\T^d$ (For $\Omega =\R^d$ we get only local norm in $x$ at least for $m>1$ ) we have:
\begin{equation}
\int_0^T\iint_{\Omega\times\R^d} |f^m_\eps-G[\rho_\eps]^m|\, \dd x\, \dd v\, dt \to 0
\end{equation}
\end{lemma}
\begin{proof}[{\bf Proof of \Cref{lem:convf}}]
First, we classically have (see \cite[Corollary 13]{DPD} and Proposition \ref{prop:entropydissipation})
\begin{equation}\label{eq:GNbis}
0\leq \E[f_\eps(t)] - \E[G[\rho_\eps(t)]]\leq\eps^2 D_\eps(t) \qquad \mbox{ if } \frac{d-1}{d}\leq m <1
\end{equation}
and so
$$\int_0^T \E[f_\eps(t)] - \E[G[\rho_\eps(t)]]\, dt \leq C \eps^2.$$
Next, Proposition 15 in \cite{DPD} (see also Proposition \ref{prop:entropm<1}) implies
\begin{equation}\label{eq:CK}
\iint_{\Omega\times\R^d} |f_\eps^m - G[\rho_\eps]^m|\, \dd v \, \dd x
\leq \max\left\{\iint_{\Omega\times \R^d} f_\eps^m \dd v\dd x ,\iint_{\Omega\times \R^d} G[\rho_\eps]^m \dd v\dd x \right\}^{\frac 1 2} \left(\E[f_\eps] - \E[G[\rho_\eps]]\right)^{\frac 1 2} .
\end{equation}
Together with \eqref{eq:GN}, this implies
$$
\int_0^T\iint_{\Omega\times\R^d} |f^m_\eps-G[\rho_\eps]^m|\, \dd x\, \dd v\, dt \leq C(T) \eps
$$
and the result follows.
\end{proof}

Lemma \ref{lem:convf} implies that $f_\eps^m$ converges strong in $L^1$, and therefore almost everywhere (up to a subsequence) to $G[\rho]^m$.
It follows that $f_\eps$ converges almost everywhere to $G[\rho]$, and therefore strongly in $L^1$.

\section{Relative entropy method}

In this section, we prove the convergence of $f_\eps$ to $G[\rho]$ using the relative entropy method.

\subsection{proof of Proposition \ref{prop:relativeentropym}}

Our first task is to prove the relative entropy inequality \eqref{eq:relativeentropy}.

Define $\mathbb{E}$ by
\begin{equation*}
\mathbb{E}[\rho,u]: v \mapsto \mu[\rho]-\frac{m-1}{2\,m}\,|v-u|^2,
\end{equation*}
so that
\begin{align*}
\mathscr H_\eps(t) &= \frac{1}{m-1}\irdxv{ \left(f_\eps^m -m \, \mathbb{E}[\rho_0,\var w_0] f_\eps \right)} + \irdx{\nu[\rho_0]}.
\end{align*}

From this,
\begin{align*}
\eps^2 \frac{d}{dt}
\mathscr H_\eps(t)
& = - \frac{m}{m-1}\irdxv{ \var^2\left(\partial_t (\mu(\rho_0))+ \var \frac{1-m}{m} (v - \var w_0) \cdot \partial_t w_0\right) f_\eps} + \eps^2\irdx{\partial_t\nu[\rho_0]}\\
&\qquad \qquad + \frac{m}{m-1}\irdxv{ \left( f_\eps^{m-1} - \, \mathbb{E}[\rho,\var w_0] \right) \eps^2 \partial_t f_\eps } \\
& = - \frac{m}{m-1}\irdxv{ \var^2\left(\partial_t (\mu(\rho_0))+ \var \frac{1-m}{m} (v - \var w_0) \cdot \partial_t w_0\right) f_\eps} + \eps^2\irdx{\partial_t\nu[\rho_0]}\\
&+ \frac{m}{m-1}\irdxv{ \var\,v\cdot\nabla_x \left( f_\eps^{m-1} -\mathbb{E}[\rho,\var w_0] \right) \, f_\var }\\
& + \frac{m}{m-1}\irdxv{ \left( f_\eps^{m-1} - \mathbb{E}[\rho,\var w_0] \right) \left(\Delta_vf_\var^m+\nabla_v\cdot (v\,f_\var) \right) } \\ 
&= I_1 + I_2 + I_3.
\end{align*}

We shall now estimate these three terms separately.

\begin{align*}
I_1 &= - \frac{m \, \eps^2}{m-1}\irdxv{ f_\eps \left(\partial_t (\mu(\rho_0))+ \var \frac{1-m}{m} (v - \var w_0) \cdot \partial_t w_0\right)} + \, \eps^2 \irdx{ \,\partial_t \left(\nu(\rho_0) \right)}\\
&= \var^2 \irdx{ \left(\nu'(\rho_0) - \frac{m}{m-1}\rho_\eps \mu'(\rho_0) \right)\partial_t \rho_0} - \eps^2\irdxv{ f_\eps \left(\var (v - \var w_0) \cdot \partial_t w_0\right)}\\
&= -\var^2 \irdx{ \left(\nu'(\rho_0) - \frac{m}{m-1}\rho_\eps \mu'(\rho_0) \right) \text{div}_x(\rho_0 w_0)} - \eps^3\irdxv{ \left(\na_v f_\eps^m + (v - \var w_0)f_\eps \right) \cdot \partial_t w_0}\\
&= - \var^2 \irdx{ \left(\nu'(\rho_0) - \frac{m}{m-1}\rho_\eps \mu'(\rho_0) \right) \rho_0 \text{div}_x w_0} - \var^2 \irdx{ \left(\nu'(\rho_0) - \frac{m}{m-1}\rho_\eps \mu'(\rho_0) \right) \nabla_x \rho_0 \cdot w_0} \\
&\qquad\qquad- \eps^3\irdxv{ \left(\na_v f_\eps^m + (v - \var w_0)f_\eps \right) \cdot \partial_t w_0}
\end{align*}
where we have used that $\partial_t w_0 + \nabla \cdot( \rho_0 w) = 0$. We pursue as follows, using that, by \eqref{eq:numu}, $\nu'(\rho) = \frac{m}{m-1} \mu'(\rho)\rho$, $ \na_x \nu(\rho_0) = \rho_0 \na \frac{m}{m-1} \mu(\rho_0) $ and $w= - \frac{m}{m-1}\na_x \mu(\rho_0)$,
\begin{align*}
&-\var^2 \irdx{ \left(\nu'(\rho_0) \rho_0 - \frac{m}{m-1}\rho_\eps \mu'(\rho_0) \rho_0 \right) \text{div}_x w_0} - \var^2 \irdx{ \left(\nabla_x\nu(\rho_0) - \frac{m}{m-1}\rho_\eps \nabla_x\mu(\rho_0) \right) \cdot w} \\
&= -\var^2 \irdx{ \nu'(\rho_0) \left(\rho_0 - \rho_\eps \right) \text{div}_x w_0} - \var^2 \irdx{ \left(\nabla_x\nu(\rho_0) - \frac{m}{m-1}\rho_\eps \nabla_x\mu(\rho_0) \right) \cdot w}\\
&= -\var^2 \irdx{ \nu'(\rho_0) \left(\rho_0 - \rho_\eps \right) \text{div}_x w_0} + \var^2 \irdx{ \nu(\rho_0) \text{div}_x w} - \var^2 \irdx{\rho_\eps \vert w \vert^2} \\
&= \var^2 \irdx{ \left( \nu(\rho_0) - \nu'(\rho_0) \left(\rho_0 - \rho_\eps \right) \right) \text{div}_x w_0} -\var^2 \irdx{\rho_\eps \vert w \vert^2}
\end{align*}

We conclude, that,
\begin{align*}
I_1 &= \var^2 \irdx{ \left( \nu(\rho_0) - \nu'(\rho_0) \left(\rho_0 - \rho_\eps \right) \right) \text{div}_x w_0} -\var^2 \irdx{\rho_\eps \vert w \vert^2} \\
&\qquad \qquad - \eps^3\irdxv{ \left(\na_v f_\eps^m + (v - \var w_0)f_\eps \right) \cdot \partial_t w_0}
\end{align*}

We have used above

\begin{align*}
I_2 &= \frac{m}{m-1}\irdxv{ \var\,v\cdot\nabla_x \left( f_\eps^{m-1} -\mathbb{E}[\rho,\var w_0] \right) \, f_\var }\\
&=-\irdxv{ f_\var\var\,v\cdot \left( \frac{m}{m-1}\nabla_x \mu(\rho_0) - \eps D_xw \,(v-\eps w_0) \right) } \\
&= \var\irdxv{ v\cdot w \, f_\var} + \eps^2 \irdxv{ v\cdot D_xw \,(v-\eps w_0) \, f_\var \, }\\
&=\var\irdxv{ v\cdot w \, f_\var} \\
&\qquad \qquad + \eps^2 \irdxv{ v\cdot D_xw \,(\na_v f_\eps^m+ (v-\eps w_0) f_\eps)} - \eps^2 \irdxv{ v\cdot D_xw \,\na_v f_\eps^m}\\
& =\var\irdxv{ v\cdot w \, f_\var} - \eps^2 \irdxv{ f_\eps^m \text{div}_x w} \\
&\qquad \qquad \qquad+ \eps^2 \irdxv{ v\cdot D_xw \,(\na_v f_\eps^m+ (v-\eps w_0) f_\eps)}
\end{align*}

\begin{align*}
I_3 &= \frac{m}{m-1}\irdxv{ \left( f_\eps^{m-1} - \mathbb{E}[\rho,\var w_0] \right) \left(\Delta_vf_\var^m+\nabla_v\cdot (v\,f_\var) \right) } \\
&= - \frac{m}{m-1}\irdxv{ \left(\nabla_v f_\var^m+ v\,f_\var \right) \, \text{div}_v \left( f_\eps^{m-1} - \mathbb{E}[\rho,\var w_0] \right) }\\
&= - \irdxv{ \left(\nabla_v f_\var^m+ v\,f_\var \right) \cdot \left( \frac{m}{m-1} \nabla_v (f_\eps^{m-1}) + (v-\eps w_0) \right) }\\
&= - \irdxv{ \frac{1}{f_\var} \left\vert \nabla_v f_\var^m+ (v-\eps w_0)\,f_\var \right\vert^2 } \\
& \qquad \qquad - \irdxv{ \left( \eps w_0 \,f_\var \right) \cdot \left( \frac{m}{m-1} \nabla_v (f_\eps^{m-1}) + (v-\eps w_0) \right) }\\
&= - \irdxv{ \frac{1}{f_\var} \left\vert \nabla_v f_\var^m+ (v-\eps w_0)\,f_\var \right\vert^2 } - \irdxv{ \eps w_0 \cdot (v-\eps w_0) \,f_\var }\\
&= - \irdxv{ \frac{1}{f_\var} \left\vert \nabla_v f_\var^m+ (v-\eps w_0)\,f_\var \right\vert^2 } - \irdxv{ \eps w_0 \cdot v \,f_\var } + \eps^2 \irdx{ \vert w \vert^2 \,\rho_\var }
\end{align*}

Summing up, we get,
\begin{align*}
&\eps^2 \frac{d}{dt}
\mathscr H_\eps(t) + \irdxv{ \frac{1}{f_\var} \left\vert \nabla_v f_\var^m+ (v-\eps w_0)\,f_\var \right\vert^2 }\\
&= \var^2 \irdx{ \left( \nu(\rho_0) - \nu'(\rho_0) \left(\rho_0 - \rho_\eps \right) \right) \text{div}_x w_0} - \eps^2 \irdxv{ f_\eps^m \text{div}_x w} \\
&\qquad \qquad - \eps^3\irdxv{ \left(\na_v f_\eps^m + (v - \var w_0)f_\eps \right) \cdot \partial_t w_0} \\
&\qquad \qquad \qquad \qquad + \eps^2 \irdxv{ v\cdot D_xw \,(\na_v f_\eps^m+ (v-\eps w_0) f_\eps)}
\end{align*}
We shall analyze the first integral of the r.h.s separately. Since $\nu(\rho_0) - \nu'(\rho_0)\rho_0 = -(k-1) \nu(\rho_0)$,
\begin{align*}
&\irdx{ \left( \nu(\rho_0) - \nu'(\rho_0) \left(\rho_0 - \rho_\eps \right) \right) \text{div}_x w_0} \\
&= \irdx{ \left( -(k-1) \nu(\rho_0) + \nu'(\rho_0) \rho_\eps \right) \text{div}_x w_0} \\
&=-(k-1) \irdx{ \left( \nu(\rho_0) - \frac{\nu'(\rho_0) }{k-1}\rho_\eps \right) \text{div}_x w_0} \\
&=-(k-1) \irdx{ \nu(\rho_0) \, \text{div}_x(w)} \\
&\qquad \qquad+(k-1) \irdxv{\left( \frac{m}{m-1} \mathbb{E}[\rho_0,\var w_0] + \frac{1}{2}\,|v-\var w_0|^2\right) f_\eps \text{div}_x w_0}\\
&=-(k-1) \irdx{ \nu(\rho_0) \, \text{div}_x w_0} +\frac{k-1}{m-1} \irdxv{ m\mathbb{E}[\rho_0,\var w_0] \, f_\eps \text{div}_x w_0}\\
&\qquad \qquad \qquad \qquad +(k-1) \irdxv{\frac{1}{2}\,|v-\var w_0|^2 f_\eps \text{div}_x w_0}
\end{align*}
Hence,
\begin{align*}
&\irdx{ \left( \nu(\rho_0) - \nu'(\rho_0) \left(\rho_0 - \rho_\eps \right) \right) \text{div}_x w_0} - \irdxv{ f_\eps^m \text{div}_x w} \\
&=-(k-1) \left( \irdx{ \nu(\rho_0) \, \text{div}_x w_0} +\frac{1}{m-1} \irdxv{ \left(f_\eps^m - m\mathbb{E}[\rho_0,\var w_0] \, f_\eps \right) \text{div}_x w_0} \right)\\
&\qquad \qquad +(k-1) \irdxv{\frac{1}{2}\,|v-\var w_0|^2 f_\eps \text{div}_x w_0} - \left( 1 - \frac{k-1}{m-1}{}\right) \irdxv{ f_\eps^m \text{div}_x w}\\
&=-(k-1) \left( \irdx{ \nu(\rho_0) \, \text{div}_x w_0} +\frac{1}{m-1} \irdxv{ \left(f_\eps^m - m\mathbb{E}[\rho_0,\var w_0] \, f_\eps \right) \text{div}_x w_0} \right) \\
&\qquad \qquad +(k-1) \irdxv{\left( \frac{1}{2}\,|v-\var w_0|^2 f_\eps -\frac d 2 f_\eps^m \right)\text{div}_x w_0},
\end{align*}
using the fact that $\frac{1}{k-1}=\frac{1}{m-1} + \frac d 2$. Now, since,
\begin{align*}
\irdv{ \frac {d}{2} f_\eps^m - \frac 1 2 |v-\eps w_0|^2 f_\eps } & =
\frac 1 2 \irdv{ \na _v \cdot (v-\eps w_0) f_\eps^m - |v-\eps w_0|^2f_\eps } \\
& = - \frac 1 2 \irdv{(v-\eps w_0)\cdot \na_v f_\eps^m + |v-\eps w_0|^2f_\eps }\\
& = - \frac 1 2 \irdv{(v-\eps w_0)\cdot ( \na _v f_\eps^m + (v-\eps w_0) f_\eps ) }.
\end{align*}
Putting previous estimates together, we finally obtain,
\begin{align*}
&\eps^2 \frac{d}{dt}
\mathscr H_\eps(t) + \irdxv{ \frac{1}{f_\var} \left\vert \nabla_v f_\var^m+ (v-\eps w_0)\,f_\var \right\vert^2 }\\
&= -(k-1) \left( \irdx{ \nu(\rho_0) \, \text{div}_x w_0} +\frac{1}{m-1} \irdxv{ \left(f_\eps^m - m\mathbb{E}[\rho_0,\var w_0] \, f_\eps \right) \text{div}_x w_0} \right)\\
&-\frac12(k-1) \var^2\irdxv{\left((v-\eps w_0)\cdot ( \na _v f_\eps^m + (v-\eps w_0) f_\eps ) \right)\text{div}_x w_0}\\
&- \eps^3\irdxv{ \left(\na_v f_\eps^m + (v - \var w_0)f_\eps \right) \cdot \partial_t w_0} + \eps^2 \irdxv{ v\cdot D_xw \,(\na_v f_\eps^m+ (v-\eps w_0) f_\eps)}
\end{align*}
We shall control the terms in the r.h.s. Using the convexity of $s \mapsto \frac{s^m}{m-1}$, and Cauchy-Schwarz three times, we finish the proof, with,
\begin{align*}
&\frac{d}{dt} \mathscr H_\eps(t) +\frac{1}{4 \eps^{2}} \irdxv{ \frac{1}{ f_\eps} |\na_v f_\eps^m +(v- \eps w_0) f_\eps |^2 } \\
&\leq \vert k-1 \vert\| \div w\|_{L^\infty} \mathscr H_\eps(t) + \eps^4 \|\pa_t w\|^2_{L^\infty} \irdxv{ f_\eps } + \eps^2 \|Dw\|_{L^\infty}^2 \irdxv{ | v| ^2 f_\eps} \\
& \quad
+\eps^2 \left(\frac {k-1 } 2\right)^2 \|\div w\|^2 _{L^\infty} \irdxv{|v-\eps w_0|^2 f_\eps }.
\end{align*}

\appendix\section{Key functional inequalities}

Throughout this section we assume that $v\mapsto f(v)$ is a non-negative function in $L^1(\R^d)$ with mass $\rho=\int_{\R^d} f(v)\, \dd v$ and we gather some important results concerning the entropy
$$
H(f) := \int_{\R^d} \frac{f^m(v)}{m-1} + \frac{|v|^2}{2} f(v)\, \dd v.
$$
We denote again by
$$G[\rho] = \left(\mu[\rho] - \frac{m-1}{2m} |v|^2\right)_+^{\frac{1}{m-1}} $$
the equilibrium with $\mu[\rho]$ chosen so that $\int_{\R^d} G[\rho](v)\, \dd v =\rho$.

First, we recall the following classical result:
\begin{proposition}{\cite[Corollary 13]{DPD}}\label{prop:entropydissipation}
Let $m\geq \frac{d-1}{d}$, $m\neq 1$ (and $m>\frac 1 2$ if $d=2$, $m>\frac 1 3$ if $d=1$) and assume that $f$ satisfies
$$
\na_v f^{m-\frac 1 2} \in L^2(\R^d), \quad |\cdot|^2 f \in L^1(\R^d) ,
$$
and
$$ f \in L^1(\R^d)\quad \text{ if } m> 1, \quad f^m\in L^1(\R^d) \quad\text{ if } m< 1.$$
Then
\begin{equation}\label{eq:entrdiss}
0\leq H(f)-H(G[\rho]) \leq \frac 1 2 \int_{\R^d} f \left| v+\frac{m}{m-1} \na_v f^{m-1}\right|^2\, \dd v.
\end{equation}
\end{proposition}

Furthermore, we can prove the following corollary:
\begin{corollary}\label{cor:ineqm}
Let $m\geq \frac{d-1}{d}$, $m\neq 1$ (and $m>\frac 1 2$ if $d=2$, $m>\frac 1 3$ if $d=1$) and assume that $f(v)$ satisfies the conditions of Proposition \ref{prop:entropydissipation}. Then the following inequality holds:
\begin{align}
\frac{1}{k-1} \left|\int_{\R^d} f^m(v) - G[\rho]^m(v)\, \dd v\right|
&\leq \frac 1 2 \int_{\R^d} f \left| v+\frac{m}{m-1} \na_v f^{m-1}\right|^2\, \dd v\nonumber \\
&\quad+ \frac{1}{2}\left(\int_{\R^d} | v|^2 f\dd v \right)^{\frac 1 2 } \left(\int_{\R^d} f \left|\frac{m}{m-1} \nabla_v f^{m-1} + v \right|^2\,\dd v \right)^{\frac 1 2 } . \label{eq:ineqm}
\end{align}
\end{corollary}
\begin{proof}
The definition of $H$ gives
\begin{align*}
\frac{1}{m-1} \left(\int_{\R^d} f^m(v) - G[\rho]^m(v)\, \dd v\right)&= H(f) - H(G[\rho])- \int_{\R^d} \frac{|v|^2}{2} \left( f(v) - G[\rho](v) \right)\, \dd v. 
\end{align*}
Next, we note that (using \eqref{eq:nu})
\begin{align*}
&\int_{\R^d} \frac{|v|^2}{2} \left( f(v) - G[\rho ](v) \right)\, \dd v \\
&= \frac{1}{2}\int_{\R^d} v \cdot \left(\na_v f^m(v)+ vf(v) \right)
\dd v- \frac{1}{2}\int_{\R^d} v \cdot \na_v f^m(v) \dd v - \frac{1}{2}\int_{\R^d} |v|^2 G[\rho](v) \, \dd v \\
&= \frac{1}{2}\int_{\R^d} v \cdot \left( \nabla_v f^m (v)+ vf(v) \right)\, \dd v + \frac{d}{2} \int_{\R^d} \left( f^m(v) - G[\rho]^m(v) \right)\, \dd v .
\end{align*} 
Combing these two equalities and using \eqref{eq:km}, we deduce
\begin{align*}
\frac{1}{k-1} \left(\int_{\R^d} f^m(v) - G[\rho]^m(v)\, \dd v\right)&= H(f) - H(G[\rho])- \frac{1}{2}\int_{\R^d} v \cdot \left( \nabla_v f^m + vf \right)\, \dd v. 
\end{align*}
Finally, \Cref{prop:entropydissipation} (see \eqref{eq:entrdiss}) implies
\begin{align*}
&\frac{1}{k-1} \left|\int_{\R^d} f^m(v) - G[\rho]^m(v)\, \dd v\right|\\
&\qquad \qquad \leq \frac 1 2 \int_{\R^d} f \left| v+\frac{m}{m-1} \na_v f^{m-1}\right|^2\, \dd v + \frac{1}{2}\int_{\R^d} | v| f \left|\frac{m}{m-1} \nabla_v f^{m-1} + v \right|\, \dd v
\end{align*}
and \eqref{eq:ineqm} follows.
\end{proof}

Corollary \ref{cor:ineqm} will be used to prove the (strong) convergence of $\int f_\eps^m(t,x,v)\dd v$ in $L^1((0,T)\times \Omega)$. The strong convergence of $f_\eps^m(t,x,v)$ in $L^1((0,T)\times \Omega\times \R^d)$ requires further work:
First, when $m<1$, one has:
\begin{proposition}{\cite[Proposition 15]{DPD}}\label{prop:entropm<1}
Let $\frac{d-2}{d}\leq m <1$ and $m>\frac 1 2$. Then
\begin{equation}
\int_{\R^d} | f(v)^m-G[\rho](v)^m|\, \dd v \lesssim_{d,m}\, \max\left\{ \int_{\R^d} f(v)^m\, \dd v,\int_{\R^d} G[\rho](v)^m\, \dd v \right\}^{\frac 1 2}( H[f]-H[G[\rho]] )^{\frac 1 2}.
\end{equation}
\end{proposition}

When $m>1$ the result differs in two significant ways: We control the $L^1$ norm of $f-G[\rho]$ rather than $f^m-G[\rho]^m$, but most importantly, the fact that $G[\rho]$ is compactly supported raises some important issues. The first result thus assumes that $f$ and $G$ have the same support:
\begin{proposition}{\cite[Lemma 4.1 and 4.2]{carrillo2000asymptotic}}\label{prop:entropm>1}
Let $m>1$ and assume that $f$ and $G[\rho]$ have the same mass $\rho$ and the same support $B_R(0) = \{G[\rho]>0\}$. Then
$$ \int_{B_R} |f(v)-G[\rho](v)|\, \dd v \lesssim_{d,m}\, \rho^{\frac{2-k}{2}}\, ( H[f]-H[G[\rho]] )^{\frac 1 2}. $$
\end{proposition}
\begin{proof}[{\bf Proof of \Cref{prop:entropm>1}}]
Due to Lemma 4.1 in \cite{carrillo2000asymptotic}, we shall split the values of $m$.

When $m>2$, Lemma 4.1 in \cite{carrillo2000asymptotic} applied with $p=q=2$ (in their notations) gives
$$ 
\int_{f>G[\rho]} |f(v)-G[\rho](v)|\, \dd v \leq \left(\frac 1 m \int_{G[\rho]>0} G[\rho]^{2-m}(v)\, \dd v\right)^{\frac{1}{2}} ( H[f]-H[G[\rho]] )^{\frac{1}{2}}
$$
where
$$\int_{G[\rho]>0} G[\rho](v)^{2-m}\, \dd v = \int_{B_R} \left(\mu - \frac{m-1}{2m} |v|^2\right)_+^{\frac{2-m}{m-1}}\, \dd v = C_{d,m}\, \mu^{\frac{2-m}{m-1}+\frac{d}{2}} = C_{m} \rho^{2-k}.
$$
We recall that $k <1+\frac {2}{d} \leq 2$ and so this integral is finite despite the singularity of the integrand on $\pa B_R$.
We conclude using the fact that $f$ and $G[\rho]$ has the same mass to write:
$$
\int_{\R^d} |f(v)-G[\rho](v)|\, \dd v = 2 \int_{f>G[\rho]} |f(v)-G[\rho](v)|\, \dd v,
$$
and conclude.

When $m\leq2$, Lemma 4.1 in \cite{carrillo2000asymptotic} applied with $p=q=2$ (in their notations) gives
$$ 
\int_{f<G[\rho]} |f(v)-G[\rho](v)|\, \dd v \leq \left(\frac 1 m \int_{G[\rho]>0} G[\rho]^{2-m}(v)\, \dd v\right)^{\frac{1}{2}} ( H[f]-H[G[\rho]] )^{\frac{1}{2}}
$$
and we can conclude similarly.
\end{proof}

In order to $f - G[\rho]$ without assuming anything about the support of these functions, we need to control how much mass $f$ has outside of the support of $G[\rho]$:
\begin{lemma}\label{lem:support}
Let $m>1$ and assume that $f$ and $G[\rho]$ have the same mass $\rho$.
Denote by $B_R = \{G[\rho]>0\}$ the support of $G[\rho]$ (that is $R^2=\frac{2 m}{m-1}\mu[\rho] $).
Then:
$$\left(\int_{|v|\geq R} f(v)\, \dd v \right)^k
\lesssim \max\left\{ (H[f]-H[G[\rho]])^{\frac 1 q} \rho^{k\left(1-\frac{1}{q}\right)} , H[f]-H[G[\rho]]\right\},
$$
with $q=\frac{2m-1}{mk}$.
\end{lemma}
\begin{proof}[{\bf Proof of \Cref{lem:support}}]
The proof follows the argument developed in the proof of Theorem 4.5 in \cite{carrillo2000asymptotic}. Recall the definiton of the relative entropy,
$$
\E[f|G[\rho]] = \int_{\R^d} \frac{f^m}{m-1}- \frac{G[\rho]^m}{m-1}-\frac{m}{m-1} G[\rho]^{m-1}(f-G[\rho])\, \dd v.
$$
Write
\begin{align*}
&H[f]-H[G[\rho]] \\
&= \int_{\R^d} \frac{f^m(v)}{m-1} - \frac{G[\rho]^m(v)}{m-1} + \frac{|v|^2}{2} \left( f(v) - G[\rho](v)\right)\, \dd v\\
&= \E[f|G[\rho]] +\int_{\{G[\rho]>0\}} \left(\frac{m}{m-1}\mu-\frac{|v|^2}{2}\right) (f-G[\rho])\, \dd v \\
& \qquad \qquad \qquad \qquad + \int_{\R^d}\frac{|v|^2}{2} \left( f(v) - G[\rho](v)\right)\, \dd v \\
&= \E[f|G[\rho]] +\int_{\{G[\rho]>0\}} \left(\frac{m}{m-1}\mu-\frac{|v|^2}{2}\right) (f-G[\rho])\, \dd v \\
& \qquad \qquad + \int_{\{G[\rho]>0\}} \frac{|v|^2}{2} \left( f(v) - G[\rho](v)\right)\, \dd v + \int_{\{G[\rho]=0\}}\frac{|v|^2}{2} \left( f(v) - G[\rho](v)\right)\, \dd v \\
&= \E[f|G[\rho]] +\int_{\{G[\rho]>0\}} \frac{m}{m-1}\mu[\rho] \, (f-G[\rho])\, \dd v + \int_{\{G[\rho]=0\}}\frac{|v|^2}{2} \left( f(v) - G[\rho](v)\right)\, \dd v \\
&= \E[f|G[\rho]] -\int_{\{G[\rho]=0\}} \frac{m}{m-1}\mu[\rho] \, (f-G[\rho])\, \dd v + \int_{\{G[\rho]=0\}}\frac{|v|^2}{2} \left( f(v) - G[\rho](v)\right)\, \dd v.
\end{align*}
where we have used $\int_{\R^d} (f(v)-G[\rho](v)) \dd v =0$ at the last step. We deduce finally,
\begin{align}
H[f]-H[G[\rho]] = \E[f|G[\rho]] + \int _{|v|\geq R} \frac{1}{2}\left(|v|^2- R^2 \right) f(v) \, \dd v,\label{eq:HH}
\end{align}
with, recall, $R^2=\frac{2 m}{m-1}\mu $.

The convexity of the function $s\mapsto \frac{s^m}{m-1}$ implies that
$\frac{f^m}{m-1}- \frac{G[\rho]^m}{m-1}-\frac{m}{m-1} G[\rho]^{m-1}(f-G[\rho])\geq 0$ a.e. and so
\begin{equation}\label{eq:Hconv}
\E[f|G[\rho]] \geq \int_{|v|\geq R }\frac{f^m}{m-1}- \frac{G[\rho]^m}{m-1}-\frac{m}{m-1} G[\rho]^{m-1}(f-G[\rho]) \dd v=\int_{|v|\geq R } \frac{f^m}{m-1} \dd v.
\end{equation}
We deduce:
\begin{equation}\label{eq:masscomp}
K_R:=\int_{|v|\geq R} \frac{1}{2}\left(|v|^2- R^2 \right)f(v) + \frac{f(v)^m}{m-1}\, \dd v
\leq H[f]-H[G[\rho]] .
\end{equation}

We now write, with $\omega_d$ the volume of the unit ball,
\begin{align*}
\int_{|v|\geq R} f(v) \dd v & = \frac{1}{\eta^2}\int_{|v|^2 \geq R^2 + \eta^2} (|v|^2-R^2) f(v) \dd v + \int_{R^2 \leq |v|^2 \leq R^2 + \eta^2} f(v) \dd v\\
& \leq \frac{1}{\eta^2}\int_{|v|^2 \geq R^2} (|v|^2-R^2) f(v) \dd v + \left(\int_{ |v| \geq R} f(v)^m \dd v \right)^{\frac 1 m} \left( \omega_d ((R^2+\eta^2)^{d/2} - R^d)\right)^{\frac{m-1}{m}}\\
& \lesssim \frac{K_R}{\eta^2} + K_R^{\frac 1 m} R^{\frac{d(m-1)}{m}} \left(\left(1+\left(\frac{\eta}{R}\right)^2\right)^{\frac{d}{2}} - 1\right)^{\frac{m-1}{m}} .
\end{align*}
We write the right-hand-side under the form $R^{-2} \Phi\left(\frac{\eta^2}{R^2}\right)$ with
$$\Phi(x) = \frac{K_R}{R^{2}x} + K_R^{\frac 1 m} R^{\frac{d(m-1)}{m}} \left((1+x)^{\frac{d}{2}} - 1\right)^{\frac{m-1}{m}}, $$
and we search for a minimum value of $\Phi$.
We note that
$$\Phi'(1) = -R^{-2}K_R + \frac{m-1}{m} \left(2^{\frac{d}{2}} - 1\right)^{\frac{-1}{m}} \frac{d}{2}2^{\frac{d}{2}-1} K_R^{\frac 1 m} R^{\frac{d(m-1)}{m}} $$
and so a transition occurs when \[
K_R = C(d,m) R^{\frac{2m + d(m-1)}{m-1}},
\]
with $C(d,m)
=
\left[
\frac{d}{2}\,\frac{m-1}{m}\,
2^{\frac{d}{2}-1}
\left(2^{\frac{d}{2}}-1\right)^{-\frac{1}{m}}
\right]^{\frac{m}{m-1}}.$
\begin{itemize}
\item When $K_R < C(d,m)R^{\frac{2m + d(m-1)}{m-1}}$, we have $\Phi'(1)>0$, so the minimum of $\Phi$ is reached for some $x\in (0,1)$. In that region, we have $(1+x)^{\frac{d}{2}} - 1\lesssim x$ and so
$$\Phi(x) \lesssim \frac{R^{-2}K_R}{x} + K_R^{\frac 1 m} R^{\frac{d(m-1)}{m}} x^{\frac{m-1}{m}}, $$
and a standard minimization argument yields:
\begin{align*}
\min_{x\in(0,1)}\Phi(x) \lesssim K_R^{\frac{m}{2m-1}} R^{\frac{(d-2)(m-1)}{2m-1}} \lesssim K_R^{\frac{m}{2m-1}} \rho^{\frac{(d-2)(m-1)(k-1)}{2(2m-1)}} \lesssim K_R^{\frac{m}{2m-1}} \rho^{\frac{m(2-k)-1}{2m-1}},
\end{align*}
where we used the fact that $R^2=\frac{2 m}{m-1}\mu_1 \rho^{k-1} $ and the equality
$$(d-2)(m-1)(k-1)= 2(m(2-k)-1).$$
\item When $K_R > C(d,m) R^{\frac{2m + d(m-1)}{m-1}}$, we have $\Phi'(1)<0$, so the minimum of $\Phi$ is reached in the region for some $x>1$. In that region, we have
$(1+x)^{\frac{d}{2}} - 1\leq 2^{\frac{d}{2}} x^{\frac{d}{2}}$ and so
$$\Phi(x) = \frac{K_R}{xR^2} + 2^{\frac{d}{2}} K_R^{\frac 1 m} R^{\frac{d(m-1)}{m}} x^{\frac{d(m-1)}{2m}}, $$
and a standard minimization argument yields:
$$ \min_{x>1}\Phi(x) \lesssim K_R^{\frac{d(m-1)+2}{d(m-1)+2m}}= K_R^{\frac{1}{k}}.
$$
\end{itemize}
We deduce
\begin{align*}
\int_{|v|\geq R} f(v) \dd v
& \lesssim \max\left\{K_R^{\frac{m}{2m-1}} \rho^{\frac{m(2-k)-1}{2m-1}}, K_R^{\frac{1}{k}}\right\},
\end{align*}
and so
\begin{align*}
\left(\int_{|v|\geq R} f(v) \dd v\right)^k
\lesssim \max\left\{ K_R^{\frac{mk}{2m-1}} \rho^{k\frac{m(2-k)-1}{2m-1}}, K_R\right\}.
\end{align*}

The result follows from \eqref{eq:masscomp} with $q=\frac{2m-1}{mk}$ and $1-\frac{1}{q} =\frac{m(2-k)-1}{2m-1} $.
\end{proof}

\begin{remark}\label{rem:Hconv2}
We did not extract all the information we could have from \eqref{eq:HH}:
Using the convexity of the function $s\mapsto \frac{s^m}{m-1}$, \eqref{eq:Hconv} also gives
\begin{align*}
\E[f|G[\rho]] & \geq \int_{|v|\leq R }\frac{f^m}{m-1}- \frac{G[\rho]^m}{m-1}-\frac{m}{m-1} G[\rho]^{m-1}(f-G[\rho]) \dd v\nonumber \\
&= \int_{|v|\leq R }\frac{f^m}{m-1} +G[\rho]^m -\frac{m}{m-1} G[\rho]^{m-1}f \dd v
\end{align*}
so that \eqref{eq:HH} implies
$$\int_{|v|\leq R }\frac{f^m}{m-1} +G[\rho]^m -\frac{m}{m-1} G[\rho]^{m-1}f \dd v
\leq H[f]-H[G[\rho]]$$
Using Young's inequality $\int_{|v|\leq R }\frac{m}{m-1} G[\rho]^{m-1}f \dd v\leq 2^{\frac{1}{m-1}}\int_{|v|\leq R } G[\rho]^m \dd v + \frac 1 2 \int_{|v|\leq R }\frac{f^m}{m-1}\dd v$, we deduce
\begin{align}
\frac 1 2 \int_{|v|\leq R }\frac{f^m}{m-1}\dd v
& \leq H[f]-H[G[\rho]] + (2^{\frac{1}{m-1}}-1) \int_{|v|\leq R } G[\rho]^m\dd v\nonumber \\
& \leq H[f]-H[G[\rho]] + \nu_1 (2^{\frac{1}{m-1}}-1) \rho^k
\label{eq:Hconv2}
\end{align}.
\end{remark}
Finally, we can state the main inequality that we use when $m>1$:
\begin{proposition}\label{prop:f-g}
Let $m>1$ and assume that $f$ and $G[\rho]$ have the same mass $\rho$. There exists a constant $C_{d,m}$ such that
\begin{align}
\int_{\R^d} |f(v)-G[\rho](v)|\dd v
& \leq 
C_{d,m}\, \rho^{\frac{2-k}{2}} \left( \left(H[ f]-H[G[\rho]]\right) + \rho^{k-1} M(R) \right)^{\frac 1 2}
\label{eq:L1conv}
\end{align}
where (see Lemma \ref{lem:support})
$$ M(R): = \int_{|v|\geq R} f(v)\, \dd v \leq
\max\left\{ (H[f]-H[G[\rho]])^{\frac 1 {kq}} \rho^{\left(1-\frac{1}{q}\right)} , (H[f]-H[G[\rho]])^{\frac 1 k}\right\}.
$$
In particular, for all $\delta>0$ there exists $C_{\delta,m,d}$ such that
$$\left(\int_{\R^d} |f(v)-G[\rho](v)|\dd v \right)^k\leq \delta \rho^k + C_{\delta,m,d} (H[f]-H[G[\rho]]).$$
\end{proposition}

\begin{proof}[{\bf Proof of \Cref{prop:f-g}}]
We recall that, under the condition $\int_{\R^d} f(v) \dd v=\int_{\R^d} G[\rho](v) \dd v= \rho$, $f(v)$ and $G[\rho](v)$ are such that
$R^2=\frac{2 m}{m-1}\mu(\rho) = \mu_1 \frac{2 m}{m-1}\rho^{k-1} $.
We consider two cases.

\medskip\noindent{\bf Case 1: When $\int_{|v|\geq R} f\dd v \geq \frac 1 2 \int_{\R^d} f\dd v = \frac 1 2 \rho$.}

In that case, Lemma \ref{lem:support} implies
\begin{align*}
\rho^k \leq 2^k \left(\int_{|v|\geq R} f\dd v\right)^k
& \leq 2^k
\left( C (H[f]-H[G[\rho]])^{\frac 1 q} \rho^{k\left(1-\frac{1}{q}\right)} + C (H[f]-H[G[\rho]])\right)\\
& \leq C (H[f]-H[G[\rho]]) + \frac 1 2 \rho^{k}
\end{align*}
(using Young's inequality).
We deduce
$$
\rho^k\leq C(H[f]-H[G[\rho]])
$$
hence
$$ \int_{\R^d} |f(v)-G[\rho](v)|\dd v \leq 2 \int_{\R^d} f(v)+G[\rho](v)\dd v =2\rho\leq C (H[f]-H[G[\rho]])^{\frac 1 k}.
$$
Writing $\rho = \left(\rho^{2-k} \rho^{k}\right)^{\frac 1 2}$, we can also write
\begin{equation*}
\int_{\R^d} |f(v)-G[\rho](v)|\dd v \leq C\left(\rho^{2-k}( H[f]-H[G[\rho]])\right)^{\frac 1 2}
\end{equation*}
thus showing that \eqref{eq:L1conv} holds in that case.

\medskip\noindent{\bf Case 2: When $\int_{|v|\geq R} f\dd v \leq \frac 1 2 \int_{\R^d} f\dd v = \frac 1 2 \rho$.}

In that case, we define
$$
\widetilde f(v) =
\begin{cases}
\alpha f(v) & \mbox{ if } |v|\leq R\\
0 & \mbox{ otherwise,}
\end{cases}
$$
with $\alpha$ chosen such that
$$ \int _{\R^d} \widetilde f(v)\dd v =\int_{\R^d} f(v) \dd v=\rho, $$
that is
$$ \rho = \alpha(\rho-h), \qquad \alpha-1 = \frac{h}{g-h} $$
The condition $\int_{|v|\geq R} f\dd v \leq \frac 1 2 \rho$ implies
$$
\alpha = \frac{\rho}{\int_{|v|\leq R} f\dd v} \in [1,2]$$
and
$$
0\leq \alpha-1 \leq \frac 2 \rho \int_{|v|\geq R} f\dd v.
$$
We now compute:
$$
\int_{\R^d} |f(v) - \widetilde f(v)|\dd v
= (\alpha-1)\int_{|v|\leq R} f(v) \dd v+\int_{|v|>R} f(v) \dd v = 2\int_{|v|\geq R} f\dd v
$$

Furthermore, since $\widetilde f(v)$ and $G[\rho]$ have the same mass $\rho$ and same support $B_R$, we can use Proposition \ref{prop:entropm>1} to write
$$ \int_{\R^d} |\widetilde f(v)-G[\rho](v)|\dd v = \int_{B_R} |\widetilde f(v)-G[\rho](v)|\dd v \leq C_{d,m}\, \rho^{\frac{2-k}{2}}\, ( H[\widetilde f]-H[G[\rho]] )^{\frac 1 2} .$$

Finally, using the fact that $\alpha\in [0,2]$, we can write:
\begin{align*}
H[\widetilde f]-H[G[\rho]] & = H[ f]-H[G[\rho]]+ H[\widetilde f]-H[f]
\\
& \leq H[ f]-H[G[\rho]] + \int_{|v|\leq R} (\alpha^m-1) \frac{f^m}{m-1} + (\alpha-1)\frac1 2|v|^2 f \dd v \\
& \leq H[ f]-H[G[\rho]] + C (\alpha-1)\int_{|v|\leq R} \frac{f^m}{m-1} \dd v + \frac{R^2}{2} (\alpha-1) \int_{|v|\leq R} f \dd v\\
& \leq H[ f]-H[G[\rho]] + \left( C \frac{1}{\int_{|v|\leq R} f \dd v} \int_{|v|\leq R} \frac{f^m}{m-1} \dd v + \frac{R^2}{2} \right) \int_{|v|\geq R} f \dd v
\end{align*}
where we used the fact that $(\alpha-1)\int_{|v|\leq R} f \dd v = \int_{|v|\geq R} f \dd v$.

Using Remark \ref{rem:Hconv2} (see \eqref{eq:Hconv2}) and fact that $R^2 =\frac{2 m}{m-1}\mu(\rho) = \mu_1 \frac{2 m}{m-1}\rho^{k-1} $, we can write:
\begin{align*}
\left( C \frac{1}{\int_{|v|\leq R} f \dd v} \int_{|v|\leq R} \frac{f^m}{m-1} \dd v + \frac{R^2}{2} \right)
& \leq
\frac{C}{\int_{|v|\leq R} f \dd v} (H[f]-H[G[\rho]] + C \rho^k)
+ C\rho^{k-1}\\
& \leq
\frac{C}{\rho} (H[f]-H[G[\rho]] + C\rho^{k-1}
\end{align*}
where we used (again) the fact that $\int_{|v|\leq R} f \dd v\geq \frac 1 2 \rho$.
Putting things together, we get:
\begin{align*}
H[\widetilde f]-H[G[\rho]] &
\leq C \left(H[ f]-H[G[\rho]]\right) + C \rho^{k-1}\int_{|v|\geq R} f \dd v.
\end{align*}
We deduce:
\begin{align*}
\int_{\R^d} |f(v)-G[\rho](v)|\dd v
& \leq \int_{\R^d} |f(v)-\widetilde f (v)|\dd v +
\int_{\R^d} | \widetilde f (v)-G[\rho](v)|\dd v\\
& \leq 2\int_{|v|\geq R} f\dd v +C_{d,m}\, \rho^{\frac{2-k}{2}}\, \left( \left(H[ f]-H[G[\rho]]\right) + \rho^{k-1}\int_{|v|\geq R} f \dd v \right)^{\frac 1 2}\\
& \leq 2\int_{|v|\geq R} f\dd v +C_{d,m}\, \left( \rho^{2-k}\left(H[ f]-H[G[\rho]]\right) + \rho\int_{|v|\geq R} f \dd v \right)^{\frac 1 2}.
\end{align*}
Finally, since $\int_{|v|\geq R} f\dd v\leq \rho$, we have $\int_{|v|\geq R} f\dd v\leq \left(\rho \int_{|v|\geq R} f\dd v\right)^{\frac{1}{2}}$ so we can simply write:
\begin{equation*}
\int_{\R^d} |f(v)-G[\rho](v)|\dd v
\leq C_{d,m}\, \left( \rho^{2-k}\left(H[ f]-H[G[\rho]]\right) + \rho\int_{|v|\geq R} f \dd v \right)^{\frac 1 2}.
\end{equation*}
which gives \eqref{eq:L1conv} in this second case.
\end{proof}

\section{A generalised Carlson-Levin inequality}

First we recall the Carlson-Levin inequality:
\begin{proposition}\label{prop:CL0}
Given $d\geq 2$ and for $m \in ( \frac{d}{d+2},1)$
there holds:
\begin{equation}\label{eq:CL0}
\int_{\R^d} f(v)^m \, \dd v \lesssim \left( \int_{\R^d} f(v) \, \dd v\right)^{\frac{2m-d(1-m)}{2}} \left( \int_{\R^d} |v|^2 f(v) \, \dd v\right)^{\frac{d(1-m)}{2}}.
\end{equation}
\end{proposition}

Next, we prove the following generalization:
\begin{proposition}\label{prop:CL}
Given $d\geq 2$ and $a\in(0,2]$, and for $m$ satisfying
\begin{equation}\label{eq:ma}
\frac{d+a-2}{d+a} <m<\frac{d+a-2}{d},
\end{equation}
there holds:
\begin{equation}\label{eq:CL}
\int_{\R^d} |v|^{a-2} f(v)^m \, \dd v \lesssim \left( \int_{\R^d} f(v) \, \dd v\right)^{\frac{a(m-1)+2-d(1-m)}{a}} \left( \int_{\R^d} |v|^a f(v) \, \dd v\right)^{\frac{a-2+d(1-m)}{a}}.
\end{equation}
\end{proposition}

Note that for every $m\in (\frac{d-2}{2},1)$, there exists $a \in (0,2)$ such that \eqref{eq:ma} holds (since $\frac{d+a-2}{d+a} \to \frac{d-2}{2}$ as $a\to 0$ and $\frac{d+a-2}{d}\to1$ as $a\to 2$).

\begin{proof}[{\bf Proof of \Cref{prop:CL}}]
We write
$$
\int_{\R^d} |v|^{a-2} f(v)^m \dd v= \int_{|v|\leq R} |v|^{a-2} f(v)^m \dd v+\int_{|v|\geq R} |v|^{a-2} f(v)^m \dd v
$$
Using H\"older's inequality, we write,
\begin{multline*}
\int_{|v|\leq R} |v|^{a-2} f(v)^m \dd v \leq \left( \int_{|v|\leq R} f(v) \dd v\right)^{m} \left( \int_{|v|\leq R} |v|^{\frac{a-2}{1-m}} \dd v\right)^{1-m}
\\\leq \left( \int_{|v|\leq R} f(v) \dd v\right)^{m}
R^{a-2+d(1-m)}. 
\end{multline*}
This inequality requires $\frac{a-2}{1-m}>-d$ which is equivalent to the right inequality in \eqref{eq:ma}.
Similarly, we have
\begin{multline*}
\int_{|v|\geq R} |v|^{a-2} f(v)^m \dd v \leq \left( \int_{|v|\geq R} |v|^a f(v) \dd v\right)^{m} \left( \int_{|v|\geq R} |v|^{a-\frac{2}{1-m}} \dd v\right)^{1-m}
\\\leq \left( \int_{|v|\geq R} |v|^a f(v) \dd v\right)^{m}
R^{a(1-m)-2+d(1-m)}. 
\end{multline*}
This inequality requires $a-\frac{2}{1-m} <-d$, which is equivalent to the left inequality in \eqref{eq:ma}.
Putting things together, we deduce:
$$
\int_{\R^d} |v|^{a-2} f(v)^m \dd v
\leq
\left( \int_{\R^d} f(v) \dd v\right)^{m}
R^{a-2+d(1-m)} +
\left( \int_{\R^d} |v|^a f(v) \dd v\right)^{m}
R^{a(1-m)-2+d(1-m)}.
$$
We note that $a-2+d(1-m)>0$ and $a(1-m)-2+d(1-m)<0$ so we can optimize with respect to $R$ to get \eqref{eq:CL}.
\end{proof}

\section{Relative entropy method when \texorpdfstring{$m=1$}{=1}: Proof of \texorpdfstring{\Cref{thm:convm1}}{thm:convm1}}

In this appendix, we present the relative entropy method for the linear Fokker-Planck equation. Consider the following (linear) rescaled equation,
\begin{equation}\label{eq:VFP}
\eps^2 \pa_t f_\eps + \eps v\cdot\na _x f_\eps = \text{div}_v \left[\na_v f_\eps + vf_\eps\right], \qquad t > 0, \; x \in \Omega, \; v \in \R^d.
\end{equation}
complemented with an initial condition $f_\eps(0,\cdot)$. This equation is associated to the entropy,
$$
\forall t >0, \qquad \mathscr E(f_\eps)(t) = \iint_{\Omega \times\R^d} f_\eps \log f_\eps + \frac{|v|^2}{2} f_\eps\, \dd x \, \dd v.
$$
In particular,
solutions of \eqref{eq:VFP} satisfy
$$
\frac{d}{dt} \mathscr E(f_\eps)(t) = -\frac 1 {\eps^{2}}
\iint_{\Omega \times\R^d} \frac{1}{ f_\eps} \left\vert \na_v f_\eps +vf_\eps \right\vert^2 \, \dd x \, \dd v
$$
Classically, the minimizers of $\mathscr E(f)$ under mass constraint $\irdv{f} =\rho$ are given by Maxwellian distributions, that is
$$\rho \, M(v) := \frac{\rho}{(2\pi)^{\frac{d}{2}}} \exp\left(- \frac{\vert v \vert^2 }{2} \right).$$
In particular, $f_\eps$ is expected to converge to $\rho_0 \, M$ where $(t,x) \mapsto \rho_0(t,x)$ solves
$\pa_t \rho_0 - \Delta \rho_0=0$, or equivalently,
\begin{equation}\label{eq:diff}
\pa_t \rho_0 + \text{div}_x (\rho_0 w) = 0 , \qquad w = - \nabla \ln \rho_0.
\end{equation}

Given $f_\eps$ a solution to \eqref{eq:VFP} and $\rho_0 $ a solution to \eqref{eq:diff}, we define
the modulated energy (or relative entropy):
\begin{align*}
\mathscr H_\eps (t)
& = \iint_{\Omega \times\R^d} f_\eps(t,x,v) \log \left( \frac { f_\eps(t,x,v)}{ \rho_0(t,x) e^{- \frac{1}{2} |v-\eps w_0(t,x)|^2}} \right)\, \dd x \, \dd v\\
& = \iint_{\Omega \times\R^d} f_\eps \log\left(\frac{f_\eps}{\rho_0}\right) +\frac{|v-\eps w_0|^2}{2} f_\eps \, \dd x \, \dd v.
\end{align*}
Observe that this quantity is not the relative entropy with respect to the expected limit $\rho_0 \, M$, but takes into account a first order correction of the flux.

\begin{theorem}\label{thm:convm1}
Let $f_\eps$ be a weak solution to \eqref{eq:VFP} and $\rho_0$ be a solution to \eqref{eq:diff}. Then we have
\begin{align*}
&\frac{d}{dt} \mathscr H_\eps (t) +\frac{1}{2\eps^2}\iint_{\Omega \times\R^d} \frac{1}{ f_\eps} |\na_v f_\eps +(v- \eps u_\eps) f_\eps |^2 \, \dd x \, \dd v + \frac 1 2 \int_\Omega \rho_\eps |u_\eps-w|^2 \, \dd x \\
& \leq
\eps^2 C \Big ( \|\pa_tw\|_{L^\infty},\|Dw\|_{L^\infty} , \iint_{\Omega \times\R^d} |v|^2 f_\eps \, \dd v \, \dd x \Big )
\end{align*}
\end{theorem}

Define $g_\eps = \frac{f_\eps}{\rho_0}$. The function $g_\eps$ satisfies,
\begin{equation}
\eps^2 \pa_t g_\eps + \eps^2 \left( \rho_0^{-1}\pa_t \rho_0 \right) g_\eps + \eps v\cdot\na _x (\rho_0 g_\eps) \rho_0^{-1} = \text{div}_v \left[\na_v g_\eps + v g_\eps\right], \qquad t > 0, \; x \in \Omega, \; v \in \R^d.
\end{equation}
The relative entropy may be rewritten,
\begin{align*}
\mathscr H_\eps (t)
& = \iint_{\Omega \times\R^d} \left[ g_\eps \log g_\eps +\frac{1}{2} |v-\eps w_0|^2 g_\eps \right]\, \rho_0 \, \dd x \, \dd v.
\end{align*}
Compute,
\begin{align*}
\mathscr H_\eps' (t)
& = \iint_{\Omega \times\R^d} \left( 1 + \log g_\eps +\frac{1}{2} |v-\eps w_0|^2 \right)\partial_t g_\eps \, \rho_0 \, \dd x \, \dd v \\
& \qquad \qquad + \iint_{\Omega \times\R^d} (v-\eps w_0)(-\eps \partial_t w_0 ) g_\eps \, \rho_0 \, \dd x \, \dd v \\
& \qquad \qquad\qquad \qquad + \iint_{\Omega \times\R^d} \left[ g_\eps \log g_\eps +\frac{1}{2} |v-\eps w_0|^2 g_\eps \right]\, \partial_t \rho_0 \, \dd x \, \dd v \\
& = I_1 + I_2 + I_3.
\end{align*}
We shall estimate the three integrals separately.
\begin{align*}
\eps^2 I_1 :=
&\eps^2 \iint_{\Omega \times\R^d} \left( 1 + \log g_\eps +\frac{1}{2} |v-\eps w_0|^2 \right)\partial_t g_\eps \, \rho_0 \, \dd x \, \dd v \\
&= \iint_{\Omega \times\R^d} \left( 1 + \log g_\eps +\frac{1}{2} |v-\eps w_0|^2 \right) \left( \text{div}_v \left[\na_v g_\eps + v g_\eps\right] \right) \, \rho_0 \, \dd x \, \dd v \\
& \qquad \qquad + \iint_{\Omega \times\R^d} \left( 1 + \log g_\eps +\frac{1}{2} |v-\eps w_0|^2 \right) \left( - \eps v\cdot\na _x (\rho_0 g_\eps) \rho_0^{-1} \right) \, \rho_0 \, \dd x \, \dd v \\
&\qquad \qquad \qquad \qquad + \iint_{\Omega \times\R^d} \left( 1 + \log g_\eps +\frac{1}{2} |v-\eps w_0|^2 \right) \left( -\eps^2 \left( \rho_0^{-1}\pa_t \rho_0 \right) g_\eps \right) \, \rho_0 \, \dd x \, \dd v \\
& = I_{1,1} + I_{1,2} + I_{1,3}.
\end{align*}
Again, we estimate all integrals separately.
\begin{align*}
I_{1,1} & = \iint_{\Omega \times\R^d} \left( 1 + \log g_\eps +\frac{1}{2} |v-\eps w_0|^2 \right) \left( \text{div}_v \left[\na_v g_\eps + v g_\eps\right] \right) \, \rho_0 \, \dd x \, \dd v \\
& = - \iint_{\Omega \times\R^d} \left( \frac{1}{g_\eps}\nabla_v g_\eps + (v-\eps w_0) \right) \cdot \left(\na_v g_\eps + v g_\eps\right) \, \rho_0 \, \dd x \, \dd v \\
&= - \iint_{\Omega \times\R^d} \frac{1}{g_\eps} \left\vert \na_v g_\eps + (v-\eps w_0) g_\eps\right\vert^2 \, \rho_0 \, \dd x \, \dd v - \iint_{\Omega \times\R^d} \eps \left( \nabla_v g_\eps + (v-\eps w_0) g_\eps \right) \cdot w \, \rho_0 \, \dd x \, \dd v \\
&= - \iint_{\Omega \times\R^d} \frac{1}{g_\eps} \left\vert \na_v g_\eps + (v-\eps w_0) g_\eps\right\vert^2 \, \rho_0 \, \dd x \, \dd v - \iint_{\Omega \times\R^d} \eps (v-\eps w_0) \cdot w_0 g_\eps \, \rho_0 \, \dd x \, \dd v
\end{align*}
Then,
\begin{align*}
I_{1,2} &= \iint_{\Omega \times\R^d} \left( 1 + \log g_\eps +\frac{1}{2} |v-\eps w_0|^2 \right) \left( - \eps v\cdot\na _x (\rho_0 g_\eps) \rho_0^{-1} \right) \, \rho_0 \, \dd x \, \dd v\\
&= \eps \iint_{\Omega \times\R^d} v\cdot\left( \frac{\na _x g_\eps}{g_\eps}- \eps (D_x w_0)(v-\eps w_0) \right) \rho_0 g_\eps \, \dd x \, \dd v \\
&= \eps \iint_{\Omega \times\R^d} v\cdot \na _x g_\eps \rho_0 \, \dd x \, \dd v - \eps^2\iint_{\Omega \times\R^d} v^\top (D_x w_0)(v-\eps w_0) \rho_0 g_\eps \, \dd x \, \dd v\\
&= - \eps \iint_{\Omega \times\R^d} v\cdot \na _x \rho_0 g_\eps \, \dd x \, \dd v - \eps^2\iint_{\Omega \times\R^d} v^\top (D_x w_0)(v-\eps w_0) \rho_0 g_\eps \, \dd x \, \dd v\\
&= \eps \iint_{\Omega \times\R^d} v\cdot w_0 \rho_0 g_\eps \, \dd x \, \dd v - \eps^2\iint_{\Omega \times\R^d} v^\top (D_x w_0)(v-\eps w_0) \rho_0 g_\eps \, \dd x \, \dd v
\end{align*}
Finally,
\begin{align*}
I_{1,3} &= \iint_{\Omega \times\R^d} \left( 1 + \log g_\eps +\frac{1}{2} |v-\eps w_0|^2 \right) \left( -\eps^2 \left( \rho_0^{-1}\pa_t \rho_0 \right) g_\eps \right) \, \rho_0 \, \dd x \, \dd v\\
&= -\eps^2\iint_{\Omega \times\R^d} \left( 1 + \log g_\eps +\frac{1}{2} |v-\eps w_0|^2 \right) \left( \pa_t \rho_0 \right) g_\eps \, \dd x \, \dd v.
\end{align*}
Observe that $I_{1,3} + \eps^2 I_3 = - \eps^2 \iint_{\Omega \times\R^d} g_\eps \pa_t \rho_0 \, \dd x \, \dd v.$
As a consequence, using a cancellation between $I_{1,1}$ and $I_{1,2}$,
\begin{align*}
\eps^2 \mathscr H_\eps' (t)
& = \eps^2 I_1 + \eps^2 I_2 + \eps^2 I_3 = I_{1,1} + I_{1,2} + I_{1,3} + \eps^2 I_2 + \eps^2 I_3\\
&= - \iint_{\Omega \times\R^d} \frac{1}{g_\eps} \left\vert \na_v g_\eps + (v-\eps w_0) g_\eps\right\vert^2 \, \rho_0 \, \dd x \, \dd v \\
&-\eps^2 \iint_{\Omega \times\R^d} v^\top (D_x w_0)((v-\eps w_0) g_\eps+ \nabla_v g_\eps) \rho_0 \, \dd x \, \dd v \\
&+ \eps^2\iint_{\Omega \times\R^d} v^\top (D_x w_0)(\nabla_v g_\eps) \rho_0 \, \dd x \, \dd v + \iint_{\Omega \times\R^d} \eps^2 \vert w_0 \vert^2 g_\eps \, \rho_0 \, \dd x \, \dd v - \eps^2\iint_{\Omega \times\R^d} g_\eps \pa_t \rho_0 \, \dd x \, \dd v \\
&+ \eps^2 \iint_{\Omega \times\R^d} (v-\eps w_0)(-\eps \partial_t w_0 ) g_\eps \, \rho_0 \, \dd x \, \dd v 
\end{align*}
From this, observe that, since
\begin{equation*}
\iint_{\Omega \times\R^d} v^\top (D_x w_0)(\nabla_v g_\eps) \rho_0 \, \dd x \, \dd v = - \iint_{\Omega \times\R^d} \text{div}_x w \, g_\eps \, \rho_0 \, \dd x \, \dd v ,
\end{equation*}
we have
\begin{align*}
& \iint_{\Omega \times\R^d} v^\top (D_x w_0)(\nabla_v g_\eps) \rho_0 \, \dd x \, \dd v + \iint_{\Omega \times\R^d} \vert w \vert^2 g_\eps \, \rho_0 \, \dd x \, \dd v - \iint_{\Omega \times\R^d} g_\eps \pa_t \rho_0 \, \dd x \, \dd v \\
&= \iint_{\Omega \times\R^d} \left( -\text{div}_x w \, \rho_0 + \vert w_0 \vert^2 \, \rho_0 - \pa_t \rho_0 \right) \, g_\eps \, \dd x \, \dd v = 0,
\end{align*}
recalling the identity $ -\text{div}_x w \, \rho_0 + \vert w_0 \vert^2 \, \rho_0 = \pa_t \rho_0$.
The last step is to bound,
\begin{align*}
&\left\vert \iint_{\Omega \times\R^d} v^\top (D_x w_0)((v-\eps w_0) g_\eps+ \nabla_v g_\eps) \rho_0 \, \dd x \, \dd v\right\vert \\ &\leq \frac{1}{4\eps^{2}} \iint_{\Omega \times\R^d} \frac{1}{g_\eps} \left\vert \na_v g_\eps + (v-\eps w_0) g_\eps\right\vert^2 \, \rho_0 \, \dd x \, \dd v + \eps^2 \Vert D_x w_0 \Vert_{\infty} \iint_{\Omega \times\R^d} \vert v \vert^2 g_\eps \rho_0 \, \dd x \, \dd v.
\end{align*}
and,
\begin{align*}
& \left\vert \iint_{\Omega \times\R^d} (v-\eps w_0) \partial_t w_0 \, g_\eps \, \rho_0 \, \dd x \, \dd v \right\vert = \left\vert \iint_{\Omega \times\R^d} \left( (v-\eps w_0)g_\eps + \nabla_v g_\eps\right) \, \partial_t w_0 \, \rho_0 \, \dd x \, \dd v \right\vert \\
&\leq \frac{1}{4\eps^{3}} \iint_{\Omega \times\R^d} \frac{1}{g_\eps} \left\vert \na_v g_\eps + (v-\eps w_0) g_\eps\right\vert^2 \, \rho_0 \, \dd x \, \dd v + \eps^3 \Vert \partial_t w_0 \Vert_{\infty} \iint_{\Omega \times\R^d} g_\eps \rho_0 \, \dd x \, \dd v. 
\end{align*}
to conclude that,
\begin{align*}
\eps^2 \mathscr H_\eps' (t)
&\leq - \frac12\iint_{\Omega \times\R^d} \frac{1}{g_\eps} \left\vert \na_v g_\eps + (v-\eps w_0) g_\eps\right\vert^2 \, \rho_0 \, \dd x \, \dd v \\
&+ \eps^4 \Vert D_x w_0 \Vert_{\infty} \iint_{\Omega \times\R^d} \vert v \vert^2 g_\eps \rho_0 \, \dd x \, \dd v + \eps^6 \Vert \partial_t w_0 \Vert_{\infty} \iint_{\Omega \times\R^d} g_\eps \rho_0 \, \dd x \, \dd v,
\end{align*}
or equivalently,
\begin{align*}
\mathscr H_\eps' (t) &+ \frac{1}{2\eps^2}\iint_{\Omega \times\R^d} \frac{1}{f_\eps} \left\vert \na_v f_\eps + (v-\eps w_0) f_\eps\right\vert^2 \, \dd x \, \dd v \\
&\leq \eps^2 \Vert D_x w_0 \Vert_{\infty} \iint_{\Omega \times\R^d} \vert v \vert^2 f_\eps \, \dd x \, \dd v + \eps^4 \Vert \partial_t w_0 \Vert_{\infty} \iint_{\Omega \times\R^d} f_\eps \, \dd x \, \dd v.
\end{align*}
This ends the proof.\newpage


\bigskip\begin{center}\rule{2cm}{0.5pt}\end{center}\bigskip

\end{document}